\documentclass[twoside]{amsart}
\usepackage[utf8]{inputenc}
\usepackage{mathrsfs}
\usepackage{amssymb}
\usepackage{bm}
\usepackage{tikz-cd}
\usetikzlibrary{calc}
\usepackage{calc}
\usepackage{cite}
\usepackage{url}
\usepackage[symbol]{footmisc}
\usepackage{etoolbox}
\usepackage{hyperref}

\newtheorem{thm}{Theorem}[section]
\newtheorem{lemma}[thm]{Lemma}
\newtheorem{cor}[thm]{Corollary}
\newtheorem{prop}[thm]{Proposition}
\newtheorem*{thm*}{Theorem}
\newtheorem*{prop*}{Proposition}

\theoremstyle{remark}
\newtheorem{rk}[thm]{Remark}
\newtheorem{constr}[thm]{Construction}
\newtheorem{warn}[thm]{Warning}
\newtheorem{ex}[thm]{Example}

\newtheorem*{claim*}{Claim}
\AtBeginEnvironment{claim*}{\let\oldqed=\qedsymbol%
\renewcommand\qedsymbol{$\triangle$}}
\AtEndEnvironment{claim*}{\let\qedsymbol=\oldqed}

\theoremstyle{definition}
\newtheorem{defi}[thm]{Definition}

\numberwithin{equation}{section}

\newcommand{\Ho}{\textup{Ho}}

\newcommand{\cat}[1]{\textbf{\textup{#1}}}

\newcommand{\blank}{{\textup{--}}}
\newcommand{\pr}{{\textup{pr}}}
\newcommand{\Hom}{{\textup{Hom}}}
\newcommand{\id}{{\textup{id}}}
\newcommand{\ev}{{\textup{ev}}}
\newcommand{\im}{\mathop{\textup{im}}}

\newcommand{\incl}{{\textup{incl}}}

\newcommand{\forget}{\mathop{\textup{forget}}\nolimits}
\newcommand{\Fun}{\textup{Fun}}

\newcommand{\Inj}{\textup{Inj}}
\newcommand{\supp}{\mathop{\textup{supp}}\nolimits}

\newcommand{\Ob}{\mathop{\textup{Ob}}}
\newcommand{\zigzag}{\mathord{\text{\raisebox{.8pt}{$\scriptstyle\bullet$}\kern-4pt$\to$\kern-1pt\raisebox{.8pt}{$\scriptstyle\bullet$}\kern-1pt$\gets$\kern-4pt\raisebox{.8pt}{$\scriptstyle\bullet$}}}}

\newcommand{\hq}{\mathord{\textup{/\kern-2.5pt/}}}

\newcommand{\ppo}{\mathbin{\raise.25pt\hbox{\scalebox{.67}{$\square$}}}}

\newcommand{\CMon}{{\textup{CMon}}}
\newcommand{\Cyl}{{\textup{Cyl}}}

\let\phi=\varphi
\let\setminus=\smallsetminus

\title[Parsummable categories vs.~symmetric monoidal Categories]{Parsummable categories as a strictification of symmetric monoidal categories}
\author{Tobias Lenz}
\address{Mathematisches Institut, Rheinische Friedrich-Wilhelms-Universit\"at Bonn, Endenicher Allee 60, 53115 Bonn, Germany \& Max-Planck-Institut für Mathematik, Vivatsgasse 7, 53111 Bonn, Germany}
\email{lenz@math.uni-bonn.de}

\begin{document}
\maketitle

\begin{abstract}
We prove that the homotopy theory of \emph{parsummable categories} (as defined by Schwede) with respect to the underlying equivalences of categories is equivalent to the usual homotopy theory of symmetric monoidal categories.
In particular, this yields a model of symmetric monoidal categories in terms of categories equipped with a strictly commutative, associative, and unital (but only partially defined) operation.
\end{abstract}

\section*{Introduction}
The study of \emph{monoidal categories} and \emph{symmetric monoidal categories} was pioneered by Bénabou \cite{cat-avec-mult} and Mac Lane \cite{nat-assoc-com}. While it is necessary in order to cover examples arising in nature that unitality, associativity (and commutativity where applicable) are only required up to specified and coherent isomorphism, keeping track of all these isomorphisms in calculations can be quite a hassle, and in practice one often suppresses the associativity and unitality isomorphisms. This is (at least partially) justified by Mac Lane's Coherence Theorems \cite[Theorem~3.1 and Theorem~4.2]{nat-assoc-com}, roughly saying that any `formal' diagram in any (symmetric) monoidal category only involving associativity, unitality (and symmetry) isomorphisms commutes.

A different interpretation of the Coherence Theorem for monoidal categories is Mac Lane's result \cite[Theorem~XI.3.1]{cat-working} that any monoidal category is strongly monoidally equivalent to a \emph{strict} monoidal category, i.e.~one in which the associativity and unitality isomorphisms are the respective identity transformations---needless to say, working with a strict monoidal category is often significantly easier than working with a general monoidal one.

It is however \emph{not} true that any symmetric monoidal category $\mathscr C$ can be replaced by one that has all associativity, unitality, and symmetry isomorphisms given by the respective identities. Instead, the appropriate strict notion turns out to be that of a \emph{permutative category}, i.e.~one in which associativity and unitality are strict, but the symmetry isomorphism is allowed to be non-trivial.

The goal of the present article is to prove that there is in fact still a model for symmetric monoidal categories in terms of categories with a strictly unital, associative \emph{and commutative} operation $+$. However, as the na\"ive strictification result for symmetric monoidal categories is not true, we of course have to pay a price for this: the toll of the present approach is that not all pairs of objects can be summed anymore, i.e.~that $+$ is only defined on a certain (full and essentially wide) subcategory of the cartesian product.

More precisely, we will consider \emph{parsummable categories} as defined by Schwede \cite{schwede-k-theory} as input for his construction of global algebraic $K$-theory. There is an explicit functor $\Phi\colon\cat{PermCat}\to\cat{ParSumCat}$ from permutative categories to parsummable categories, that is used to define the global algebraic $K$-theory of permutative and more generally symmetric monoidal categories. Our main result can then be stated as follows:

\begin{thm*}[see Theorem~\ref{thm:main-thm}]
The functor $\Phi$ is a homotopy equivalence with respect to the underlying equivalences of categories on both sides.
\end{thm*}

While this result is in line with previous work on strict models for symmetric monoidal categories (that we will recall below), our motivation mostly comes from global algebraic $K$-theory. Namely, while the underlying equivalences of categories in $\cat{ParSumCat}$ are usually not preserved by the global algebraic $K$-theory functor $\textbf{K}_{\textup{gl}}$, it does indeed send underlying equivalences between so-called \emph{saturated} parsummable categories to weak equivalences of global spectra. Using this, the above theorem is then one step (of many) in our proof that global algebraic $K$-theory expresses connective global stable homotopy theory as $\infty$-categorical localizations of both $\cat{PermCat}$ and $\cat{ParSumCat}$ \cite{g-global,sym-mon-global}, generalizing a classical non-equivariant result due to Thomason \cite{thomason}.

\subsection*{Related work}
Ours is by far not the first strict model for `coherently commutative' monoids, and while the proof of the above theorem will be almost entirely self-contained, the underlying construction is closely related to previous work.

In the topological setting (or, more precisely, in the context of the category $\cat{SSet}$ of simplicial sets), Sagave and Schlichtkrull \cite{sagave-schlichtkrull} have shown that commutative monoids in $\cat{$\bm I$-SSet}\mathrel{:=}\Fun(I,\cat{SSet})$ with respect to the \emph{box product} $\boxtimes$ model $E_\infty$-algebras in $\cat{SSet}$. Here $I$ is the category of finite sets and injections.

If we write $\mathcal M$ for the monoid of injective self-maps of $\omega=\{1,2,\dots\}$, then there is a natural functor $\ev_\omega\colon\cat{$\bm I$-SSet}\to\cat{$\bm{\mathcal M}$-SSet}$ to simplicial sets with an $\mathcal M$-action, given by `evaluating at infinity.' Sagave and Schwede show in \cite{I-vs-M-1-cat} that this is a homotopy equivalence (with respect to suitable notions of weak equivalence) when one restricts the target to the full subcategory $\cat{$\bm{\mathcal M}$-SSet}^\tau\subset\cat{$\bm{\mathcal M}$-SSet}$ spanned by those $\mathcal M$-simplicial sets that satisfy an additional condition called \emph{tameness}. On $\cat{$\bm{\mathcal M}$-SSet}^\tau$ they again construct a symmetric monoidal box product $\boxtimes$, and $\ev_\omega$ is in fact strong symmetric monoidal with respect to this. The proof of \cite[Theorem~5.13]{I-vs-M-1-cat} then shows that $\ev_\omega$ also induces an equivalence of the corresponding homotopy theories of commutative monoids. Parsummable categories, finally, can be viewed as a categorical version of the commutative $\boxtimes$-monoids in $\cat{$\bm{\mathcal M}$-SSet}^\tau$.

In the categorical context, the work of Kodjabachev and Sagave \cite{joyal-strict} shows that for any $E_\infty$-operad $\mathscr O$ in the Joyal model structure on \cat{SSet} (modelling quasi-categories), the inclusion of the strictly commutative monoids in $\cat{$\bm I$-SSet}$ into the $\mathscr O$-algebras is a homotopy equivalence (with respect to the weak equivalences of the Joyal model structure). This ought to give the corresponding result for ordinary categories via appropriate (Bousfield) localization on both sides.

Moreover, Solberg constructed an explicit functor
\begin{equation}\tag{\ensuremath{*}}\label{eq:Phi-solberg}
\Phi\colon\cat{PermCat}\to\cat{$\bm I$-Cat}.
\end{equation}
If we take the above $1$-categorical version of the Kodjabachev-Sagave result for granted, \cite[discussion before Theorem~7.1]{solberg-schlichtkrull} shows that $(\ref{eq:Phi-solberg})$ is a homotopy equivalence. On the other hand, the composition
\begin{equation*}
\cat{PermCat}\xrightarrow{\Phi}\CMon(\cat{$\bm I$-Cat})\xrightarrow{\ev_\omega}\CMon(\cat{$\bm{\mathcal M}$-Cat}^\tau)
\end{equation*}
is closely related to our functor $\Phi$, see \cite[Remark~11.3]{schwede-k-theory}.

Finally, \cite[Remark~4.20]{schwede-k-theory} sketches how parsummable categories can be identified with `tame' algebras over a specific $E_\infty$-operad $\mathcal I$ in $\cat{Cat}$. On the other hand, permutative categories can be identified with algebras over the `categorical Barratt-Eccles operad' $E\Sigma_*$, and the general theory then implies that the homotopy theory of not necessarily tame $\mathcal I$-algebras is equivalent to the one of permutative categories. However, this equivalence is very inexplicit, and it is moreover not clear how we can deduce the corresponding statement for \emph{tame} $\mathcal I$-algebras from this.

\subsection*{Outline}
So while all of the above results give strong evidence for our main theorem, they either work in a different context or there would still be a significant amount of work left to do in order to prove the theorem along the suggested route.

In the present article, we will therefore proceed by different and more elementary means. Namely, we will give an explicit functor $\Sigma\colon\cat{ParSumCat}\to\cat{PermCat}$ and then exhibit concrete homotopies between the two composites and the respective identities. This is organized as follows:

In Section~\ref{sec:background} we recall basic facts about parsummable categories, permutative categories, as well as the construction of the functor $\Phi$.

Section~\ref{sec:sigma-construction} is devoted to the construction of $\Sigma$ and the proof that it is left homotopy inverse to $\Phi$. In Section~\ref{sec:main-thm} we will then complete the proof of the main theorem by showing that $\Sigma$ is also right homotopy inverse to $\Phi$.

Finally, Section~\ref{sec:outlook} is concerned with some aspects of the behaviour of $\Phi$ with respect to a different notion of weak equivalence that appears in the study of global algebraic $K$-theory.

\subsection*{Acknowledgements}
The results in this article were obtained as part of my ongoing PhD thesis at the University of Bonn. I would like to thank my advisor Stefan Schwede for suggesting this project and for several helpful remarks on a previous version of this article. I would moreover like to thank the anonymous referee for a very careful reading and for their helpful questions that lead to a major simplification of the argument in Section~\ref{sec:main-thm}.

I am grateful to the Max Planck Institute for Mathematics in Bonn for their hospitality and support.

The author is an associate member of the Hausdorff Center for Mathematics, funded by the Deutsche Forschungsgemeinschaft (DFG, German Research Foundation) under Germany's Excellence Strategy (GZ 2047/1, project ID390685813).

\section{Background}\label{sec:background}
In this section we recall the basics about parsummable categories on the one hand and about permutative categories on the other hand.

\subsection{Parsummable categories}
For the reader's convenience, we briefly recall the notion of a \emph{parsummable category}; a detailed treatment can be found in \cite[Sections~2~and~4]{schwede-k-theory}, which we follow rather closely.

We write $\omega\mathrel{:=}\{1,2,\dots\}$ and we denote the monoid of injective self maps of $\omega$ by $\mathcal M$. We moreover recall the functor $E\colon\cat{Set}\to\cat{Cat}$ that is right adjoint to $\Ob\colon\cat{Cat}\to\cat{Set}$. Explicitly, $ES$ is the `chaotic category' with set of objects $S$, i.e.~there is a unique morphism $x\to y$ for any $x,y\in S$; we denote this morphism by $(y,x)$.

In particular, we can apply this to $\mathcal M$, yielding a category $E\mathcal M$. There is a unique way to extend the multiplication $\mathcal M\times\mathcal M\to\mathcal M$ to $E\mathcal M$ and this way $E\mathcal M$ becomes a monoid in the $1$-category $\cat{Cat}$, i.e.~a strict monoidal category.

\begin{warn}
In our main reference \cite{schwede-k-theory}, Schwede defines $\omega\mathrel{:=}\{0,1,2,\dots\}$, but the above definition, used for example in \cite{schwede-semistable} and \cite{I-vs-M-1-cat}, will be more convenient here. From a purely formal point of view, the distinction between the two conventions is of course irrelevant.

We also remark that the above monoid $\mathcal M$ is more commonly denoted by $M$ in the literature, whereas \cite{schwede-k-theory} uses $\mathcal M$ for what we call $E\mathcal M$. We think however that $\mathcal M$ is important enough that it deserves to be notationally distinguished from a generic monoid.
\end{warn}

An \emph{$E\mathcal M$-category} is simply a category with a \emph{strict} $E\mathcal M$-action; we denote $E\mathcal M$-categories by the calligraphic letters $\mathcal C,\mathcal D$, and so on. If $\mathcal C$ is any $E\mathcal M$-category, then we write $u_*\colon\mathcal C\to\mathcal C$ for the action of $u\in\mathcal M$, and if $v\in\mathcal M$ is any other injection, then we write $[v,u]\colon u_*\Rightarrow v_*$ for the natural isomorphism given by the action of the morphism $(v,u)$ of $E\mathcal M$. Thus, we in particular get an $\mathcal M$-action on $\Ob(\mathcal C)$ and for any $X\in\mathcal C$ and $u\in\mathcal M$ an isomorphism
\begin{equation*}
u^X_\circ\mathrel{:=}[u,1]_X\colon X\to u_*X
\end{equation*}
(here we follow the usual convention to denote the components of a natural transformation like $[u,1]$ by lower indices, whereas \cite{schwede-k-theory} uses upper indices for the components of the structure isomorphisms).
Functoriality together with the action property then implies the relation
\begin{equation}\label{eq:u-circ-relation}
(uv)^{X}_\circ = u^{v_*X}_\circ v^X_\circ\colon X\to u_*v_*X=(uv)_*X
\end{equation}
as morphisms in $\mathcal C$. The following convenient lemma tells us conversely, that this data is enough to specify an $E\mathcal M$-action:

\begin{lemma}\label{lemma:EM-action-concise}
Let $\mathcal C$ be a category together with an $\mathcal M$-action on $\Ob\mathcal C$ and for any $X\in\mathcal C$ and $u\in\mathcal M$ an isomorphism $u^X_\circ\colon X\to u_*X$ such that the relation $(\ref{eq:u-circ-relation})$ holds for any further injection $v\in\mathcal M$.

Then there exists a unique $E\mathcal M$-action on $\mathcal C$ such that the underlying $\mathcal M$-action on $\Ob(\mathcal C)$ is the given one and such that $[u,1]_X=u^X_\circ$ for all $u\in\mathcal M,X\in\mathcal C$.
\begin{proof}
This is \cite[Proposition~2.6]{schwede-k-theory}. For later reference we briefly recall how one can recover the remaining structure: if $u\in\mathcal M$ is any injection and $f\colon X\to Y$ is a morphism in $\mathcal C$, then $u_*(f)= u^Y_\circ\circ f\circ (u^X_\circ)^{-1}$. Moreover,
\begin{equation*}
[v,u]_X= v^X_\circ\circ (u^X_\circ)^{-1}\colon u_*X\to v_*X
\end{equation*}
for any further injection $v\in\mathcal M$.
\end{proof}
\end{lemma}

A \emph{morphism of $E\mathcal M$-categories} is simply a strictly $E\mathcal M$-equivariant functor, i.e.~a functor $F\colon\mathcal C\to\mathcal D$ such that the diagram
\begin{equation*}
\begin{tikzcd}[column sep=large]
E\mathcal M\times\mathcal C\arrow[d, "\text{act}"']\arrow[r, "E\mathcal M\times F"] & E\mathcal M\times\mathcal D\arrow[d, "\text{act}"]\\
\mathcal C\arrow[r, "F"'] &\mathcal D
\end{tikzcd}
\end{equation*}
commutes \emph{strictly}. We denote the category of small $E\mathcal M$-categories by $\cat{$\bm{E\mathcal M}$-Cat}$.

The following criterion will be used frequently to check that a given functor is $E\mathcal M$-equivariant:

\begin{cor}
Let $\mathcal C,\mathcal D$ be $E\mathcal M$-categories. Then a functor $F\colon\mathcal C\to\mathcal D$ of their underlying categories is $E\mathcal M$-equivariant if and only if $\Ob(F)\colon\Ob(\mathcal C)\to\Ob(\mathcal D)$ is $\mathcal M$-equivariant and the relation
\begin{equation*}
F(u^X_\circ)= u^{F(X)}_\circ\colon F(X)\to u_*(FX)=F(u_*X)
\end{equation*}
holds for all $X\in\mathcal C$ and $u\in\mathcal M$.
\begin{proof}
This is immediate from the description of the functors $u_*$ and the natural isomorphisms $[v,u]$ given in the proof of the previous lemma.
\end{proof}
\end{cor}

\begin{rk}
It is a non-trivial fact (crucially depending on the structure of the monoid $\mathcal M$) that a functor is $E\mathcal M$-equivariant if and only if it is $\mathcal M$-equivariant. However, since we will construct all our $E\mathcal M$-actions via Lemma~\ref{lemma:EM-action-concise}, it will be more natural to check the condition from the corollary than to unravel the definition of $u_*$ on morphisms. Accordingly, we will have no use for this alternative criterion and we leave its proof to the curious reader.
\end{rk}

Now we can introduce the notion of the \emph{support} of an object $X$ in an $E\mathcal M$-category:

\begin{defi}
Let $\mathcal C$ be an $E\mathcal M$-category, let $X\in\mathcal C$, and let $A\subset\omega$ be any finite set. Then we say that \emph{$X$ is supported on $A$} if $u_*X=X$ for all $u\in\mathcal M$ fixing $A$ pointwise (i.e.~such that $u(a)=a$ for all $a\in A$).

The object $X$ is said to be \emph{finitely supported} if it is supported on some finite set $A$. In this case, its \emph{support} $\supp(X)$ is defined to be the intersection of all finite sets on which it is supported.

An $E\mathcal M$-category $\mathcal C$ is called \emph{tame} if all $X\in\mathcal C$ are finitely supported. We write $\cat{$\bm{E\mathcal M}$-Cat}^\tau\subset\cat{$\bm{E\mathcal M}$-Cat}$ for the full subcategory spanned by the tame (small) $E\mathcal M$-categories.
\end{defi}

\begin{ex}
An object $X\in\mathcal C$ is supported on the empty set if and only if it is an $\mathcal M$-fixed point.
\end{ex}

We now recall some basic properties of the support that we will use without further reference:

\begin{lemma}\label{lemma:support-change}
Let $\mathcal C$ be a tame $E\mathcal M$-category and let $X\in\mathcal C$.
\begin{enumerate}
\item $X$ is supported on $\supp(X)$.
\item If $u\in\mathcal M$ is any injection, then $\supp(u_*X)=u(\supp(X))$.\label{item:sc-action}
\item If $F\colon\mathcal C\to\mathcal D$ is a map of tame $E\mathcal M$-categories, then $\supp(F(X))\subset\supp(X)$.
\item Let $u,v,u',v'\in\mathcal M$ such that $u$ agrees with $u'$ on $\supp(X)$ and $v$ agrees with $v'$ on $\supp(X)$. Then $[v',u']_X=[v,u]_X$. In particular, $u^X_\circ=\id$ whenever $u$ fixes $\supp(X)$ pointwise.
\end{enumerate}
\begin{proof}
These are part of \cite[Proposition~2.13]{schwede-k-theory}.
\end{proof}
\end{lemma}

If $\mathcal C,\mathcal D$ are tame $E\mathcal M$-categories, then we define
\begin{equation*}
\mathcal C\boxtimes\mathcal D\subset\mathcal C\times\mathcal D
\end{equation*}
to be the full subcategory spanned by those $(X,Y)$ such that $\supp(X)\cap\supp(Y)=\varnothing$. It is not hard to show that this defines a subfunctor of $\blank\times\blank$ and that this is the monoidal product of a preferred symmetric monoidal structure on $\cat{$\bm{E\mathcal M}$-Cat}^\tau$, see \cite[Proposition~2.34]{schwede-k-theory}.

\begin{defi}
A \emph{parsummable category} is a commutative monoid in the symmetric monoidal category $(\cat{$\bm{E\mathcal M}$-Cat}^\tau,\boxtimes)$. We write $\cat{ParSumCat}$ for the category whose objects are the parsummable categories and whose morphisms are the monoid homomorphisms.
\end{defi}

Explicitly, this means that we are given a small tame $E\mathcal M$-category $\mathcal C$ together with a distinguished object $0$ having empty support and the following data: we are given for all $X,Y\in\mathcal C$ with $\supp(X)\cap\supp(Y)=\varnothing$ a `sum' $X+Y$ and for all $f\colon X\to X',g\colon Y\to Y'$ such that additionally $\supp(X')\cap\supp(Y')=\varnothing$ a `sum' $f+g$. The sum is required to be \emph{strictly} unital, associative, and commutative in the evident sense. Moreover, it is functorial---i.e.~$(f'+g')\circ(f+g)=(f'\circ f)+(g'\circ g)$ whenever this makes sense---and $E\mathcal M$-equivariant in the sense that $u^{X+Y}_\circ = u^X_\circ+u^Y_\circ$, and hence in particular $u_*(X+Y)=u_*(X)+u_*(Y)$.

\begin{rk}
As explained in \cite[discussion after Definition~4.1]{schwede-k-theory} the name `parsummable' is short for `partially summable,' alluding to the fact that the sum $+$ is not defined on all of $\mathcal C\times\mathcal C$. However, \cite[proof of Theorem~2.32]{schwede-k-theory} in particular tells us that the inclusion $\mathcal C\boxtimes\mathcal C\hookrightarrow\mathcal C\times\mathcal C$ is an equivalence of categories: namely, it is fully faithful by definition, and if $(X,Y)\in\mathcal C\times\mathcal C$ is arbitrary, then $(u_\circ^X,v_\circ^Y)$ defines an isomorphism $(X,Y)\cong(u_*X,v_*Y)$ in $\mathcal C\times\mathcal C$ for any $u,v\in\mathcal M$. If we now pick injections $u$ and $v$ with disjoint images, then $(u_*X,v_*Y)$ already lies in $\mathcal C\boxtimes\mathcal C$, so the inclusion is also essentially surjective, hence an equivalence.

Thus, while two given objects $X,Y\in\mathcal C$ might not be summable, we can always replace them by isomorphic $X',Y'\in\mathcal C$ whose sum is defined.
\end{rk}

On parsummable categories, there exists an interesting natural extension of the $E\mathcal M$-action:

\begin{constr}
Let $A$ be a finite set, and let $\phi\colon A\times\omega\rightarrowtail\omega$ be any injection. Then we define
\begin{equation*}
\phi_*\colon\mathcal C^{\times A}\to\mathcal C
\end{equation*}
via $\phi_*\big(X_\bullet)\mathrel{:=}\sum_{a\in A}\phi(a,\blank)_*(X_a)$ for any $X_\bullet=(X_a)_{a\in A}$, and analogously on morphisms. Here we have used that the $\phi(a,\blank)_*(X_a)$ have pairwise disjoint support by Lemma~\ref{lemma:support-change}-$(\ref{item:sc-action})$ and that the sum is strictly associative and commutative (so that we can sum over arbitrary finite index sets).

Moreover, if $\psi\colon A\times\omega\rightarrowtail\omega$ is any other such injection, then we define
\begin{equation*}
[\psi,\phi]_{X_\bullet}\mathrel{:=}\sum_{a\in A}[\psi(a,\blank),\phi(a,\blank)]_{X_a}\colon\phi_*(X_\bullet)\to \psi_*(Y_\bullet).
\end{equation*}
\end{constr}

This appears as \cite[Construction~5.1]{schwede-k-theory} in the special case that $A=\bm{n}\mathrel{:=}\{1,\dots,n\}$, which is not a real restriction because of commutativity. Note that for $A=\textbf1$ the above recovers the original structure isomorphisms in the following sense: if $X\in\mathcal C$ and $u,v\in\mathcal M$ are arbitrary, then $[v,u]_X=[v\circ\pr,u\circ\pr]_{(X)}$, where $\pr\colon\cat1\times\omega\mathrel{\smash{\xrightarrow{\lower3pt\hbox{$\scriptstyle\cong$}}}}\omega$ is the projection, and we write $(X)$ for $X$ considered as a $1$-tuple (i.e.~a $\bm1$-indexed family).

The following two lemmas follow easily from the definitions and we omit their proofs:

\begin{lemma}\label{lemma:generalized-action}
Let $\mathcal C$ be a parsummable category, let $A$ be a finite set, let $\phi\colon A\times\omega\rightarrowtail\omega$ be any injection, and let $X_\bullet\in\mathcal C^{\times A}$.
\begin{enumerate}
\item If $u\in\mathcal M$ is any injection, then $u_*(\phi_*(X_\bullet))=(u\phi)_*(X_\bullet)$.
\item If $u\in\mathcal M$ is any injection, then $\phi_*((u_*X_a)_{a\in A})=(\phi\circ(\id\times u))_*(X_\bullet)$.
\item If $A=A_1\sqcup A_2$ is any partition of $A$, then $\phi_*(X_\bullet)=(\phi|_{A_1})_*((X_a)_{a\in A_1})+(\phi|_{A_2})_*((X_a)_{a\in A_2})$; in particular, the sum on the right hand side is well-defined. Here we use the shorthand $\phi|_{A_i}$ for $\phi|_{A_i\times\omega}$ ($i=1,2$).\label{item:lga-partition}
\item If $\sigma\colon A'\to A$ is bijective, then $\phi_*\big(X_\bullet)=(\phi\circ(\sigma\times\id))_*(\sigma^*X_\bullet)$, where $\sigma^*X_\bullet$ is the $A'$-indexed family with $(\sigma^*X_\bullet)_{a'}=X_{\sigma(a')}$ for all $a'\in A'$.\label{item:lga-perm}
\qed
\end{enumerate}
\end{lemma}

\begin{lemma}\label{lemma:generalized-structure-maps}
Let $\mathcal C$ be a parsummable category, let $A$ be a finite set, let $\phi,\psi\colon A\times\omega\rightarrowtail\omega$ be injections, and let $X_\bullet\in\mathcal C^{\times A}$.
\begin{enumerate}
\item If $u\in\mathcal M$ is any injection, then $u_*([\psi, \phi]_{X_\bullet})=[u\psi, u\phi]_{X_\bullet}$. Moreover, if $v\in\mathcal M$ is yet another injection, then $[v,u]_{\phi_*(X_\bullet)}=[v\phi,u\phi]_{X_\bullet}$.\label{item:gsm-EM-equivariant}
\item If $u\in\mathcal M$ is any injection, then $[\psi,\phi]_{((u_*X_a)_{a\in A})}=[\psi\circ(\id\times u),\phi \circ(\id\times u)]_{X_\bullet}$.\label{item:gsm-balanced}
\item If $A=A_1\sqcup A_2$ is any partition of $A$, then $[\psi,\phi]_{X_\bullet}=[\psi|_{A_1},\phi|_{A_1}]_{(X_a)_{a\in A_1}}+[\psi|_{A_2},\phi|_{A_2}]_{(X_a)_{a\in A_2}}$.
\item If $\sigma\colon A'\to A$ is bijective, then $[\psi,\phi]_{X_\bullet}=[\psi\circ(\sigma\times\id),\phi\circ(\sigma\times\id)]_{\sigma^*X_\bullet}$.
\item $[\phi,\phi]=\id_{\phi_*(\blank)}$
\item If $\theta\colon A\times\omega\rightarrowtail\omega$ is yet another injection, then $[\theta,\psi]\circ[\psi,\phi]=[\theta,\phi]$.\qed
\end{enumerate}
\end{lemma}

In part $(\ref{item:lga-perm})$ of the above lemmas we will be particularly interested in the case $A=A'$; in this case, replacing $X$ by $(\sigma^{-1})^*X$ shows that $(\phi\circ(\sigma\times\id))_*(X_\bullet)=\phi_*(\sigma.X_\bullet)$ and similarly for the structure isomorphisms, where $\Sigma_A$ acts on $\mathcal C^{\times A}$ from the left as usual, i.e.~$(\sigma.X_\bullet)_a=((\sigma^{-1})^*X_\bullet)_a=X_{\sigma^{-1}(a)}$.

\begin{rk}\label{rk:universally-natural-trafo}
It is evident from the construction that the above functors and transformations are preserved by any morphism of parsummable categories, i.e.~the above natural transformations for varying $\mathcal C$ assemble into \emph{universally natural transformations} in the sense of \cite[Definition~5.4]{schwede-k-theory}. One can in fact show that there is for any pair $\phi,\psi$ of injections $A\times\omega\rightarrowtail\omega$ precisely one universally natural transformation from $\phi_*$ to $\psi_*$, see \cite[Proposition~5.5]{schwede-k-theory}. This provides an alternative way to prove the second lemma above.
\end{rk}

\begin{rk}
As mentioned without proof in \cite[Remark~4.20]{schwede-k-theory}, the above maps make any parsummable category $\mathcal C$ into an algebra over the so-called \emph{injection operad} $\mathcal I$, whose $n$-ary operations are given by $\mathcal I(n)\mathrel{:=}E\Inj(\bm n\times\omega,\omega)$ where $\Inj$ denotes the set of injections. While we will not need this perspective below, we remark that (for $A=\bm n$) several of the above statements are actually consequences of this structure: for example, parts $(\ref{item:gsm-EM-equivariant})$ and $(\ref{item:gsm-balanced})$ are special cases of the compatibility of the action maps $\mathcal I(n)\times\mathcal C^{\times n}\to\mathcal C$ with the structure maps of the operad $\mathcal I$.
\end{rk}

\subsection{Permutative categories}
We recall that a \emph{permutative category} is a symmetric monoidal category in which the associativity and unitality isomorphisms are the respective identities. Explicitly this means, cf.~\cite[Definition~4.1]{may-permutative}:

\begin{defi}
A permutative category is a triple of a category $\mathscr C$ equipped with a functor $\blank\otimes\blank\colon\mathscr C\times\mathscr C\to\mathscr C$, that is strictly associative in the sense that
\begin{equation*}
\begin{tikzcd}[column sep=large]
\mathscr C\times\mathscr C\times\mathscr C\arrow[d, "\id\times(\blank\otimes\blank)"']\arrow[r, "(\blank\otimes\blank)\times\id"] & \mathscr C\times\mathscr C\arrow[d, "\blank\otimes\blank"]\\
\mathscr C\times\mathscr C\arrow[r, "\blank\otimes\blank"'] & \mathscr C
\end{tikzcd}
\end{equation*}
commutes strictly, and strictly unital in the sense that there exists a (necessarily unique) $\textbf{1}\in\mathscr C$ satisfying $\blank\otimes\textbf{1}=\id=\textbf{1}\otimes\blank$, together with a natural transformation $\tau\colon (\blank\otimes\blank)\Rightarrow(\blank\otimes\blank)\circ\text{twist}$, where $\text{twist}$ is the automorphism of $\mathscr C\times\mathscr C$ exchanging the two factors, such that the following conditions are satisfied:
\begin{enumerate}
\item (\textit{$\tau$ is unital}) For all $X\in\mathscr C$, $\tau_{\textbf{1},X}=\id_X=\tau_{X,\textbf{1}}$.
\item (\textit{$\tau$ is self-inverse}) For all $X,Y\in\mathscr C$, $\tau_{Y,X}\tau_{X,Y}=\id_{X\otimes Y}$.
\item (\textit{$\tau$ is associative}) For all $X,Y,Z\in\mathscr C$ the following diagram commutes:
\begin{equation*}
\begin{tikzcd}[column sep=small]
X\otimes Y\otimes Z\arrow[rr, "\tau_{X\otimes Y,Z}"]\arrow[rd, "\id_X\otimes\tau_{Y,Z}"', bend right=5pt] && Z\otimes X\otimes Y.\\
& X\otimes Z\otimes Y\arrow[ur, bend right=5pt, "\tau_{X,Z}\otimes \id_Y"']
\end{tikzcd}
\end{equation*}
\end{enumerate}
\end{defi}

We will denote permutative categories by the script letters $\mathscr C$, $\mathscr D$, and so on.

In the statement of the third condition above, we have already used that strict associativity means that we do not have to care about in which way an expression like $X\otimes Y\otimes Z$ is bracketed. In the same way, we can therefore make sense of $\bigotimes_{i\in I}X_i$ for any finite totally ordered set $I$ and any family $(X_i)_{i\in I}$ of objects of $\mathscr C$. More generally, if $I$ is any (not necessarily finite) totally ordered set and $(X_i)_{i\in I}$ is a family of objects in $\mathscr C$, almost all of which are equal to $\textbf{1}$, then we can define $\bigotimes_{i\in I}X_i\mathrel{:=}\bigotimes_{i\in J}X_i$ where $J\subset I$ is any finite set containing all $i\in I$ with $X_i\not=\bm1$; by strict unitality, this is independent of the choice of $J$ and in particular it agrees with the previous definition whenever $I$ itself is already finite.

We also recall from \cite{may-permutative} the \emph{categorical Barratt-Eccles operad} $E\Sigma_*$; the operad structure appears implicitly in \cite[discussion after Proposition~4.2; Lemma~4.4]{may-permutative}, and it is then remarked as \cite[Lemma~4.8]{may-permutative} that taking classifying spaces levelwise recovers the classical Barratt-Eccles operad.

Explicitly, the $n$-ary operations are given by $(E\Sigma_*)_n=E\Sigma_n$ with right $\Sigma_n$-action via $(\pi_2,\pi_1).\sigma=(\pi_2\sigma,\pi_1\sigma)$, and the structure maps are induced from
\begin{align*}
\Sigma_k\times\Sigma_{n_1}\times\cdots\times\Sigma_{n_k}&\to\Sigma_{n_1+\cdots+n_k}\\
(\sigma,\pi_1,\dots,\pi_k)&\mapsto\sigma_{(n_1,\dots,n_k)}\circ (\pi_1\times\cdots\times\pi_k) \\&\qquad{}= (\pi_{\sigma^{-1}(1)}\times\cdots\times\pi_{\sigma^{-1}(k)})\circ\sigma_{(n_1,\dots,n_k)};
\end{align*}
here $\pi_1\times\cdots\times\pi_k$ denotes the image of $(\pi_1,\dots,\pi_k)$ under the usual `block sum' embedding $\Sigma_{n_1}\times\cdots\times\Sigma_{n_k}\hookrightarrow\Sigma_{n_1+\cdots+n_k}$ (and similarly for $\pi_{\sigma^{-1}(1)}\times\cdots\times\pi_{\sigma^{-1}(k)}$), and $\sigma_{(n_1,\dots,n_k)}\in\Sigma_{n_1+\cdots+n_k}$ is the permutation shuffling the blocks $\{1,\dots,n_1\},\{n_1+1,\dots,n_1+n_2\},\dots,\{n_1+\cdots+n_{k-1}+1,\dots,n_1+\cdots+n_k\}$ according to $\sigma$, i.e.
\begin{equation*}
\sigma_{(n_1,\dots,n_k)}(n_1+n_2+\cdots+n_{\ell-1}+m)=n_{\sigma^{-1}(1)}+n_{\sigma^{-1}(2)}+\cdots+n_{\sigma^{-1}(\sigma(\ell)-1)}+m
\end{equation*}
for all $1\le\ell\le k$ and $1\le m\le n_{\ell}$; beware that there is a typo in \cite[Lemma~4.4]{may-permutative} where the permutation $\sigma_{(n_1,\dots,n_k)}$ is missing, so that the action map is mistakenly written as $(\sigma,\pi_1,\dots,\pi_k)\mapsto \pi_{\sigma^{-1}(1)}\times\cdots\times\pi_{\sigma^{-1}(k)}$.

\begin{thm}[May]\label{thm:permutative-operadic-description}
If $\mathscr C$ is any permutative category, then there is a unique way to equip $\mathscr C$ with the structure of an algebra over $E\Sigma_*$ such that the following holds:
\begin{enumerate}
\item For any $n\ge 0$, the restriction of the action map $\alpha_n\colon E\Sigma_n\times\mathscr C^{\times n}\to\mathscr C$ to the object $\id\in E\Sigma_n$ is given by the functor $(X_1,\dots,X_n)\mapsto X_1\otimes\cdots\otimes X_n$ (and likewise on morphisms). In particular, $*\cong E\Sigma_0\times\mathscr C^{\times0}\to\mathscr C$ is given by the inclusion of $\textbf{\textup 1}$.
\item For any $1\le k<n$ and any $X_1,\dots,X_n\in\mathscr C$, the image of $\big(((k\; k+1),\id), (\id_{X_1},\dots,\id_{X_n})\big)$ under the action map is the morphism
\begin{align*}
\bigotimes_{i=1}^n X_i&=\bigotimes_{i=1}^{k-1}X_i\otimes X_k\otimes X_{k+1}\otimes\bigotimes_{i=k+2}^n X_i\\&\xrightarrow{\bigotimes_{i=1}^{k-1}\id_{X_i}\otimes\tau_{X_k,X_{k+1}}\otimes\bigotimes_{i=k+2}^n\id_{X_i}}
\bigotimes_{i=1}^{k-1}X_i\otimes X_{k+1}\otimes X_{k}\otimes\bigotimes_{i=k+2}^nX_i.
\end{align*}
\end{enumerate}
\begin{proof}
There is indeed such an algebra structure by \cite[Lemmas~4.3~and~4.4]{may-permutative}. For uniqueness we observe that the action map $\alpha_n$ is determined on maps of the form $\big((\sigma,\sigma),(f_1,\dots,f_n)\big)$ for all morphisms $f_1,\dots,f_n$ in $\mathscr C$ and permutations $\sigma\in\Sigma_n$ by the first condition together with the fact that $\alpha_n$ has to satisfy the relation $\alpha_n(\blank.\sigma,\blank)=\alpha_n(\blank,\sigma.\blank)$ as it is a part of an $E\Sigma_*$-algebra structure. Likewise, the second condition together with $\Sigma_n$-equivariance determines the image of $\big((\pi\sigma,\sigma), (\id_{X_1},\dots,\id_{X_n})\big)$ for any $\sigma\in\Sigma_n$, $X_1,\dots,X_n\in\mathscr C$ and any transposition $\pi$ of the form $(k\;k+1)$. By functoriality, the action map is then more generally determined on $\big((\pi_r\cdots\pi_1\sigma,\sigma),(\id_{X_1},\dots,\id_{X_n})\big)$ whenever $\pi_1,\dots,\pi_r$ are transpositions as above. Finally, if $\vartheta\in\Sigma_n$ is arbitrary, then we can write $\vartheta\sigma^{-1}=\pi_r\cdots\pi_1$ with $\pi_i$ as above (as these transpositions generate $\Sigma_n$), finishing the proof.
\end{proof}
\end{thm}

\begin{rk}\label{rk:rk-before-coherence-isomorphisms}
In the above we have only used equivariance and functoriality of the individual action maps $E\Sigma_n\times\mathscr C^{\times n}\to\mathscr C$. In fact, using the compatibility of these with the structure maps of the operad $E\Sigma_*$, it would have been enough to assume the above conditions for $k\le2$.
\end{rk}

In particular,
\begin{align*}
\alpha_n\big(\sigma, (X_1,\dots,X_n)\big)&=\alpha_n\big(\id,\sigma.(X_1,\dots,X_n)\big)=\alpha_n\big(\id,(X_{\sigma^{-1}(1)},\dots,X_{\sigma^{-1}(n)})\big)\\
&=\bigotimes_{i=1}^n X_{\sigma^{-1}(i)}
\end{align*}
for all $X_1,\dots,X_n\in\mathscr C$, $\sigma\in\Sigma_n$, and similarly for morphisms.

Thus, evaluating at $(\sigma,1)$ gives us for each $X_1,\dots,X_n$ as above a specific natural isomorphism
$\bigotimes_{i=1}^n X_i\to \bigotimes_{i=1}^n X_{\sigma^{-1}(i)}$, usually referred to as the \emph{coherence isomorphism associated to $\sigma$}. More generally, if $A,B$ are finite totally ordered sets, and $\sigma\colon A\to B$ is a bijection of their underlying sets, then we get a coherence isomorphism $\bigotimes_{a\in A}X_a\to\bigotimes_{b\in B}X_{\sigma^{-1}(b)}$ associated to $\sigma$ for any $X_\bullet\in\mathcal C^{\times A}$ by applying the above to the permutation
\begin{equation*}
\{1,\dots,|A|\}\xrightarrow\pi A\xrightarrow\sigma B\xrightarrow\vartheta\{1,\dots,|B|\}=\{1,\dots,|A|\}
\end{equation*}
and the $|A|$-tuple $\pi^*X_\bullet=(X_{\pi(1)},\dots,X_{\pi(|A|)})$, where $\pi$ and $\vartheta$ are the unique order preserving bijections.

\begin{rk}
While May does not prove this, it is in fact easy to show (once the above theorem has been established) that the above construction provides an isomorphism of categories between the category $\cat{PermCat}$ of small permutative categories and strict symmetric monoidal functors and the category of $E\Sigma_*$-algebras in $\cat{Cat}$. While this is an important result, we will only need the above version of the theorem, and we will only use it as a convenient way to prove that several coherence isomorphisms agree.
\end{rk}

While general symmetric monoidal categories will play no role below, we still recall for motivational purposes:

\begin{thm}\label{thm:perm-vs-sym}
The inclusion $\cat{PermCat}\to\cat{SymMonCat}$ of small permutative categories (and \emph{strict} symmetric monoidal functors!) into the category of small symmetric monoidal categories and \emph{strong} symmetric monoidal functors is a homotopy equivalence with respect to the underlying equivalences of categories on both sides.
\end{thm}

Here we call a functor $F\colon\mathscr C\to\mathscr D$ between two categories equipped with classes of weak equivalences containing all isomorphisms and closed under $2$-out-of-$3$ a \emph{homotopy equivalence} if it is homotopical (i.e.~it sends weak equivalences to weak equivalences) and it admits a \emph{homotopy inverse}, i.e.~a homotopical functor $G\colon\mathscr D\to\mathscr C$ such that the two composites $GF$ and $FG$ can be connected by zig-zags of natural transformations, that are levelwise given by weak equivalences, to the respective identities. A homotopy equivalence in particular induces an equivalence of (ordinary or $\infty$-categorical) localizations at the given classes of weak equivalences.

\begin{proof}[Sketch of proof]
This result is well-known, but I do not know of an explicit reference, so let me briefly explain how we can prove this using several results from the literature.

We first recall from \cite[Theorem XI.3.1]{cat-working} Mac Lane's construction that associates to any monoidal category $\mathscr C$ a strict monoidal category $J\mathscr C$ together with strong monoidal functors $\mu\colon J\mathscr C\to\mathscr C$ and $\nu\colon\mathscr C\to J\mathscr C$ that are underlying equivalences of categories:

The category $J\mathscr C$ has objects the finite sequences $(X_1,\dots,X_m)$; here $m=0$ is allowed, in which case we also denote the resulting object by $\epsilon$. We moreover define
\begin{equation*}
\Hom_{J\mathscr C}((X_1,\dots,X_m), (Y_1,\dots,Y_n))\mathrel{:=}\Hom_{\mathscr C}(\mu(X_1,\dots,X_m), \mu(Y_1,\dots,Y_n))
\end{equation*}
with
\begin{equation*}
\mu(Z_1,\dots,Z_o)\mathrel{:=}(\cdots(Z_1\otimes Z_2)\otimes\cdots)\otimes Z_o
\end{equation*}
where we have bracketed `from left to right.' The composition in $J\mathscr C$ is given by the composition in $\mathscr C$.

We make $J\mathscr C$ into a strict monoidal category as follows: the unit is given by $\epsilon$ and the tensor product is given on the level of objects by concatenation; for the definition on morphisms see \emph{loc.cit.}

The functor $\mu$ is defined on objects as above and it is given on hom sets by the respective identities. The strong monoidal structure on $\mu$ is as follows: the unitality isomorphism is the identity, and for all $(X_1,\dots,X_m),(Y_1,\dots,Y_n)\in J\mathscr C$ the structure isomorphism
\begin{align*}
\mu(X_1,\dots,X_m)\otimes\mu(Y_1,\dots,Y_n)&\to\mu((X_1,\dots,X_m)\otimes(Y_1,\dots,Y_n))\\
&=\mu(X_1,\dots,X_m,Y_1,\dots,Y_n)
\end{align*}
is given by the evident composition of the associativity isomorphisms.

On the other hand, we define $\nu\colon\mathscr C\to J\mathscr C$ on objects by $\nu(X)=(X)$; again, $\nu$ is the identity on hom-sets. We make $\nu$ into a strong symmetric monoidal functor by choosing both the unitality isomorphism $\epsilon\to(\textbf1)$ as well as the isomorphisms $(X,Y)\to (X\otimes Y)$ to be the corresponding identity morphisms in $\mathscr C$ (but note that these are \emph{not} identity morphisms in $J\mathscr C$ as source and target do not agree!).

As sketched in \cite[proof of Proposition~4.2]{may-permutative}, $J\mathscr C$ actually becomes a permutative category whenever $\mathscr C$ is symmetric monoidal by taking the symmetry isomorphism $(X_1,\dots,X_m)\otimes (Y_1,\dots,Y_n)\to (Y_1,\dots,Y_n)\otimes (X_1,\dots,X_m)$ to be the composition
\begin{align*}
&\mu(X_1,\dots,X_m,Y_1,\dots,Y_n)\cong\mu(X_1,\dots,X_m)\otimes\mu(Y_1,\dots,Y_n)\\
&\quad\xrightarrow{\tau_{\mu(X_1,\dots,X_m),\mu(Y_1,\dots,Y_n)}}\mu(Y_1,\dots,Y_n)\otimes\mu(X_1,\dots,X_m)\\
&\quad\cong\mu(Y_1,\dots,Y_n,X_1,\dots,X_m)
\end{align*}
where the outer isomorphisms are defined as before. With respect to this structure, $\mu$ and $\nu$ are actually strong \emph{symmetric} monoidal.

It is not hard to check that $J$ becomes a functor $\cat{SymMonCat}\to\cat{PermCat}$ if we define $JF\colon J\mathscr C\to J\mathscr D$ for any strong symmetric monoidal functor $F\colon\mathscr C\to\mathscr D$ via
$(JF)(X_1,\dots,X_m)=(FX_1,\dots,FX_m)$ and on morphisms as follows: if $f\colon (X_1,\dots,X_m)\to (Y_1,\dots,Y_n)$ is any morphism (i.e.~$f$ is a morphism $\bigotimes_{i=1}^m X_i\to\bigotimes_{i=1}^nY_i$ in $\mathscr C$), then $(JF)(f)\colon (FX_1,\dots,FX_m)\to (FY_1,\dots, FY_n)$ is given by the composition
\begin{equation*}
\bigotimes_{i=1}^m FX_i\xrightarrow{\cong} F\left(\bigotimes_{i=1}^m X_i\right)\xrightarrow{Ff}F\left(\bigotimes_{i=1}^nY_i\right)\xrightarrow{\cong} \bigotimes_{i=1}^n F(Y_i)
\end{equation*}
where the unlabeled isomorphism on the left is given by the unitality isomorphism of the strong symmetric monoidal functor $F$ if $m=0$ or by the composition
\begin{align*}
&\phantom{\to}\bigotimes_{i=1}^m FX_i \xrightarrow{\nabla_{X_1,X_2}\otimes\bigotimes_{i=3}^m\id_{F(X_i)}} F(X_1\otimes X_2)\otimes\bigotimes_{i=3}^n F(X_i)\to\cdots\\
&\to F\left(\bigotimes_{i=1}^{m-1}X_i\right)\otimes F(X_m)\xrightarrow{\nabla_{\bigotimes_{i=1}^{m-1}X_i,X_m}}F\left(\bigotimes_{i=1}^m X_i\right)
\end{align*}
of the structure isomorphisms $\nabla$ of $F$ if $m>0$, and the remaining isomorphism is defined analogously; here all the tensor products are bracketed from left to right again. It is moreover not hard to check that with respect to this the strong symmetric monoidal functor $\nu$ is natural, yielding a natural levelwise underlying equivalence $\id\Rightarrow\incl\circ J$.

However, as $\nu$ is not a \emph{strict} symmetric monoidal functor, it does not define a transformation $\id\Rightarrow J\circ\incl$. On the other hand, if the associativity isomorphisms in $\mathscr C$ are strict (for example, if $\mathscr C$ is permutative), then $\mu$ is actually strict symmetric monoidal, and while it is in general only pseudonatural in \emph{strong} symmetric monoidal functors, it is strictly natural in \emph{strict} monoidal functors. Thus, $\mu$ provides a natural underlying equivalence $J\circ\incl\Rightarrow\id$, which completes the proof of the theorem.
\end{proof}

\begin{warn}
The above theorem should not be mistaken for a statement about the corresponding $2$-categories. In particular, it is \emph{not} true that any strong symmetric monoidal functor between permutative categories is isomorphic to a strict symmetric monoidal functor.
\end{warn}

\begin{rk}\label{rk:hom-sets-not-disjoint}
As we will later have to be careful about a similar issue (see Warning~\ref{warn:Phi-hom-sets}), we want to emphasize again that $\nu$ is not a strict monoidal functor: although its structure isomorphisms are given by identity arrows in $\mathscr C$, they are not the identity arrows in $J\mathscr C$; in fact, the object $\nu(X)\otimes\nu(Y)$ is never equal to $\nu(X\otimes Y)$.

More generally, two morphisms between different pairs of objects in $J\mathscr C$ might be given by the same morphism of $\mathscr C$, so giving a morphism in $\mathscr C$ is not enough to describe a morphism in $J\mathscr C$ as long as we do not fix source and target.
\end{rk}

\subsection{From permutative categories to parsummable categories}
Finally, we recall the functor $\Phi\colon\cat{PermCat}\to\cat{ParSumCat}$ arising in Schwede's construction of the global algebraic $K$-theory of symmetric monoidal categories. We will be as brief as possible, and in particular we will omit almost all proofs; the curious reader can find the relevant details in \cite[Section~11]{schwede-k-theory}.

\begin{constr}
Let $\mathscr C$ be a small permutative category. We define a parsummable category $\Phi(\mathscr C)$ as follows: an object of $\Phi(\mathscr C)$ is a countably infinite sequence $X_\bullet\mathrel{:=}(X_1,X_2,\dots)$ of objects of $\mathscr C$ such that $X_i=\bm1$ for almost all $i\in\omega$. If $Y_\bullet\mathrel{:=}(Y_1,Y_2,\dots)$ is another such object, then we define
\begin{equation*}
\Hom_{\Phi(\mathscr C)}(X_\bullet,Y_\bullet)\mathrel{:=}\Hom_{\mathscr C}\left(\bigotimes_{i\in\omega} X_i,\bigotimes_{i\in\omega} Y_i\right)
\end{equation*}
with composition given by the composition in $\mathscr C$.

If now $u\in\mathcal M$ is any injection, then $u_*(X_\bullet)$ is defined via
\begin{equation*}
u_*(X_\bullet)_i=\begin{cases}
X_j & \text{if }i=u(j)\\
\textbf{1} & \text{if }i\notin\im(u).
\end{cases}
\end{equation*}
The structure isomorphism $u^{X_\bullet}_\circ$ is given as follows: we pick $K\gg0$ such that $X_i=\bm1=u_*(X_\bullet)_i$ for all $i>K$, and a permutation $\sigma\in\Sigma_K$ such that $\sigma(i)=u(i)$ for all $i\le K$ with $X_i\not=\bm1$. Then $u^{X_\bullet}_\circ$ is the coherence isomorphism
\begin{equation*}
\bigotimes_{i\in\omega} X_i=\bigotimes_{i=1}^KX_i\to\bigotimes_{i=1}^K X_{\sigma^{-1}(i)}=\bigotimes_{i=1}^K u_*(X_\bullet)_i=\bigotimes_{i\in\omega} u_*(X_\bullet)_i
\end{equation*}
associated to $\sigma$; we omit the verification that $u^{X_\bullet}_\circ$ is independent of choices and that the relation $(\ref{eq:u-circ-relation})$ holds, making $\Phi(\mathscr C)$ into an $E\mathcal M$-category.

One easily checks that $\supp(X_\bullet)=\{i\in\omega : X_i\not=\bm1\}$. Thus, we can define the sum $X_\bullet+Y_\bullet$ of two disjointly supported objects by
\begin{equation*}
(X_\bullet+Y_\bullet)_i=\begin{cases}
X_i & \text{if }i\in\supp(X_\bullet)\\
Y_i & \text{otherwise}.
\end{cases}
\end{equation*}
If $X'_\bullet,Y'_\bullet$ are two further disjointly supported objects, and $f\colon X_\bullet\to X'_\bullet$, $g\colon Y_\bullet\to Y'_\bullet$ are morphisms, then we define $f+g$ as the composition
\begin{equation*}
\bigotimes_{i\in\omega}(X_\bullet+Y_\bullet)_i\cong
\bigotimes_{i\in\omega} X_i\otimes \bigotimes_{i\in\omega} Y_i\xrightarrow{f\otimes g}\bigotimes_{i\in\omega} X_i'\otimes \bigotimes_{i\in\omega} Y_i'\cong\bigotimes_{i\in\omega}(X_\bullet'+Y_\bullet')_i.
\end{equation*}
Here the unlabeled isomorphism on the left is given by the coherence isomorphism
\begin{equation}\label{eq:Phi-sum-of-morphisms-coherence}
\begin{aligned}
&\bigotimes_{i\in\omega}(X_\bullet+Y_\bullet)_i=
\bigotimes_{i\in\supp(X_\bullet)\cup\supp(Y_\bullet)}(X_\bullet+Y_\bullet)_i
\\
&\quad\xrightarrow{\cong}\bigotimes_{i\in\supp(X_\bullet)} X_i\otimes\bigotimes_{i\in\supp(Y_\bullet)} Y_i=
\bigotimes_{i\in\omega} X_i\otimes\bigotimes_{i\in\omega} Y_i
\end{aligned}
\end{equation}
corresponding to the evident permutation, and the remaining isomorphism is defined analogously. More precisely, we can make the disjoint union $\supp(X_\bullet)\amalg\supp(Y_\bullet)$ into a totally ordered set by demanding that the ordering on each individual summand be as before but that in addition any element of the first summand should be smaller than any element of the second summand. By disjointness, there is a tautological bijection $\supp(X_\bullet)\cup\supp(Y_\bullet)\to\supp(X_\bullet)\amalg\supp(Y_\bullet)$ from the internal union to the disjoint union, and if we view the left hand side as a totally ordered set via the total ordering on $\omega$, then $(\ref{eq:Phi-sum-of-morphisms-coherence})$ is the coherence isomorphism (in the sense of the discussion after Remark~\ref{rk:rk-before-coherence-isomorphisms}) associated to this tautological bijection.

Finally, if $F\colon\mathscr C\to\mathscr D$ is a strict symmetric monoidal functor of permutative categories, then we define $\Phi(F)\colon\Phi(\mathscr C)\to\Phi(\mathscr D)$ as follows: $\Phi(F)(X_\bullet)_i=F(X_i)$ for all $i\in\omega$ and $X_\bullet\in\Phi(\mathscr C)$, and if $f\colon X_\bullet\to Y_\bullet$ is any morphism, then
\begin{equation*}
\Phi(F)(f):\bigotimes_{i\in\omega} F(X_i)=F\left(\bigotimes_{i\in\omega}X_i\right) \xrightarrow{Ff}
F\left(\bigotimes_{i\in\omega}Y_i\right) = \bigotimes_{i\in\omega} F(Y_i).
\end{equation*}
We omit the verification that this defines a functor $\Phi\colon\cat{PermCat}\to\cat{ParSumCat}$.
\end{constr}

\begin{warn}\label{warn:Phi-hom-sets}
Similarly to Remark~\ref{rk:hom-sets-not-disjoint}, we have to remember that writing down a morphism in $\mathscr C$ is \emph{not} enough to pin down a morphism in $\Phi\mathscr C$ as this does not record the source and target of the map. This means that we will at several points below first need to verify that two given morphisms, which we want to prove to be equal, actually share the same source as well as the same target. The same caveat applies to several of the constructions below.
\end{warn}

\begin{rk}
We can define $\Phi(F)$ more generally for strong symmetric monoidal functors $F$ that are strictly unital in the sense that the structure isomorphism $\textbf1\to F(\textbf1)$ is the identity. However, in this generality the definition of $\Phi(F)$ on morphisms becomes more complicated, see \cite[Construction~11.6]{schwede-k-theory}.
\end{rk}

\begin{lemma}\label{lemma:Phi-reflect-preserve}
$\Phi$ preserves and reflects underlying equivalences.
\begin{proof}
If $\mathscr C$ is any permutative category, then the underlying category of $\Phi(\mathscr C)$ is naturally equivalent to the underlying category of $\mathscr C$ by \cite[Remark~11.4]{schwede-k-theory}. Thus, the claim follows from $2$-out-of-$3$.
\end{proof}
\end{lemma}

\section{From parsummable categories to permutative categories}\label{sec:sigma-construction}
In this section we will construct the functor $\Sigma\colon\cat{ParSumCat}\to\cat{PermCat}$ that we will later prove to be the desired homotopy inverse of $\Phi$.

Let us begin by giving some intuition: if $\mathcal C$ is any parsummable category, then the tensor product on $\Sigma\mathcal C$ should somehow be obtained from the sum in $\mathcal C$. We have seen that while a given pair $(X,Y)\in\mathcal C^{\times 2}$ might not be summable, the sum $u_*X+v_*Y$ exists for all $u,v\in\mathcal M$ with disjoint images. Such a choice of $u,v$ is equivalent to a choice of an injection $\mu\colon\cat2\times\omega\rightarrowtail\omega$ via $u=\mu(1,\blank)$ and $v=\mu(2,\blank)$, in which case $u_*X+v_*Y=\mu_*(X,Y)$ by definition. This motivates thinking of $\mu_*\colon\mathcal C^{\times 2}\to\mathcal C$ for any $\mu\colon\cat2\times\omega\rightarrowtail\omega$ as a `generalized sum' and the tensor product we are looking after, and more generally of $\phi_*\colon\mathcal C^{\times m}\to\mathcal C$ for any injection $\phi\colon\bm m\times\omega\rightarrowtail\omega$ as an `$m$-ary tensor product.'

While each choice of $\mu$ as above indeed gives rise to a symmetric monoidal structure on the underlying category of $\mathcal C$ \cite[Proposition~5.7]{schwede-k-theory} (also see Construction~\ref{constr:comparison-schwede} below), it will usually not be strictly associative or unital on objects. Of course, we can appeal to Theorem~\ref{thm:perm-vs-sym} to strictify this into a permutative category, but the result is rather hard to deal with.

Instead, we will give a direct construction. Following the explicit description of the strictification in Theorem~\ref{thm:perm-vs-sym}, its objects should be finite sequences of objects in $\mathcal C$, with morphisms given by maps in $\mathcal C$ between iterated tensor products. By the above intuition, any injection $\phi\colon\bm m\times\omega\rightarrowtail\omega$ yields an `$m$-ary tensor product' $\phi_*$, and none of these is preferred over the others. This motivates allowing all such injections at the same time:

\begin{constr}
Let $\mathcal C$ be a parsummable category, let $m,n\ge 0$ be integers, and let $X_\bullet=(X_1,\dots, X_m)\in\mathcal C^{\times m}, Y_\bullet\mathrel=(Y_1,\dots,Y_n)\in\mathcal C^{\times n}$. We define
\begin{equation*}
\Sigma_{\mathcal C}(X_\bullet,Y_\bullet)\mathrel{:=}\{(\psi,f,\phi) : \phi\colon \bm{m}\times\omega\rightarrowtail\omega, \psi\colon\bm{n}\times\omega\rightarrowtail\omega,
f\colon \phi_*(X_\bullet)\to \psi_*(Y_\bullet)\}/\sim
\end{equation*}
where $(\psi,f,\phi)\sim(\psi',f',\phi')$ if and only if $f'=[\psi',\psi]\circ f\circ[\phi,\phi']$; here and in what follows we will usually omit the indices of the natural transformations from Lemma~\ref{lemma:generalized-structure-maps} (and the structure isomorphisms of the $E\mathcal M$-action) as long as no confusion can arise from this.

We omit the easy verification that $\sim$ is an equivalence relation (which uses Lemma~\ref{lemma:generalized-structure-maps}). In particular, the above condition is (by symmetry of the relation) equivalent to $f=[\psi,\psi']\circ f'\circ[\phi',\phi]$.
\end{constr}

\begin{lemma}\label{lemma:free-to-choose-phi-psi}
In the situation above let $\phi\colon\bm{m}\times\omega\rightarrowtail\omega$ and $\psi\colon\bm{n}\times\omega\rightarrowtail\omega$ be arbitrary injections. Then
\begin{align*}
\Hom_{\mathcal C}(\phi_*(X_\bullet),\psi_*(Y_\bullet))&\to\Sigma_{\mathcal C}(X_\bullet,Y_\bullet)\\
f&\mapsto [\psi,f,\phi]
\end{align*}
is bijective.
\begin{proof}
It is immediate from the definition of the equivalence relation that this is injective. To see that it is also surjective, we observe that a general element on the right hand side is of the form $[\psi',f',\phi']$ with $\phi'\colon\bm{m}\times\omega\rightarrowtail\omega$, $\psi'\colon\bm{n}\times\omega\rightarrowtail\omega$ and $f'\colon\phi_*'(X_\bullet)\to\psi_*'(Y_\bullet)$. Then $[\psi',f',\phi']=[\psi,[\psi,\psi']\circ f'\circ[\phi',\phi],\phi]$
by the above reformulation of the equivalence relation, so $[\psi',f',\phi']$ is in particular contained in the image.
\end{proof}
\end{lemma}

\begin{constr}
Let $\mathcal C$ be a parsummable category. We define $\Sigma(\mathcal C)$ to be the following (small) category: an object of $\Sigma(\mathcal C)$ is a finite sequence $X_\bullet=(X_1,\dots,X_m)$ of objects of $\mathcal C$; here $m=0$ is allowed, in which case we also denote the resulting object by $\epsilon$. If $Y_\bullet=(Y_1,\dots,Y_n)$ is another object, then
\begin{equation*}
\Hom_{\Sigma\mathcal C}(X_\bullet, Y_\bullet)\mathrel{:=}\Sigma_{\mathcal C}(X_\bullet,Y_\bullet).
\end{equation*}
If $Z_\bullet$ is yet another object, then the composition of $[\psi,f,\phi]\colon X_\bullet\to Y_\bullet$ and $[\rho,g,\theta]\colon\allowbreak Y_\bullet\to Z_\bullet$ is defined to be
\begin{equation*}
[\rho,g,\theta]\circ[\psi,f,\phi]\mathrel{:=}  [\rho,g\circ[\theta,\psi]\circ f,\phi].
\end{equation*}
If $F\colon\mathcal C\to\mathcal D$ is a morphism of parsummable categories, then we define the functor $\Sigma(F)\colon\Sigma(\mathcal C)\to\Sigma(\mathcal D)$ on objects via $\Sigma(F)(X_1,\dots,X_m)=(FX_1,\dots, FX_m)$ and on morphisms via $\Sigma(F)[\psi,f,\phi]=[\psi, Ff,\phi]$.
\end{constr}

\begin{prop}\label{prop:sigma-cat}
$\Sigma$ is a well-defined functor $\cat{ParSumCat}\to\cat{Cat}$.
\begin{proof}
Let us first show that $\Sigma(\mathcal C)$ is a well-defined category for any parsummable category $\mathcal C$, i.e.~the above composition is independent of the choice of representatives, and moreover associative and unital.

To this end, we let $(\psi,f,\phi)$ and $(\psi',f',\phi')$ represent the same morphism $X_\bullet\to Y_\bullet$, and we let $(\rho,g,\theta)$ and $(\rho',g',\theta')$ represent the same morphism $Y_\bullet\to Z_\bullet$. We have to show that
\begin{equation*}
(\rho, g\circ [\theta,\psi] \circ f ,\phi)\sim (\rho', g'\circ [\theta',\psi'] \circ f',\phi'),
\end{equation*}
i.e.~$g'\circ [\theta',\psi'] \circ f' =[\rho',\rho]\circ g\circ [\theta,\psi] \circ f\circ[\phi,\phi']$. But indeed,
\begin{equation*}
f\circ[\phi,\phi']=[\psi,\psi']\circ[\psi',\psi]\circ f\circ[\phi,\phi']=
[\psi,\psi']\circ f'
\end{equation*}
(where we have used Lemma~\ref{lemma:generalized-structure-maps} and that $(\psi,f,\phi)\sim(\psi',f',\phi')$, respectively), and similarly
\begin{equation*}
[\rho',\rho]\circ g = g'\circ [\theta',\theta].
\end{equation*}
Plugging this in yields
\begin{equation*}
[\rho',\rho]\circ g\circ [\theta,\psi] \circ f\circ[\phi,\phi']= g'\circ[\theta',\theta]\circ [\theta,\psi] \circ [\psi,\psi']\circ f'= g'\circ[\theta',\psi']\circ f'
\end{equation*}
as desired, where the final equation uses Lemma~\ref{lemma:generalized-structure-maps} again. This completes the proof that the composition is independent of the choice of representatives.

Next, we claim that for any $(X_1,\dots,X_m)$ the element $[\phi,\id,\phi]$ is a unit of $X_\bullet$ in $\Sigma(\mathcal C)$, where $\phi\colon\bm{m}\times\omega\rightarrowtail\omega$ is any injection. But indeed, if $(Y_1,\dots,Y_n)$ is any other object, then Lemma~\ref{lemma:free-to-choose-phi-psi} implies that any element of $\Hom_{\Sigma(\mathcal C)}(X_\bullet,Y_\bullet)$ can be written as $[\psi,f,\phi]$ for some $\psi\colon\bm{n}\times\omega\rightarrowtail\omega$ and $f\colon \phi_*(X_1,\dots,X_m)\to\psi_*(Y_1,\dots,Y_n)$. But then obviously, $[\psi,f,\phi][\phi,\id,\phi]=[\psi,f\circ [\phi,\phi]\circ\id,\phi]=[\psi,f,\phi]$ as desired, i.e.~$[\phi,\id,\phi]$ is a right unit. Analogously one proves that it is also a left unit. Moreover, a similar argument shows that the composition is associative, finishing the proof that $\Sigma(\mathcal C)$ is a well-defined category.

Now let $F\colon\mathcal C\to\mathcal D$ be any map of parsummable categories. We claim that $\Sigma(F)$ is a well-defined functor, i.e.~its action on morphisms is independent of the choice of representatives and compatible with compositions and the unit. But indeed, if we again let $(\psi,f,\phi),(\psi',f',\phi')$ be representatives of the same morphism $X_\bullet\to Y_\bullet$, then $f'=[\psi',\psi]\circ f \circ [\phi,\phi']$ and hence
\begin{equation*}
F(f')=F([\psi',\psi]_{Y_\bullet}\circ f \circ [\phi,\phi']_{X_\bullet})=[\psi',\psi]_{\Sigma(F)(Y_\bullet)}\circ (Ff)\circ [\phi,\phi']_{\Sigma(F)(X_\bullet)}
\end{equation*}
by universality of the natural transformations $[\psi',\psi]$ and $[\phi,\phi']$ (see Remark~\ref{rk:universally-natural-trafo}). Plugging this in, we immediately see that $[\psi, Ff,\phi]=[\psi',Ff',\phi']$ as desired, i.e.~$\Sigma(F)$ is independent of the choices of representatives. With this established, we consider any further morphism $[\rho,g,\theta]\colon Y_\bullet\to Z_\bullet$; by Lemma~\ref{lemma:free-to-choose-phi-psi} we may again assume that $\theta=\psi$, in which case
\begin{align*}
(\Sigma(F)[\rho,g,\psi])(\Sigma(F)[\psi,f,\phi])&=[\rho, Fg,\psi][\psi, Ff,\phi]=[\rho, (Fg)(Ff),\phi]\\
&=[\rho, F(gf),\phi]=\Sigma(F)[\rho,gf,\phi]\\
&=\Sigma(F)([\rho,g,\psi][\psi,f,\phi])
\end{align*}
as desired. Finally, $\Sigma(F)[\phi,\id,\phi]=[\phi, F(\id),\phi]=[\phi,\id,\phi]$ by definition, completing the argument that $\Sigma(F)$ is a functor.

It only remains to show that $\Sigma$ itself is a functor, which is immediate from the definitions.
\end{proof}
\end{prop}

\begin{rk}\label{rk:middle-term-identity}
Let $X_\bullet=(X_1,\dots,X_m),Y_\bullet=(Y_1,\dots,Y_n)\in\Sigma(\mathcal C)$, let $\phi\colon\bm m\times\omega\rightarrowtail\omega,\psi\colon\bm n\times\omega\rightarrowtail\omega$ be any injections, and let $f\colon\phi_*(X_\bullet)\to\psi_*(Y_\bullet)$ be an isomorphism in $\mathcal C$. We claim that $[\psi,f,\phi]\colon X_\bullet\to Y_\bullet$ is an isomorphism in $\Sigma(\mathcal C)$. Indeed, $[\phi,f^{-1},\psi]$ is a morphism $Y_\bullet\to X_\bullet$, and $[\phi,f^{-1},\psi][\psi,f,\phi]=[\phi,f\circ f^{-1},\phi]=[\phi,\id,\phi]$ is the identity of $X_\bullet$, while $[\psi,f,\phi][\phi,f^{-1},\psi]$ is the identity of $Y_\bullet$ by an analogous computation.
\end{rk}

\begin{constr}
Let $\mathcal C$ be a parsummable category. We define $\blank\otimes\blank\colon \Sigma(\mathcal C)\times\Sigma(\mathcal C)\to\Sigma(\mathcal C)$ as follows: if $X_\bullet=(X_1,\dots,X_m)$, $Y_\bullet=(Y_1,\dots,Y_n)$ are objects of $\Sigma(\mathcal C)$, then $X_\bullet\otimes Y_\bullet\mathrel{:=}(X_1,\dots,X_m,Y_1,\dots,Y_n)$. If
\begin{equation*}
\alpha\colon (X_1,\dots,X_m)\to (Y_1,\dots,Y_n) \quad\text{and}\quad \beta\colon (X'_1,\dots,X'_{m'})\to (Y'_1,\dots,Y'_{n'})
\end{equation*}
are morphisms in $\Sigma(\mathcal C)$, then we define $\alpha\otimes\beta\colon X_\bullet\otimes X'_\bullet\to Y_\bullet\otimes Y'_\bullet$ as follows: we pick representatives $(\psi,f,\phi)$ of $\alpha$ and $(\rho,g,\theta)$ of $\beta$ such that $\im(\phi)\cap\im(\theta)=\varnothing$ and $\im(\psi)\cap\im(\rho)=\varnothing$, and we set
\begin{equation*}
[\psi,f,\phi]\otimes[\rho,g,\theta]\mathrel{:=}[\psi+\rho,f+g,\phi+\theta].
\end{equation*}
Here $f+g$ is the sum in the parsummable category, $\phi+\theta$ is the injection $\bm{(m+m')}\times\omega\to\omega$ with
\begin{equation*}
(\phi+\theta)(i,x)=\begin{cases}
\phi(i,x) & \text{if }i\le m\\
\theta(i-m,x) & \text{otherwise},
\end{cases}
\end{equation*}
and $\psi+\rho$ is defined analogously.

Finally, if $(X_1,\dots,X_m), (Y_1,\dots,Y_n)\in\Sigma(\mathcal C)$, then we define $\tau_{X_\bullet,Y_\bullet}\colon X_\bullet\otimes Y_\bullet \to Y_\bullet \otimes X_\bullet$ as $[\bar\phi,\id,\phi]$, where $\phi$ is any injection $\bm{(m+n)}\times\omega\rightarrowtail
\omega$ and $\bar\phi\colon\bm{(n+m)}\times\omega\rightarrowtail\omega$ is defined via
\begin{equation*}
\bar\phi(i,x)=\begin{cases}
\phi(i+m,x) & \text{if $i\le n$}\\
\phi(i-n,x) & \text{otherwise}.
\end{cases}
\end{equation*}
\end{constr}

\begin{prop}
The above is well-defined and makes $\Sigma(\mathcal C)$ into a permutative category. If $F\colon\mathcal C\to\mathcal D$ is a morphism of parsummable categories, then $\Sigma(F)$ is strict symmetric monoidal with respect to the above permutative structures. This way, $\Sigma$ lifts to a functor $\cat{ParSumCat}\to\cat{PermCat}$.
\begin{proof}
Let us first show that $\otimes$ is well-defined on morphisms. To this end we first observe that $[\psi+\rho,f+g,\phi+\theta]$ indeed defines a map $(X_1,\dots, X_m,X'_1,\dots,X'_{m'})\to (Y_1,\dots, Y_n,Y'_1,\dots,Y'_{n'})$, i.e.~$f+g$ is a map $(\phi+\theta)_*(X_1,\dots, X_m,X'_1,\dots,X'_{m'})\to(\psi+\rho)_* (Y_1,\dots, Y_n,Y'_1,\dots,Y'_{n'})$: namely, by Lemma~\ref{lemma:generalized-action}
\begin{equation*}
(\phi+\theta)_*(X_1,\dots, X_m,X'_1,\dots,X'_{m'})=\phi_*(X_1,\dots,X_m)+\theta_*(X'_1,\dots,X'_{m'})
\end{equation*}
and
\begin{equation*}
(\psi+\rho)_*(Y_1,\dots,Y_n,Y'_1,\dots,Y'_{n'})=\psi_*(Y_1,\dots,Y_n)+\rho_*(Y'_1,\dots, Y'_{n'}).
\end{equation*}
Now let us prove that it is independent of the choices of representatives: to this end, we let $(\psi',f',\phi')$ be another representative of $[\psi,f,\phi]$ and we let $(\rho',g',\theta')$ be another representative of $[\rho,g,\theta]$ such that $\im(\phi')\cap\im(\theta')=\varnothing$ and $\im(\psi')\cap\im(\rho')=\varnothing$. We have to show that $(\psi+\rho,f+g,\phi+\theta)\sim(\psi'+\rho',f'+g',\phi'+\theta')$, i.e.~$f'+g'= [\psi'+\rho',\psi+\rho] \circ(f+g)\circ[\phi+\theta,\phi'+\theta']$. To this end, we observe that by Lemma~\ref{lemma:generalized-structure-maps}
\begin{equation*}
[\phi+\theta,\phi'+\theta']_{(X_1,\dots,X_m,X'_1,\dots,X'_{m'})}=[\phi,\phi']_{(X_1,\dots,X_m)}+[\theta,\theta']_{(X'_1,\dots,X'_{m'})}
\end{equation*}
and $[\psi'+\rho',\psi+\rho]=[\psi',\psi]+[\rho',\rho]$. Thus,
\begin{equation*}
[\psi'+\rho',\psi+\rho] \circ(f+g)\circ[\phi+\theta,\phi'+\theta']=([\psi',\psi]\circ f\circ[\phi,\phi'])+([\rho',\rho]\circ g\circ[\theta,\theta'])=f' + g',
\end{equation*}
where we have used that $+$ is a functor and that $(\psi,f,\phi)\sim(\psi',f',\phi')$ and $(\rho,g,\theta)\sim(\rho',g',\theta')$. This completes the proof that $\otimes$ is well-defined. With this established, functoriality is once again obvious.

Next, let us show that $\otimes$ is strictly unital with unit the empty sequence $\epsilon$. This is obvious on objects. On morphisms, we will only show that it is a left unit, the other case being analogous. To this end, we let $[\psi,f,\phi]$ be any morphism $X_\bullet\to Y_\bullet$. Then the identity of $\epsilon$ is given by $[\varnothing,\id_0,\varnothing]$, where we write $\varnothing$ for the unique injection $\bm{0}\times\omega=\varnothing\rightarrowtail\omega$. Obviously $\im(\varnothing)\cap\im(\phi)=\varnothing=\im(\varnothing)\cap\im(\psi)$, and hence $[\varnothing,\id_0,\varnothing]\otimes[\psi,f,\phi]=[\varnothing+\psi,\id_0+f,\varnothing+\phi]$. But by definition $\varnothing+\psi=\psi,\varnothing+\phi=\phi$, and moreover $\id_0+f=f$ since $0$ is a strict unit for the addition in $\mathcal C$, so that $[\varnothing,\id_0,\varnothing]\otimes[\psi,f,\phi]=[\psi,f,\phi]$ as desired.

The associativity of $\otimes$ is again obvious on objects. To see that
\begin{equation}\label{eq:associativity-tensor}
[\psi,f,\phi]\otimes([\rho,g,\theta]\otimes[\zeta,h,\eta])=([\psi,f,\phi]\otimes[\rho,g,\theta])\otimes[\zeta,h,\eta]
\end{equation}
we may assume without loss of generality that $\im(\phi)$,$\im(\theta)$, and $\im(\eta)$ are pairwise disjoint, and that $\im(\psi)$, $\im(\rho)$, and $\im(\zeta)$ are pairwise disjoint. In this case one immediately checks from the definition that the left hand side of $(\ref{eq:associativity-tensor})$ is given by $[\psi+(\rho+\zeta),f+(g+h), \phi+(\theta+\eta)]$ while the right hand side is given by $[(\psi+\rho)+\zeta,(f+g)+h,(\phi+\theta)+\eta]$. But $\phi+(\theta+\eta)=(\phi+\theta)+\eta$ and $\psi+(\rho+\zeta)=(\psi+\rho)+\zeta$ by direct inspection, and moreover $f+(g+h)=(f+g)+h$ by strict associativity of the sum in $\mathcal C$. This completes the proof that $\otimes$ is associative.

Now we will show that $\tau$ is well-defined. To this end, we first observe that $[\bar\phi,\id,\phi]$ indeed defines a map $(X_1,\dots,X_m,Y_1,\dots,Y_n)\to (Y_1,\dots,Y_n,X_1,\dots,X_m)$ as
\begin{align*}
&\phi_*(X_1,\dots,X_m,Y_1,\dots,Y_n)\\
\quad&=\phi(1,\blank)_*(X_1)+\cdots+\phi(m,\blank)_*(X_m)+\phi(m+1,\blank)_*(Y_1)+\cdots+\phi(m+n,\blank)_*(Y_n)\\
&=\bar\phi(n+1,\blank)_*(X_1)+\cdots+\bar\phi(n+m,\blank)_*(X_m)+\bar\phi(1,\blank)_*(Y_1)+\cdots+\bar\phi(n,\blank)_*(Y_n)\\
&=\bar\phi(1,\blank)_*(Y_1)+\cdots+\bar\phi(n,\blank)_*(Y_n)+\bar\phi(n+1,\blank)_*(X_1)+\cdots+\bar\phi(n+m,\blank)_*(X_m)\\
&=\bar\phi_*(Y_1,\dots,Y_n,X_1,\dots,X_m).
\end{align*}
To see that $\tau_{X_\bullet,Y_\bullet}$ is independent of the choice of $\phi$, we let $\phi'\colon\bm{(m+n)}\times\omega\rightarrowtail\omega$ be another injection. An analogous calculation to the above then shows that
\begin{equation*}
[\phi',\phi]_{(X_1,\dots,X_m,Y_1,\dots,Y_n)}=[\bar\phi',\bar\phi]_{(Y_1,\dots,Y_n,X_1,\dots,X_m)}
\end{equation*}
and hence $[\bar\phi',\id,\phi']=[\bar\phi, [\bar\phi,\bar\phi']\circ[\phi',\phi] ,\phi]=[\bar\phi,\id,\phi]$ as desired.

Next, let us check that $\tau$ is natural. To this end, we let $[\psi,f,\phi]$ and $[\rho,g,\theta]$ represent morphisms $(X_1,\dots,X_m)\to (Y_1,\dots,Y_n)$ and $(X'_1,\dots,X'_{m'})\to (Y'_1,\dots,Y'_{n'})$, respectively, where $\im(\phi)\cap\im(\theta)=\varnothing$ and $\im(\psi)\cap\im(\rho)=\varnothing$. We may then choose the injection $\bm{(m+m')}\times\omega\rightarrowtail\omega$ from the definition of $\tau_{X_\bullet,X'_\bullet}$ to be $\phi+\theta$, and thus one immediately checks that $\tau_{X_\bullet,X'_\bullet}=[\theta+\phi,\id,\phi+\theta]$; analogously, we get $\tau_{Y_\bullet,Y'_\bullet}=[\rho+\psi,\id,\psi+\rho]$. The commutativity of the naturality square
\begin{equation*}
\begin{tikzcd}
X_\bullet \otimes X'_\bullet \arrow[r, "\tau"]\arrow[d, "{[\psi,f,\phi]\otimes[\rho,g,\theta]}"'] & X'_\bullet\otimes X_\bullet\arrow[d, "{[\rho,g,\theta]\otimes[\psi,f,\phi]}"]\\
X'_\bullet\otimes X_\bullet\arrow[r, "\tau"'] & X_\bullet \otimes X'_\bullet
\end{tikzcd}
\end{equation*}
is therefore equivalent to the assertion $[\rho+\psi,g+f,\theta+\phi][\theta+\phi,\id,\phi+\theta]=[\rho+\psi,\id,\psi+\rho][\psi+\rho,f+g,\phi+\theta]$, which itself is immediate from the definition of the composition together with the strict commutativity of the sum in $\mathcal C$. This completes the argument that $\tau$ is a natural transformation.

To finish the proof that this makes $\Sigma(\mathcal C)$ into a permutative category, it remains to check the usual coherence conditions on $\tau$:

\emph{$\tau$ is self-inverse:} let $(X_1,\dots,X_m),(Y_1,\dots,Y_n)$ be objects of $\Sigma(\mathcal C)$ and let $\phi\colon\bm{(m+n)}\times\omega\to\omega$ be any injection, so that $\tau_{X_\bullet, Y_\bullet}=[\bar\phi,\id,\phi]$. But $\bar\phi$ is an injection $\bm{(n+m)}\times\omega\to\omega$ and $\bar{\bar\phi}=\phi$, so that $\tau_{Y_\bullet,X_\bullet}=[\phi,\id,\bar\phi]$ and hence indeed $\tau_{Y_\bullet,X_\bullet}\tau_{X_\bullet,Y_\bullet}=[\phi,\id,\phi]$ which is the identity of $X_\bullet\otimes Y_\bullet$.

\emph{$\tau$ is unital:} let $(X_1,\dots,X_m)\in\Sigma(\mathcal C)$. We have to show that $\tau_{\epsilon,X_\bullet}=\id$. But indeed, if $\phi\colon\bm{(0+m)}\times\omega\to\omega$ is any injection, then $\bar\phi=\phi$ and hence $\tau_{\epsilon,X_\bullet}=[\phi,\id,\phi]$ is the identity. Analogously, one shows that also $\tau_{X_\bullet,\epsilon}=\id$.

\emph{$\tau$ is associative:} let $(X_1,\dots, X_m), (Y_1,\dots,Y_n), (Z_1,\dots, Z_o)\in\Sigma(\mathcal C)$. We have to show that the diagram
\begin{equation}\label{diag:tau-associative}
\begin{tikzcd}[column sep=small]
X_\bullet\otimes Y_\bullet\otimes Z_\bullet\arrow[rd, bend right=10pt, "\id_{X_\bullet}\otimes\tau_{Y_\bullet,Z_\bullet}"']\arrow[rr, "\tau_{X_\bullet\otimes Y_\bullet, Z_\bullet}"] && Z_\bullet\otimes X_\bullet\otimes Y_\bullet\\
& X_\bullet\otimes Z_\bullet\otimes Y_\bullet\arrow[ru, bend right=10pt, "\tau_{X_\bullet,Z_\bullet}\otimes \id_{Y_\bullet}"']
\end{tikzcd}
\end{equation}
commutes. To this end, we pick an injection $\phi\colon\bm{(m+n+o)}\times\omega\rightarrowtail\omega$. Then
\begin{align*}
\tau_{X_\bullet\otimes Y_\bullet, Z_\bullet}&= [\phi|_{\{m+n+1,\dots,m+n+o\}}+\phi|_{\{1,\dots,m+n\}},\id,\phi]\\
\tau_{Y_\bullet, Z_\bullet} &= [\phi|_{\{m+n+1,\dots,m+n+o\}}+\phi|_{\{m+1,\dots,m+n\}},\id,\phi|_{\{m+1,\dots,m+n+o\}}]\\
\tau_{X_\bullet, Z_\bullet} &=
[\phi|_{\{m+n+1,\dots,m+n+o\}}+\phi|_{\{1,\dots,m\}},\id, \phi|_{\{1,\dots,m\}}+\phi|_{\{m+n+1,\dots,m+n+o\}}]\\
\id_{X_\bullet}&=[\phi|_{\{1,\dots,m\}},\id,\phi|_{\{1,\dots,m\}}]\\
\id_{Y_\bullet}&=[\phi|_{\{m+1,\dots,m+n\}},\id, \phi|_{\{m+1,\dots,m+n\}}]
\end{align*}
(by slight abuse of notation, we write $\phi|_{\{m+n+1,\dots,m+n+o\}}$ for the injection $\bm o\times\omega\rightarrowtail\omega$ sending $(i,x)$ to $\phi(m+n+i,x)$ for all $1\le i\le o$, $x\in\omega$, and similarly for the other cases).
Plugging this in, the commutativity of $(\ref{diag:tau-associative})$ is equivalent to the assertion
\begin{align*}
&[\phi|_{\{m+n+1,\dots,m+n+o\}}+\phi|_{\{1,\dots,m+n\}},\id,\phi]\\
&\quad=[\phi|_{\{m+n+1,\dots,m+n+o\}}+\phi|_{\{1,\dots,m\}}+\phi|_{\{m+1,\dots,m+n\}},(\id+\id)(\id+\id),\\
&\quad\phantom{{}=[}\phi|_{\{1,\dots,m\}}+\phi|_{\{m+1,\dots,m+n+o\}}]
\end{align*}
which in turn follows immediately from the fact that the sum in $\mathcal C$ is functorial together with the evident identities $\phi|_{\{m+n+1,\dots,m+n+o\}}+\phi|_{\{1,\dots,m\}}+\phi|_{\{m+1,\dots,m+n\}}=\phi|_{\{m+n+1,\dots,m+n+o\}}+\phi|_{\{1,\dots,m+n\}}$ and $\phi|_{\{1,\dots,m\}}+\phi|_{\{m+1,\dots,m+n+o\}}=\phi$.
Altogether, we have shown that $\Sigma(\mathcal C)$ is indeed a well-defined permutative category.

Finally, we have to show that $\Sigma(F)$ is a strict symmetric monoidal functor, i.e. it strictly commutes with $\otimes$ and with $\tau$, and it strictly preserves the tensor unit. It is obvious from the definition that $\Sigma(F)$ commutes with $\otimes$ on objects and that $\Sigma(F)(\epsilon)=\epsilon$. If now $[\psi,f,\phi]\colon X_\bullet\to Y_\bullet$ and $[\psi',f',\phi']\colon X'_\bullet\to Y'_\bullet$ are any two morphisms, then we have to prove that
\begin{equation*}
\Sigma(F)([\psi,f,\phi]\otimes[\psi',f',\phi'])=(\Sigma(F)[\psi,f,\phi])\otimes(\Sigma(F)[\psi',f',\phi']).
\end{equation*}
For this we may assume without loss of generality that $\im(\phi)\cap\im(\phi')=\varnothing$ and $\im(\psi)\cap\im(\psi')=\varnothing$. In this case, the left hand side evaluates to $\Sigma(F)[\psi+\psi',f+g,\phi+\phi']=[\psi+\psi', F(f+g),\phi+\phi']$ whereas the right hand side evaluates to $[\psi+\psi', Ff+Fg,\phi+\phi']$. These are in fact equal as $F$ strictly preserves the sum in $\mathcal C$ by assumption.

In order to check the compatibility of $\Sigma(F)$ with $\tau$ we have to show that for any $(X_1,\dots,X_m), (Y_1,\dots,Y_n)\in\Sigma(\mathcal C)$ the equality $\Sigma(F)(\tau_{X_\bullet,Y_\bullet})=\tau_{\Sigma(F)(X_\bullet), \Sigma(F)(Y_\bullet)}$ holds in $\Sigma(\mathcal D)$. But if $\phi\colon\bm{(m+n)}\times\omega\to\omega$ is any injection, then the left hand side is given by $\Sigma(F)[\bar\phi,\id,\phi]=[\bar\phi,F(\id),\phi]$ while the right hand side is given by $[\bar\phi,\id,\phi]$. Thus, this is simply an instance of functoriality of $F$. This completes the argument that $\Sigma(F)$ is strict symmetric monoidal. Since the forgetful functor from permutative categories and \emph{strict} symmetric monoidal functors to ordinary categories is faithful, $\Sigma$ remains a functor when viewed as mapping into $\cat{PermCat}$. This completes the proof of the proposition.
\end{proof}
\end{prop}

Next, we will give a description of certain coherence isomorphisms in $\Sigma(\mathcal C)$, that will become useful at several points later:

\begin{lemma}\label{lemma:sigma-coherence}
Let $\mathcal C$ be a parsummable category, let $X_1,\dots,X_m\in\mathcal C$ and let $\sigma\in\Sigma_m$. Then the coherence isomorphism
\begin{equation*}
(X_1,\dots,X_m)=\bigotimes_{i=1}^m (X_i)\to\bigotimes_{i=1}^m (X_{\sigma^{-1}(i)})= (X_{\sigma^{-1}(1)},\dots,X_{\sigma^{-1}(m)})
\end{equation*}
associated to the permutation $\sigma$ and the objects $(X_1),\dots,(X_m)$ of $\Sigma(\mathcal C)$ is given by $[\phi,\id,\phi\circ(\sigma\times\id)]$ where $\phi\colon\bm{m}\times\omega\rightarrowtail\omega$ is any injection.
\begin{proof}
We first observe that by Lemma~\ref{lemma:generalized-action}
\begin{equation*}
\big(\phi\circ(\sigma\times\id)\big)_*(X_1,\dots,X_m)
= \phi_*(X_{\sigma^{-1}(1)},\dots,X_{\sigma^{-1}(m)}),
\end{equation*}
so $[\phi,\id,\phi\circ(\sigma\times\id)]$ indeed defines a map $(X_1,\dots,X_m)\to (X_{\sigma^{-1}(1)},\dots, X_{\sigma^{-1}(m)})$.

Next, we will prove the claim in the special case that $\sigma=(k\;k+1)$ for some $1\le k<m$. Then the coherence isomorphism associated to $\sigma$ is by definition (see Theorem~\ref{thm:permutative-operadic-description})
\begin{equation}\label{eq:coherence-transposition}
\id_{(X_1,\dots,X_{k-1})}\otimes \tau_{(X_k),(X_{k+1})}\otimes \id_{(X_{k+2},\dots,X_m)}
\end{equation}
Now $\phi|_{\{k,k+1\}}$ is an injection $\bm{(1+
1)}\times\omega\rightarrowtail\omega$, thus
\begin{equation*}
\tau_{(X_k),(X_{k+1})}=[\overline{\phi|_{\{k,k+1\}}},\id,\phi|_{\{k,k+1\}}]=[\phi|_{\{k,k+1\}}\circ\big((1\;2)\times\id\big),\id,\phi|_{\{k,k+1\}}].
\end{equation*}
On the other hand
\begin{align*}
\id_{(X_1,\dots,X_{k-1})}&=[\phi|_{\{1,\dots,k-1\}},\id,\phi|_{\{1,\dots,k-1\}}]\\
\id_{(X_{k+2},\dots,X_{m})}&=[\phi|_{\{k+2,\dots,m\}},\id,\phi|_{\{k+2,\dots,m\}}],
\end{align*}
and hence $(\ref{eq:coherence-transposition})$ agrees with
\begin{align*}
&[\phi|_{\{1,\dots,k-1\}}+\big(\phi|_{\{k,k+1\}}\circ\big((1\;2)\times\id\big)\big)+\phi|_{\{k+2,\dots,m\}},\id,\\
&\qquad\qquad\phi|_{\{1,\dots,k-1\}}+\phi|_{\{k,k+1\}}+\phi|_{\{k+2,\dots,m\}}]\\
&\quad=[\phi\circ\big((k\;k+1)\times\id\big),\id,\phi]
\end{align*}
by definition of the monoidal structure, where we have used that the images of the restrictions of $\phi$ are suitably disjoint. The claim now follows by replacing $\phi$ with $\phi\circ\big((k\;k+1)\times\id\big)$.

If $\sigma$ is now a general permutation, then we write $\sigma=\pi_r\cdots\pi_1$ such that each $\pi_i$ is a transposition exchanging two adjacent elements of $\bm m$, and we define $\sigma_i\mathrel{:=}\pi_i\cdots\pi_1$. Then the coherence isomorphism $\bigotimes_{i=1}^m(X_i)\to\bigotimes_{i=1}^m(X_{\sigma^{-1}(i)})$ associated to $\sigma$ is (by construction) the composition
\begin{equation*}
\bigotimes_{i=1}^m (X_i)\to
\bigotimes_{i=1}^m (X_{\sigma_1^{-1}(i)})\to
\bigotimes_{i=1}^m (X_{\sigma_2^{-1}(i)})\to
\cdots\to
\bigotimes_{i=1}^m (X_{\sigma_r^{-1}(i)})=
\bigotimes_{i=1}^m (X_{\sigma^{-1}(i)})
\end{equation*}
where the $i$-th arrow is the coherence isomorphism associated to the permutation $\pi_i$. Applying the above special case for the injection $\phi\circ\big((\pi_r\cdots\pi_{i+1})\times\id\big)$ we therefore see that the $i$-th arrow is given by $[\phi\circ\big((\pi_r\cdots\pi_{i+1})\times\id\big),\id,\phi\circ\big((\pi_r\cdots\pi_i)\times\id\big)]$. The composition of these is just $[\phi,\id,\phi\circ\big((\pi_r\cdots\pi_1)\times\id\big)]=[\phi,\id,\phi\circ(\sigma\times\id)]$, which completes the proof of the lemma.
\end{proof}
\end{lemma}

As sketched above, the motivation behind $\Sigma$ was to give a more rigid version of Schwede's construction \cite[Construction~5.6]{schwede-k-theory} associating to any fixed injection $\mu\colon\bm2\times\omega\rightarrowtail\omega$ and any parsummable category $\mathcal C$ a symmetric monoidal category $\mu^*\mathcal C$ with the same underlying category and with tensor product given by $\mu_*\colon\mathcal C^{\times2}\to\mathcal C$. We now want make the relationship between the two constructions precise, for which we first have to make the remaining structure of $\mu^*\mathcal C$ explicit:

\begin{constr}\label{constr:comparison-schwede}
The tensor unit $\textbf1$ of $\mu^*\mathcal C$ is the additive unit $0$ of $\mathcal C$.

For any $X\in\mathcal C$, the left unitality isomorphism $\bm1\otimes X=\mu_*(0,X)=\mu(2,\blank)_*(X)\to X$ is the structure isomorphism $[1,\mu(2,\blank)]_X$ of the $E\mathcal M$-action; analogously, the right unitality isomorphism is given by $[1,\mu(1,\blank)]_X$.

If $Y\in\mathcal C$ is another object, then the symmetry isomorphism $\tau_{X,Y}\colon X\otimes Y=\mu_*(X,Y)\to \mu_*(Y,X)=Y\otimes X$ is given by $[\mu\circ(\pi\times\id),\mu]_{(X,Y)}$, where $\pi\in\Sigma_2$ is the unique non-trivial permutation.

Finally, we define for any $X,Y,Z\in\mathcal C$ the associativity isomorphism $(X\otimes Y)\otimes Z=\mu_*(\mu_*(X,Y),Z)\to\mu_*(X,\mu_*(Y,Z))=X\otimes(Y\otimes Z)$ as $[\mu\circ(1 \amalg\mu),\mu\circ(\mu\amalg 1)]_{(X,Y,Z)}$; here we write (by slight abuse of notation) $1\amalg\mu\colon\bm3\times\omega\to\bm2\times\omega$ for the injection
\begin{equation*}
\bm3\times\omega\cong\omega\amalg(\bm2\times\omega)\xrightarrow{1\amalg\mu}\omega\amalg\omega\cong\bm2\times\omega
\end{equation*}
obtained by conjugating the usual coproduct in $\cat{Set}$ with the canonical isomorphisms, i.e.~$(1\amalg\mu)(1,x)=(1,x)$ and $(1\amalg\mu)(i,x)=(2,\mu(i-1,x))$ for $i\ge2$; the injection $\mu\amalg1$ is defined analogously.
\end{constr}

The above is indeed a symmetric monoidal structure by \cite[Proposition~5.7]{schwede-k-theory}, and if $F\colon\mathcal C\to\mathcal D$ is any morphism of parsummable categories, then it is even strict symmetric monoidal when viewed as a functor $\mu^*\mathcal C\to\mu^*\mathcal D$ by \cite[Remark~5.8]{schwede-k-theory}; in particular, we get a functor $\mu^*\colon\cat{ParSumCat}\to\cat{SymMonCat}$.

One can actually show that the identity on objects extends to a natural strict symmetric monoidal isomorphism $J(\mu^*\mathcal C)\to\Sigma(\mathcal C)$ for any parsummable category $\mathcal C$. Here we will instead only prove the following slightly weaker statement that avoids talking about the concrete construction of $J$:

\begin{prop}\label{prop:comparison-schwede}
There exists a preferred and natural strong symmetric monoidal equivalence $\mu^*\mathcal C\to\Sigma(\mathcal C)$ for any parsummable category $\mathcal C$.
\begin{proof}
We define the functor $\widetilde\nu\colon\mu^*\mathcal C\to\Sigma(\mathcal C)$ as follows: an object $X\in\mathcal C$ is sent to the $1$-tuple $(X)$, and a morphism $f\colon X\to Y$ is sent to $[1,f,1]\colon(X)\to(Y)$. The unit isomorphism $\iota\colon\epsilon\to(0)$ is given by $[1,\id_0,\varnothing]$, while the structure isomorphism $\nabla_{X,Y}$ is given by
\begin{equation*}
\widetilde\nu(X)\otimes \widetilde\nu(Y)=(X,Y)\xrightarrow{[1,\id,\mu]} (\mu_*(X,Y))=\widetilde\nu(X\otimes Y)
\end{equation*}
for all $X,Y\in\mathcal C$. Note that these are indeed isomorphisms by Remark~\ref{rk:middle-term-identity}.

Let us first show that $\widetilde\nu$ is indeed a strong symmetric monoidal functor. The naturality of $\nabla$ amounts to the relation $\nabla_{Y,Y'}\circ(\widetilde\nu(f)\otimes\widetilde\nu(g))=\widetilde\nu(f\otimes g)\circ\nabla_{X,X'}$ for all $f\colon X\to Y,g\colon X'\to Y'$ in $\mathcal C$. To prove this, we first observe that if $[\psi,h,\phi]\colon A_\bullet\to B_\bullet$ is any morphism in $\Sigma(\mathcal C)$ and $u\in\mathcal M$, then
\begin{align*}
[u\circ\psi,u_*(h),u\circ\phi]&=[\psi,[\psi,u\circ\psi]_{B_\bullet}\circ u_*(h)\circ [u\circ\phi,\phi]_{A_\bullet},\phi]\\
&=[\psi,[1,u]_{\psi_*(B_\bullet)}\circ u_*(h)\circ [u,1]_{\phi_*(A_\bullet)},\phi]
=[\psi,h,\phi]
\end{align*}
by Lemma~\ref{lemma:generalized-structure-maps}-$(\ref{item:gsm-EM-equivariant})$ and naturality of the structure isomorphisms of the $E\mathcal M$-action. In particular, $\widetilde\nu(f)=[1,f,1]=[\mu(1,\blank),\mu(1,\blank)_*(f),\mu(1,\blank)]$ and analogously $\widetilde\nu(g)=[\mu(2,\blank),\mu(2,\blank)_*(g),\mu(2,\blank)]$, so that $\widetilde\nu(f)\otimes\widetilde\nu(g)=[\mu,\mu(1,\blank)_*(f)+\mu(2,\blank)_*(g),\mu]=[\mu,\mu_*(f,g),\mu]$, and hence indeed
\begin{align*}
\nabla_{Y,Y'}\circ(\widetilde\nu(f)\otimes\widetilde\nu(g))&=[1,\id,\mu][\mu,\mu_*(f,g),\mu]=[1,\mu_*(f,g),\mu]\\
&=[1,\mu_*(f,g),1][1,\id,\mu]=\widetilde\nu(f\otimes g)\circ\nabla_{X,X'}.
\end{align*}

As $\Sigma(\mathcal C)$ is permutative, the compatibility of $\widetilde\nu$ with the associativity isomorphisms amounts to saying that
\begin{equation}\label{diag:compatibility-assoc}
\begin{tikzcd}[column sep=1.75in]
(X,Y,Z)\arrow[r, equal]\arrow[d, "\nabla_{X,Y}\otimes\id_Z"'] & (X,Y,Z)\arrow[d, "\id_X\otimes\nabla_{Y,Z}"]\\
(\mu_*(X,Y),Z)\arrow[d, "\nabla_{\mu_*(X,Y),Z}"'] & (X,\mu_*(Y,Z))\arrow[d, "\nabla_{X,\mu_*(Y,Z)}"]\\
\big(\mu_*(\mu_*(X,Y),Z)\big)\arrow[r, "{[1,[\mu\circ(1\amalg\mu),\mu\circ(\mu\amalg1)],1]}"'] & \big(\mu_*(X,\mu_*(Y,Z))\big)
\end{tikzcd}
\end{equation}
commutes for all $X,Y,Z\in\mathcal C$. To this end we note that by the above observation $\nabla_{X,Y}=[u,\id,u\circ\mu]$ for all $u\in\mathcal M$; taking $u=\mu(1,\blank)$, we then see that the top left vertical arrow in $(\ref{diag:compatibility-assoc})$ is given by
\begin{align*}
[\mu(1,\blank),\id,\mu(1,\blank)\circ\mu]\otimes[\mu(2,\blank),\id,\mu(2,\blank)]&=[\mu,\id+\id,\mu(1,\blank)\circ\mu+\mu(2,\blank)]\\
&=[\mu,\id,\mu\circ(\mu\amalg1)],
\end{align*}
so that left hand vertical composite equals $[1,\id,\mu\circ(\mu\amalg1)]$. Analogously, one shows that the right hand vertical composite equals $[1,\id,\mu\circ(1\amalg\mu)]$. The claim therefore amounts to the equality
\begin{equation*}
[1,[\mu\circ(1\amalg\mu),\mu\circ(\mu\amalg1)],\mu\circ(\mu\amalg1)]=[1,\id,\mu\circ(1\amalg\mu)]
\end{equation*}
which is in turn immediate from the construction of the hom-sets of $\Sigma(\mathcal C)$.

For the compatibility with the right unitality isomorphisms we have to show that the composite
\begin{equation*}
(X)=(X)\otimes\epsilon\xrightarrow{\id_{(X)}\otimes\iota}(X,0)\xrightarrow{\nabla_{X,0}}(\mu_*(X,0))\xrightarrow{[1,[1,\mu(1,\blank)],1]}(X)
\end{equation*}
is the identity for all $X\in\mathcal C$. But $\id_{(X)}=[\mu(1,\blank),\id,\mu(1,\blank)]$ whereas $\epsilon=[1,\id,\varnothing]=[\mu(2,\blank),\id,\varnothing]$, so that the above composite is given by
\begin{equation*}
[1,[1,\mu(1,\blank)],1][1,\id,\mu][\mu,\id,\mu(1,\blank)]=[1,[1,\mu(1,\blank)],\mu(1,\blank)]=[1,\id,1]
\end{equation*}
which is indeed the identity as desired. The proof of left unitality is analogous (and in fact it also follows automatically from right unitality together with symmetry).

Finally, the compatibility with symmetry isomorphisms amounts to the relation $\widetilde\nu(\tau_{X,Y})\nabla_{X,Y}=\nabla_{Y,X}\tau_{(X),(Y)}$ for all $X,Y\in\mathcal C$. By Lemma~\ref{lemma:sigma-coherence}, the symmetry isomorphism in $\Sigma(\mathcal C)$ is given by $[\mu,\id,\mu\circ(\pi\times\id)]$, so the right hand side equals $[1,\id,\mu\circ(\pi\times\id)]$. On the other hand, the left hand side is given by $[1,[\mu\circ(\pi\times\id),\mu],1][1,\id,\mu]=[1,[\mu\circ(\pi\times\id),\mu],\mu]$ and these two again agree by definition of $\Sigma(\mathcal C)$. This completes the proof that $\widetilde\nu$ is strong symmetric monoidal.

It is trivial to check that $\widetilde\nu$ is natural. To finish the proof it therefore only remains to show that $\widetilde\nu\colon\mu^*\mathcal C\to\Sigma(\mathcal C)$ is an equivalence of categories. Indeed, Lemma~\ref{lemma:free-to-choose-phi-psi} immediately implies that $\widetilde\nu$ is fully faithful. For essential surjectivity we let $(X_1,\dots,X_m)\in\Sigma(\mathcal C)$ be arbitrary, and we pick an injection $\phi\colon\bm m\times\omega\rightarrowtail\omega$. Then $[\phi,\id_{\phi_*(X_\bullet)},1]\colon\widetilde\nu(\phi_*(X_\bullet))\to X_\bullet$ is an isomorphism by Remark~\ref{rk:middle-term-identity}, which completes the proof of the proposition.
\end{proof}
\end{prop}

As the final result of this section, we will construct a natural levelwise underlying equivalence $\Sigma\Phi\Rightarrow\id$. Let us first describe the basic idea of this construction: an object of $\Sigma\Phi(\mathscr C)$ is by definition a tuple $(X^{(1)},\dots,X^{(m)})$ of objects of $\Phi(\mathscr C)$, i.e.~sequences of objects of $\mathscr C$ such that almost all entries are given by the tensor unit (here and in what follows it will be convenient to denote elements of $\Phi(\mathscr C)$ simply by `$X$' instead of the notation `$X_\bullet$' used before as this avoids confusing expressions like `$\phi_*(X^\bullet_\bullet)$,' and we will do so for the remainder of this section).

Sending such an object $(X^{(1)},\dots,X^{(m)})$ to $\bigotimes_{i=1}^m\bigotimes_{j\in\omega} X^{(i)}_j\in\mathscr C$ is strictly compatible with the tensor products on $\Sigma\Phi\mathscr C$ and $\mathscr C$. If $[\psi,f,\phi]\colon (X^{(1)},\dots,X^{(m)})\to (Y^{(1)},\dots,Y^{(n)})$ is a morphism in $\Sigma\Phi\mathscr C$, then the morphism $f\colon\phi_*(X^\bullet)\to\psi_*(Y^\bullet)$ of $\Phi\mathscr C$ is given by a morphism $\bigotimes_{i\in\omega}\phi_*(X^\bullet)_i\to\bigotimes_{i\in\omega}\psi_*(Y^\bullet)_i$ in $\mathscr C$. If $\phi,\psi$ were order preserving with respect to the lexicographical order on $\bm m\times\omega$ and $\bm n\times\omega$, respectively, then $f$ would indeed be a morphism $\bigotimes_{i=1}^m\bigotimes_{j\in\omega} X^{(i)}_j\to\bigotimes_{i=1}^n\bigotimes_{j\in\omega}Y^{(i)}_j$, and we could extend the above by $[\psi,f,\phi]\mapsto f$ to get a functor, but there are of course no such injections $\phi,\psi$. However, using that the tensor product in $\mathscr C$ is strictly unital, a weaker condition turns out to be enough:

\begin{defi}
Let $\mathscr C$ be a permutative category and let $X^{(1)},\dots,X^{(m)}\in\Phi\mathscr C$. Then an injection $\phi\colon\bm{m}\times\omega\rightarrowtail\omega$ is called \emph{$X^\bullet$-monotone}, if $\phi(i,j)<\phi(i',j')$ for all $i,i'\in\bm{m}$, $j,j'\in\omega$ such that $(i,j)$ is lexicographically smaller than $(i',j')$ and such that $X^{(i)}_{j},X^{(i')}_{j'}\not=\bm{1}$.
\end{defi}

\begin{prop}\label{prop:suitably-monotone}
Let $X^{(1)},\dots,X^{(m)}$ be as above.
\begin{enumerate}
\item There exists an $X^\bullet$-monotone injection $\phi\colon\bm{m}\times\omega\rightarrowtail\omega$.\label{item:sm-existence}
\item If $\phi,\phi'\colon\bm{m}\times\omega\rightarrowtail\omega$ are two $X^\bullet$-monotone injections, then the morphism $[\phi',\phi]_{X^\bullet}:\phi_*(X^\bullet)\to\phi'_*(X^\bullet)$ in $\Phi(\mathscr C)$ is given by the identity of
\begin{equation*}
\bigotimes_{i\in\omega}(\phi_*(X^\bullet))_i=\bigotimes_{i\in\omega}(\phi'_*(X^\bullet))_i
\end{equation*}
in $\mathscr C$.\label{item:sm-uniqueness}
\end{enumerate}
\end{prop}

The proof of the proposition relies on the following coherence result:

\begin{lemma}\label{lemma:shuffle-coherence}
Let $k\ge0$, $\sigma\in\Sigma_k$ and assume we are given $1\le i_1<\cdots<i_m\le k$. Then there exists a unique permutation $\widetilde\sigma\in\Sigma_m$ such that
\begin{equation*}
\sigma(i_{\widetilde\sigma^{-1}(1)})<\cdots<\sigma(i_{\widetilde\sigma^{-1}(m)}).
\end{equation*}
Moreover, if $\mathscr C$ is a permutative category and $X_1,\dots,X_k\in\mathscr C$ such that $X_i=\bm1$ for all $i\notin\{i_1,\dots,i_m\}$, then the coherence isomorphism
\begin{equation}\label{eq:big-coherence}
\bigotimes_{i=1}^k X_i\to\bigotimes_{i=1}^k X_{\sigma^{-1}(i)}
\end{equation}
associated to $\sigma$ agrees with the coherence isomorphism
\begin{equation*}
\bigotimes_{j=1}^m X_{i_j}\to \bigotimes_{j=1}^m X_{i_{\widetilde\sigma^{-1}(j)}}
\end{equation*}
associated to $\widetilde\sigma$.
\begin{proof}
It is obvious that there is at most one such $\widetilde\sigma$. Now we define $m_1,\dots,m_k$ via
\begin{equation*}
m_i=\begin{cases}
1 & i \in \{i_1,\dots,i_m\}\\
0 & \textup{otherwise}
\end{cases}
\end{equation*}
(so that in particular $m=m_1+\cdots+m_k$), and we claim that $\hat\sigma\mathrel{:=}\sigma_{(m_1,\dots,m_k)}$, i.e.~the image of $(\sigma,\id_{m_1},\dots,\id_{m_k})$ under the structure map $E\Sigma_k\times E\Sigma_{m_1}\times\cdots\times E\Sigma_{m_k}\to E\Sigma_m$ of the Barratt-Eccles operad, has the desired properties. Indeed, by definition
\begin{equation*}
\hat\sigma(j)=\hat\sigma(m_1+\cdots+m_{\ell-1}+1)=m_{\sigma^{-1}(1)}+\cdots+m_{\sigma^{-1}(\sigma(\ell)-1)}+1
\end{equation*}
for any $1\le j\le m$, where $1\le\ell\le k$ is the unique index with $m_\ell=1$ and $m_1+\cdots+m_{\ell-1}+1=j$. Clearly $\ell=i_j$, so if $1\le j'\le m$ with $\sigma(i_{j'})\ge\sigma(i_j$), then
\begin{equation*}
\hat\sigma(j')-\hat\sigma(j)=m_{\sigma^{-1}(\sigma(i_j))}+m_{\sigma^{-1}(\sigma(i_j)+1)}+\cdots+m_{\sigma^{-1}(\sigma(i_{j'})-1)}\ge0,
\end{equation*}
hence $\hat\sigma(j')\ge\hat\sigma(j)$. By contraposition we conclude that $\sigma(i_{j'})<\sigma(i_{j})$ whenever $\hat\sigma(j')<\hat\sigma(j)$. The claim follows as evidently $\hat\sigma(\hat\sigma^{-1}(1))<\cdots<\hat\sigma(\hat\sigma^{-1}(m))$.

It only remains to verify the claim about the coherence isomorphism, for which it suffices (by the above identification of $\widetilde\sigma$ with $\sigma_{(m_1,\dots,m_k)}$) to chase the morphism $\big((\sigma,\id_k),\big((\id_{{m_1}},\id_{{m_1}}),\dots,(\id_{{m_k}},\id_{{m_k}})\big),(\id_{X_1},\dots,\id_{X_k})\big)$ through the commutative diagram
\begin{equation*}
\begin{tikzcd}[row sep=small]
E\Sigma_k\times\big(\prod_{i=1}^kE\Sigma_{m_i}\big)\times\mathscr C^{\times m}\arrow[dd, "\text{shuffle}"']\arrow[rr, "\text{compose}\times\id"] && E\Sigma_m\times\mathscr C^{\times m}\arrow[rd, "\text{act}", bend left=10pt]\\
&&& \mathscr C\\
E\Sigma_k\times\prod_{i=1}^k(E\Sigma_{m_i}\times\mathscr C^{\times m_i})\arrow[rr, "E\Sigma_k\times\prod\text{act}"'] && E\Sigma_k\times\mathscr C^{\times k}\arrow[ur, "\text{act}"', bend right=10pt]
\end{tikzcd}
\end{equation*}
coming from the usual $E\Sigma_*$-algebra structure on $\mathscr C$.
\end{proof}
\end{lemma}

\begin{cor}\label{cor:comparison-iso-Phi-reexpressed}
Let $\mathscr C$ be a permutative category, let $X\in\Phi\mathscr C$, let $u\in\mathcal M$, and write $\supp X=\{i_1<\cdots<i_m\}$. Then there exists a unique $\widetilde\sigma\in\Sigma_m$ such that $u(i_{\widetilde\sigma^{-1}(1)})<\cdots<u(i_{\widetilde\sigma^{-1}(m)})$. Moreover, as a map in $\mathscr C$, $[u,1]_X$ is the coherence isomorphism associated to $\widetilde\sigma$.
\end{cor}

By definition of the coherence isomorphisms associated to bijections between finite totally ordered sets (see the discussion after Remark~\ref{rk:rk-before-coherence-isomorphisms}), we can also reformulate this as saying that as a morphism in $\mathscr C$, $[u,1]_X$ is the coherence isomorphism $\bigotimes_{i\in\supp(X)} X_i\to\bigotimes_{i\in u(\supp(X))} (u_*X)_i$ associated to the bijection $u\colon \supp(X)\to u(\supp(X))$ of totally ordered sets.

\begin{proof}
We recall the construction of $[u,1]_X$: this was given by taking $K\gg0$ such that $X_i=\bm1=(u_*X)_i$ for all $i> K$ together with a permutation $\sigma\in\Sigma_K$ such that $\sigma(i)=u(i)$ for all $i$ with $X_i\not=\bm1$. With respect to these choices $[u,1]_X$ is the coherence isomorphism
\begin{equation*}
\bigotimes_{i=1}^K X_i\to \bigotimes_{i=1}^K X_{\sigma^{-1}(i)}=\bigotimes_{i=1}^K (u_*X)(i)
\end{equation*}
associated to $\sigma$. Therefore, the claim is an instance of the previous lemma.
\end{proof}

The special case $\widetilde\sigma=\id$ is worth making explicit:

\begin{cor}
Let $\mathscr C$ be a permutative category, let $X\in\Phi\mathscr C$,
and let $u\in\mathcal M$. Assume $u$ is monotone when restricted to $\supp X$. Then $[u,1]_X$ is the identity as a morphism $\bigotimes_{i\in\omega} X_i\to\bigotimes_{i\in\omega} (u_*X)_i$ in $\mathscr C$.\qed
\end{cor}

With this established we are now ready to tackle the proposition:

\begin{proof}[Proof of Proposition~\ref{prop:suitably-monotone}]
For the first statement, we let $(i_1,j_1), (i_2,j_2),\dots, (i_{\ell}, j_{\ell})$ be the sequence of those indices $(i,j)$ such that $X^{(i)}_{j}\not=\bm1$, ordered lexicographically. Then there exists an injection $\phi_0\colon (\bm{m}\times\omega)\setminus \{(i_1,j_1),\dots,(i_\ell,j_\ell)\}\rightarrowtail\omega\setminus\{1,\dots,\ell\}$ as both sides are countably infinite. We can then simply define
\begin{equation*}
\phi(i,j)=\begin{cases}
k & \text{if }(i,j)=(i_k,j_k)\\
\phi_0(i,j) & \text{otherwise}.
\end{cases}
\end{equation*}
For the second statement we choose any $u\in\mathcal M$ with $u(\phi(i_k,j_k))=\phi'(i_k,j_k)$ for all $1\le k\le\ell$. Then $u\phi$ and $\phi'$ agree on $(\{1\}\times\supp(X^{(1)}))\cup\cdots\cup(\{m\}\times\supp(X^{(m)}))=\{(i_1,j_1),\dots,(i_\ell,j_\ell)\}$, hence $[\phi',\phi]_{X^\bullet}=[u\phi,\phi]_{X^\bullet}$, which further agrees with $[u,1]_{\phi_*(X^\bullet)}$ by Lemma~\ref{lemma:generalized-structure-maps}-$(\ref{item:gsm-EM-equivariant})$. We observe that
$\phi_*(X^\bullet)$ is supported on $\{\phi(i_1,j_1),\dots,\phi(i_{\ell},j_{\ell})\}$ and that $\phi(i_1,j_1)<\cdots<\phi(i_{\ell},j_{\ell})$ by monotonicity of $\phi$. Moreover, $\phi'(i_1,j_1)<\cdots<\phi'(i_{\ell},j_{\ell})$ by monotonicity of $\phi'$, i.e.~$u$ is monotone when restricted to $\supp(\phi_*(X^\bullet))$. Thus, the claim follows from the previous corollary.
\end{proof}

Now we can finally turn the above intuition into a rigourous construction of the comparison map $\Sigma\Phi\mathscr C\to\mathscr C$:

\begin{constr}
Let $\mathscr C$ be a small permutative category. We define $T_{\mathscr C}\colon\Sigma\Phi\mathscr C\to\mathscr C$ as follows: an object $(X^{(1)},\dots,X^{(m)})\in\Sigma\Phi\mathscr C$ is sent to
\begin{equation*}
\bigotimes_{i=1}^m\bigotimes_{j\in\omega} X^{(i)}_j;
\end{equation*}
if $\alpha\colon (X^{(1)},\dots,X^{(m)})\to(Y^{(1)},\dots,Y^{(n)})$ is any morphism, then we choose a representative $(\psi,f,\phi)$ of $\alpha$ such that $\psi$ is $Y^\bullet$-monotone and $\phi$ is $X^\bullet$-monotone, and set $T_{\mathscr C}(\alpha)=f$ (where we view $f$ as a morphism in $\mathscr C$).
\end{constr}

\begin{prop}\label{prop:T-equivalence}
Let $\mathscr C$ be any small permutative category. Then $T_{\mathscr C}$ is a well-defined strict symmetric monoidal functor $\Sigma\Phi\mathscr C\to\mathscr C$. For varying $\mathscr C$, the $T$'s assemble into a natural levelwise underlying equivalence $\Sigma\Phi\Rightarrow\id_{\cat{PermCat}}$.
\end{prop}

The proof will again involve a basic coherence statement:

\begin{lemma}
Let $\mathscr C$ be a permutative category, let $1\le \ell\le m$, $X_1,\dots,X_m\in\mathscr C$, and let $\sigma\in\Sigma_m$ be the permutation that moves the first $\ell$ entries to the end, i.e.~
\begin{equation*}
\sigma(i)=\begin{cases}
i + m - \ell & \text{if }i\le\ell\\
i - \ell & \text{otherwise.}
\end{cases}
\end{equation*}
Then the coherence isomorphism
\begin{equation}\label{eq:shuffle}
\bigotimes_{i=1}^k X_i\to\bigotimes_{i=1}^m X_{\sigma^{-1}(i)}
\end{equation}
associated to $\sigma$ agrees with the symmetry isomorphism
\begin{equation}\label{eq:transposition}
\tau_{\bigotimes_{i=1}^\ell X_i,\bigotimes_{i=\ell+1}^mX_i}.
\end{equation}
\begin{proof}
By Theorem~\ref{thm:permutative-operadic-description}, $\mathscr C$ is naturally an algebra over the Barratt-Eccles operad with $(\ref{eq:shuffle})$ being the action of the morphism $(\sigma,\id)$ of $ E\Sigma_m$ on $(\id_{X_1},\dots,\id_{X_m})$ and  $(\ref{eq:transposition})$ being the action of $(\pi,\id)$ on $(\bigotimes_{i=1}^\ell \id_{X_i},\bigotimes_{i=\ell+1}^m \id_{X_i})$, where $\pi\in\Sigma_2$ is the unique non-trivial element. In this setup, the required identity simply follows by chasing $\big((\pi, \id),\big((\id,\id),(\id,\id)\big),(\id_{X_1},\dots,\id_{X_m})\big)$ through the diagram
\begin{equation*}
\mskip-5mu
\begin{tikzcd}[row sep=tiny, column sep=small]
E\Sigma_2\times E\Sigma_\ell\times E\Sigma_{m-\ell}\times\mathscr C^{\times m}\arrow[dd, "\text{shuffle}"']\arrow[rr, "\text{composition}\times\mathscr C^{\times m}"] && E\Sigma_m\times\mathscr C^{\times m}\arrow[rd, bend left=10pt, "\text{action}"]\\
&\phantom{dummydu}& & \mathscr C\\
E\Sigma_2\times E\Sigma_\ell\times\mathscr C^{\times\ell}\times E\Sigma_{m-\ell}\times\mathscr C^{\times(m-\ell)}\arrow[rr, "E\Sigma_2\times\text{action}\times\text{action}"'] && E\Sigma_2\times\mathscr C^{\times2}\arrow[ur, bend right=10pt, "\text{action}"']
\end{tikzcd}
\end{equation*}
which commutes by what it means to be an $E\Sigma_*$-algebra.
\end{proof}
\end{lemma}

\begin{proof}[Proof of Proposition~\ref{prop:T-equivalence}]
We begin by proving that $T_{\mathscr C}$ is well-defined. For this we first observe that any morphism $(X^{(1)},\dots,X^{(m)})\to(Y^{(1)},\dots,Y^{(n)})$ in $\Sigma\Phi\mathscr C$ indeed admits a representative  $(\psi,f,\phi)$ such that $\phi$ is $X^\bullet$-monotone and $\psi$ is $Y^\bullet$-monotone by Proposition~\ref{prop:suitably-monotone}-$(\ref{item:sm-existence})$.

We will now show that $f$ is indeed a morphism
\begin{equation*}
T_{\mathscr C}(X^\bullet)=\bigotimes_{i=1}^m\bigotimes_{j\in\omega} X_{j}^{(i)}\to\bigotimes_{i=1}^n\bigotimes_{j\in\omega} Y_{j}^{(i)}=T_{\mathscr C}(Y^\bullet);
\end{equation*}
namely, if $(i_1,j_1),\dots,(i_k,j_k)$ is the sequence of those $(i,j)$ such that $X^{(i)}_j\not=\bm1$, ordered lexicographically, and $(i_1',j_1'),\dots,(i_\ell',j_\ell')$ is the sequence of those $(i,j)$ such that $Y^{(i)}_j\not=\bm1$, again ordered lexicographically, then
\begin{equation*}
T_{\mathscr C}(X^\bullet)=X^{(i_1)}_{j_1}\otimes\cdots\otimes X^{(i_k)}_{j_k}
\quad\text{and}\quad
T_{\mathscr C}(Y^\bullet)=Y^{(i_1')}_{j_1'}\otimes\cdots\otimes X^{(i_\ell')}_{j_\ell'}
\end{equation*}
by strict unitality of $\otimes$. On the other hand, $f$ is by definition a morphism $\phi_*(X^\bullet)\to\psi_*(Y^\bullet)$ in $\Phi(\mathscr C)$, i.e.~a morphism
\begin{equation*}
\bigotimes_{i\in\omega}\phi_*(X^\bullet)_i\to\bigotimes_{i\in\omega}\psi_*(Y^\bullet)_i\
\end{equation*}
in $\mathscr C$. But
\begin{equation*}
\bigotimes_{i\in\omega}\phi_*(X^\bullet)_i=
\bigotimes_{i\in\omega\atop i = \phi(i_r,j_r)}X^{(i_r)}_{j_r}=
\bigotimes_{r=1}^k X^{(i_r)}_{j_r},
\end{equation*}
where we have used the strict unitality of $\otimes$ and the assumption that $\phi$ be $X^\bullet$-monotone. Analogously,
\begin{equation*}
\bigotimes_{i\in\omega}\psi_*(Y^\bullet)_i=
\bigotimes_{r=1}^\ell Y^{(i_r')}_{j_r'},
\end{equation*}
so $f$ indeed has the specified source and target.

Next, let us check that $T_{\mathscr C}$ is independent of the choice of representative. Indeed, if $[\psi,f,\phi]=[\psi',f',\phi']$ such that $\phi,\phi'$ are $X^\bullet$-monotone and $\psi,\psi'$ are $Y^\bullet$-monotone, then
\begin{equation*}
f'=[\psi',\psi]_{Y^\bullet}\circ f \circ[\phi,\phi']_{X^\bullet}.
\end{equation*}
But by Proposition~\ref{prop:suitably-monotone}-$(\ref{item:sm-uniqueness})$ both $[\phi,\phi']_{X^\bullet}$ and $[\psi',\psi]_{Y^\bullet}$ are the respective identities in $\mathscr C$, i.e.~$f=f'$ as morphisms in $\mathscr C$.

Now we can prove that $T_{\mathscr C}$ preserves compositions and identities. For this we let $\alpha\colon(X^{(1)},\dots,X^{(m)})\to (Y^{(1)},\dots,Y^{(n)})$ and $\beta\colon (Y^{(1)},\dots,Y^{(n)})\to (Z^{(1)},\dots,Z^{(o)})$ be any two composable morphisms. We then choose an $X^\bullet$-monotone injection $\phi\colon\bm{m}\times\omega\rightarrowtail\omega$, a $Y^\bullet$-monotone injection $\psi\colon\bm{n}\times\omega\rightarrowtail\omega$, and a $Z^\bullet$-monotone injection $\theta\colon\bm{o}\times\omega\rightarrowtail\omega$, yielding representatives $(\psi,f,\phi)$ of $\alpha$ and $(\theta,g,\psi)$ of $\beta$. Then $\beta\alpha=[\theta,gf,\phi]$ by definition of the composition, and hence
\begin{equation*}
T_{\mathscr C}(\beta\alpha)=gf=T_{\mathscr C}(\beta)T_{\mathscr C}(\alpha)
\end{equation*}
as desired. Similarly, if $(X^{(1)},\dots,X^{(m)})$ is any object in $\Sigma\Phi{\mathscr C}$, and $\phi\colon\bm{m}\times\omega\rightarrowtail\omega$ is $X^\bullet$-monotone, then the identity of $X^\bullet$ is given by $[\phi,\id_{\phi_*(X^\bullet)},\phi]$; by the same argument as above, the identity of $\phi_*(X^\bullet)$ in $\Phi(\mathscr C)$ is given by the identity of $\bigotimes_{i=1}^m\bigotimes_{j\in\omega}X^{(i)}_j$ in $\mathscr C$, so
$T_{\mathscr C}(\id_{X^\bullet})=\id_{\bigotimes_{i=1}^m\bigotimes_{j\in\omega}X^{(i)}_j}=\id_{T_{\mathscr C}(X^\bullet)}$. This completes the proof that $T_{\mathscr C}$ is a well-defined functor.

Now we will show that $T_{\mathscr C}$ is strict symmetric monoidal. It is again obvious that $T_{\mathscr C}$ commutes with $\otimes$ on the level of objects and that it strictly preserves the tensor unit. To see that $T_{\mathscr C}$ also commutes with $\otimes$ on the level of morphisms, we let $\alpha\colon(A^{(1)},\dots,A^{(m)})\to(B^{(1)},\dots,B^{(n)})$ and $\beta\colon(C^{(1)},\dots,C^{(o)})\to(D^{(1)},\dots,D^{(p)})$ be any two morphisms in $\Sigma\Phi\mathscr C$. We then pick an $(A^\bullet\otimes C^\bullet)$-monotone $\phi\colon\bm{(m+o)}\times\omega\rightarrowtail\omega$ and a $(B^\bullet\otimes D^\bullet)$-monotone $\psi\colon\bm{(n+p)}\times\omega\rightarrowtail\omega$. Then $\phi|_{\{1,\dots,m\}}$ is $A^\bullet$-monotone, $\phi|_{\{m+1,\dots,m+o\}}$ is $C^\bullet$-monotone, $\psi|_{\{1,\dots,n\}}$ is $B^\bullet$-monotone, and $\psi|_{\{n+1,\dots,n+p\}}$ is $D^\bullet$-monotone. We can therefore write $\alpha=[\psi|_{\{1,\dots,n\}},f,\phi|_{\{1,\dots,m\}}]$ and $\beta=[\psi|_{\{n+1,\dots,n+p\}},g,\phi|_{\{m+1,\dots,m+o\}}]$ for suitable
\begin{align*}
(\phi|_{\{1,\dots,m\}})_*(A^\bullet)&\xrightarrow{f}(\psi|_{\{1,\dots,n\}})_*(B^\bullet)\\
(\phi|_{\{m+1,\dots,m+o\}})_*(C^\bullet)&\xrightarrow{g}(\psi|_{\{n+1,\dots,n+p\}})_*(D^\bullet),
\end{align*}
and by definition $T_{\mathscr C}(\alpha)=f$ and $T_{\mathscr C}(\beta)=g$. On the other hand,
\begin{equation*}
\alpha\otimes\beta=[\psi|_{\{1,\dots,n\}}+\psi|_{\{n+1,\dots,n+p\}},f+g,\phi|_{\{1,\dots,m\}}+\phi|_{\{m+1,\dots,m+o\}}]=[\psi,f+g,\phi]
\end{equation*}
and hence $T_{\mathscr C}(\alpha\otimes\beta)=f+g$. The claim therefore amounts to showing that $f+g=f\otimes g$. To this end we recall that the sum in $\Phi\mathscr C$ of two morphisms $a\colon X\to Y$ and $b\colon X'\to Y'$ such that $\supp(X)\cap\supp(X')=\varnothing$ and $\supp(Y)\cap\supp(Y')=\varnothing$ is defined by conjugating $a\otimes b$ with the coherence isomorphisms associated to the tautological bijections
\begin{equation}\label{eq:supp-shuffle}
\begin{aligned}
\supp(X)\amalg\supp(X')&\to\supp(X)\cup\supp(X')\\ \supp(Y)\amalg\supp(Y')&\to\supp(Y)\cup\supp(Y')
\end{aligned}
\end{equation}
between the disjoint unions and the internal unions, where in each case the left hand side is again made into a totally ordered set by demanding that that any element of the first summand should be smaller than any element of the second summand. It therefore suffices to show that the maps in $(\ref{eq:supp-shuffle})$ are actually monotone for $X=(\phi|_{\{1,\dots,m\}})_*(A^\bullet), X'=(\phi|_{\{m+1,\dots,m+o\}})_*(C^\bullet)$ and $Y=(\psi|_{\{1,\dots,n\}})_*(B^\bullet),Y'=(\psi|_{\{n+1,\dots,n+p\}})_*(D^\bullet)$.

We will prove this for the first map, the argument for the second one being analogous. Namely, we observe that  $\phi|_{\{1,\dots,m\}}(i,j)<\phi|_{\{m+1,\dots,m+o\}}(i',j')$ whenever $A^{(i)}_j\not=\bm1$, $C^{(i')}_{j'}\not=\bm1$ because $\phi$ is $(A^\bullet\otimes C^\bullet)$-monotone. It follows immediately as desired that $k<k'$ for all $k,k'\in\omega$ such that $(\phi|_{\{1,\dots,m\}})_*(A^\bullet)_k\not=\bm1$ and $(\phi|_{\{m+1,\dots,m+o\}})_*(C^\bullet)_{k'}\not=\bm1$, which then completes the argument that $F$ is strict monoidal.

To see that $T_{\mathscr C}$ commutes with $\tau$, we let $(X^{(1)},\dots,X^{(m)}),(Y^{(1)},\dots,Y^{(n)})\in\Sigma\Phi\mathscr C$ be arbitrary. We pick an $(X^\bullet\otimes Y^\bullet)$-monotone injection $\phi\colon\bm{(m+n)}\times\omega\rightarrowtail\omega$ as in the proof of Proposition~\ref{prop:suitably-monotone}-$(\ref{item:sm-existence})$, so that $\supp\big(\phi_*(X^\bullet\otimes Y^\bullet)\big)=\{1,\dots,k\}$ for some $k\in\omega$; then $(X^\bullet\otimes Y^\bullet)$-monotonicity already implies that $\supp\big((\phi|_{\{1,\dots,m\}})_*(X^\bullet)\big)=\{1,\dots,\ell\}$ for some $\ell\le k$.

As $\phi$ is in particular an injection, $\tau_{X^\bullet,Y^\bullet}\colon X^\bullet\otimes Y^\bullet\to Y^\bullet\otimes X^\bullet$ is given by $[\bar\phi,\id,\phi]$. However, $\bar\phi$ is usually \emph{not} $(Y^\bullet\otimes X^\bullet)$-monotone. We now define $u\in\mathcal M$ via
\begin{equation*}
u(i)=\begin{cases}
i+k-\ell & \text{if }i\le\ell\\
i-\ell & \text{if }\ell<i\le k\\
i & \text{otherwise}.
\end{cases}
\end{equation*}
Then one easily checks that $u\bar\phi$ is indeed $(Y^\bullet\otimes X^\bullet)$-monotone. On the other hand
\begin{equation*}
[\bar\phi,\id_{\phi_*(X^\bullet\otimes Y^\bullet)},\phi]=[u\bar\phi, [u\bar\phi,\bar\phi]_{Y^\bullet\otimes X^\bullet},\phi] =
[u\bar\phi, [u,1]_{\bar\phi_*(Y^\bullet\otimes X^\bullet)},\phi]
\end{equation*}
by Lemma~\ref{lemma:generalized-structure-maps}-$(\ref{item:gsm-EM-equivariant})$, hence $T_{\mathscr C}(\tau_{X^\bullet,Y^\bullet})=[u,1]_{\bar\phi_*(Y^\bullet\otimes X^\bullet)}$.

We now plug in the definition of $[u,1]_{Z}$: this was given by choosing a $K\gg0$ such that $Z_i=\bm1=(u_*Z)_i$ for $i>K$ and a permutation $\sigma\in\Sigma_K$ such that $u(i)=\sigma(i)$ for all $i\le K$ with $Z_i\not=\bm1$; then as a morphism in $\mathscr C$, $[u,1]_{Z}$ is the coherence isomorphism associated to $\sigma$. In the case that $Z=\bar\phi_*(Y^\bullet\otimes X^\bullet)$ we may take $K=k$ and $\sigma(i)=u(i)$ for \emph{all} $i\le k$, so we altogether see that $T_{\mathscr C}(\tau_{X^\bullet,Y^\bullet})=[u,1]_{\bar\phi_*(Y^\bullet\otimes X^\bullet)}$ is the coherence isomorphism
\begin{align*}
&\bigotimes_{i=1}^k \big(\bar\phi_*(Y^\bullet\otimes X^\bullet)\big)_i=\bigotimes_{i=1}^\ell \big((\phi|_{\{1,\dots,m\}})_*X)_i\otimes
\bigotimes_{i=\ell+1}^k \big((\phi|_{\{m+1,\dots,m+n\}})_*Y)_i\\
&\quad \to\bigotimes_{i=\ell+1}^k \big((\phi|_{\{m+1,\dots,m+n\}})_*Y)_i\otimes\bigotimes_{i=1}^\ell \big((\phi|_{\{1,\dots,m\}})_*X)_i=\bigotimes_{i=1}^k \big((u\bar\phi)_*(Y^\bullet\otimes X^\bullet)\big)_i
\end{align*}
associated to the permutation $\sigma$ that moves the first $\ell$ entries to the end. On the other hand
\begin{equation*}
\tau_{T_{\mathscr C}(X^\bullet), T_{\mathscr C}(Y^\bullet)}=\tau_{\bigotimes_{i=1}^\ell(\phi|_{\{1,\dots,m\}})_*(X^\bullet)_i,\bigotimes_{i=\ell+1}^k(\phi|_{\{m+1,\dots,m+n\}})_*(Y^\bullet)_i}
\end{equation*}
(where we have again used the construction of $\phi$ and strict unitality of $\otimes$), and these two agree by the previous lemma.

Altogether, we have now shown that $T_{\mathscr C}$ is a strict symmetric monoidal functor. Let us now show that it is an underlying equivalence of categories. Indeed, if $X\in\mathscr C$ is arbitrary, then the object $(\overline X)\in\Sigma\Phi\mathscr C$ with $\overline X\in\Phi\mathscr C$ being the sequence with $\overline X_1=X$ and $\overline X_i=\bm1$ otherwise, is a preimage, i.e.~$T_{\mathscr C}$ is even surjective on objects. On the other hand, if $(X^{(1)},\dots, X^{(m)}), (Y^{(1)},\dots,Y^{(n)})\in\Sigma\Phi\mathscr C$, then we can pick an $X^\bullet$-monotone injection $\phi\colon\bm{m}\times\omega\rightarrowtail\omega$ and a $Y^\bullet$-monotone injection $\psi\colon\bm{n}\times\omega\rightarrowtail\omega$. Then
\begin{align*}
\Hom_{\mathscr C}\big(\bigotimes_{i=1}^m\bigotimes_{j\in\omega} X^{(i)}_j,
\bigotimes_{i=1}^n\bigotimes_{j\in\omega} Y^{(i)}_j\big)=
\Hom_{\Phi\mathscr C}(\phi_*X^\bullet, \psi_*Y^\bullet) &\to \Hom_{\Sigma\Phi\mathscr C}(X^\bullet, Y^\bullet)\\
f&\mapsto[\psi,f,\phi]
\end{align*}
is by construction right inverse to
\begin{equation}\label{eq:T-C-hom}
\begin{aligned}
\Hom_{\Sigma\Phi\mathscr C}(X^\bullet, Y^\bullet)&\xrightarrow{T_{\mathscr C}}\Hom_{\mathscr C}(T_{\mathscr C}(X^\bullet), T_{\mathscr C}(Y^\bullet))\\
&=\Hom_{\mathscr C}\big(\bigotimes_{i=1}^m\bigotimes_{j\in\omega} X^{(i)}_j,\bigotimes_{i=1}^n\bigotimes_{j\in\omega} Y^{(i)}_j\big),
\end{aligned}
\end{equation}
but it is also bijective by Lemma~\ref{lemma:free-to-choose-phi-psi}. Thus also $(\ref{eq:T-C-hom})$ is bijective, i.e.~$T_{\mathscr C}$ is fully faithful and thus an equivalence of categories.

Finally, we have to show that $T_{\mathscr C}$ is natural in strict symmetric monoidal functors, i.e.~for any strict symmetric monoidal functor $F\colon\mathscr C\to\mathscr D$ of permutative categories $T_{\mathscr D}\circ(\Sigma\Phi F)=F\circ T_{\mathscr C}$. This is obvious on the level of objects; to prove the claim on the level of morphisms, we let $(X^{(1)},\dots,X^{(m)}),(Y^{(1)},\dots,Y^{(n)})\in\Sigma\Phi\mathscr C$, and we pick an $X^\bullet$-monotone injection $\phi\colon\bm{m}\times\omega\rightarrowtail\omega$ and a $Y^\bullet$-monotone injection $\psi\colon\bm{n}\times\omega\rightarrowtail\omega$. Then any morphism $X^\bullet\to Y^\bullet$ is of the form $[\psi,f,\phi]$ for some $f\colon\phi_*(X^\bullet) \to\psi_*(Y^\bullet)$ in $\Phi\mathscr C$, i.e.~a map $f\colon\bigotimes_{i=1}^m\bigotimes_{j\in\omega}X^{(i)}_j\to\bigotimes_{i=1}^n\bigotimes_{j\in\omega} Y^{(i)}_j$ in $\mathscr C$, and by definition $(\Sigma\Phi F)[\psi,f,\phi]=[\psi, (\Phi F)f,\phi]=[\psi, Ff, \phi]\colon (\Sigma\Phi F)(X^\bullet)\to (\Sigma\Phi F)(Y^\bullet)$. We now claim that $\phi$ is $(\Sigma\Phi F)(X^\bullet)$-monotone and that $\psi$ is $(\Sigma\Phi F)(Y^\bullet)$-monotone. Indeed, it suffices to prove this for the first statement, for which we observe that $(\Sigma\Phi F)(X^\bullet)^{(i)}_j = F(X^{(i)}_j)$. As $F$ is strict symmetric monoidal, this means that $(\Sigma\Phi F)(X^\bullet)^{(i)}_j=\bm1$ whenever $X^{(i)}_j=\bm 1$; thus, the claim follows immediately from the definition of $(\Sigma\Phi F)(X^\bullet)$-monotonicity. But then
\begin{equation*}
T_{\mathscr D}\big((\Sigma\Phi F)[\psi,f,\phi]\big)=T_{\mathscr D}[\psi, Ff, \phi]=Ff = F(T_{\mathscr C}[\psi, f,\phi]).
\end{equation*}
This completes the proof of the proposition.
\end{proof}

Together with Proposition~\ref{prop:comparison-schwede} we now immediately conclude:

\begin{cor}
For any injection $\mu\colon\bm2\times\omega\rightarrowtail\omega$ there exists a preferred natural levelwise equivalence from $\mu^*\circ\Phi$ to the inclusion $\cat{PermCat}\to\cat{SymMonCat}$.\qed
\end{cor}

\begin{rk}\label{rk:comparison-schwede}
In \cite[Remark~11.4]{schwede-k-theory}, Schwede already sketched the existence of a natural levelwise equivalence from the inclusion to $\mu^*\circ\Phi$. Taking this result as well as Theorem~\ref{thm:perm-vs-sym} above for granted, we could have proven our main result without using Proposition~\ref{prop:T-equivalence}: namely, they imply that $\mu^*\circ\Phi$ is a homotopy equivalence, while we will see in the next section that $\Sigma$ is homotopical and that $\Phi\circ\Sigma$ is connected by a zig-zag of levelwise underlying equivalences to the identity, hence in particular a homotopy equivalence. By $2$-out-of-$6$ one can then conclude that $\Phi$ is a homotopy equivalence, and it then follows formally that $\Sigma$ is not only right but also left homotopy inverse to it.

However, we have decided to give the above proof for two reasons: firstly, it keeps the proof of the main theorem self-contained, and secondly (and more importantly) it avoids appealing to Theorem~\ref{thm:perm-vs-sym} and \cite[Remark~11.4]{schwede-k-theory}, both of whose proofs have only been sketched.
\end{rk}

\section{Proof of the main theorem}\label{sec:main-thm}
In this section we will finally prove the main result of this article:

\begin{thm}\label{thm:main-thm}
The functors $\Phi$ and $\Sigma$ define mutually inverse homotopy equivalences $\cat{PermCat}\rightleftarrows\cat{ParSumCat}$ with respect to the underlying equivalences of categories.
\end{thm}

To this end, we have to construct a zig-zag of natural levelwise underlying equivalences between $\Phi\Sigma$ and the identity of $\cat{ParSumCat}$. Let us give the basic idea for this first: if $\mathcal C$ is any parsummable category, then an object of $\Phi\Sigma(\mathcal C)$ is given by a sequence $X_\bullet^\bullet=(X^\bullet_1,X^\bullet_2,\dots)$ of tuples of objects of $\mathcal C$, such that almost all these tuples are equal to the empty tuple $\epsilon$. We want to cook up a single object of $\mathcal C$ from this, and an obvious thing to try is to sum the individual entries. Of course, in general the entries of $X^\bullet_\bullet$ might not be summable, but we can arrange this by choosing (once and for all) an injection $\phi\colon\omega^{\times3}\rightarrowtail\omega$ and replacing $X^{(i)}_j$ by $\phi(i,j,\blank)_*(X^{(i)}_j)$ first. Put differently, if we view $X^\bullet_\bullet$ in the evident way as a family indexed by a finite set $A\subset\omega\times\omega$, then we are sending $X^\bullet_\bullet$ to $(\phi|_A)_*(X^\bullet_\bullet)$.

While one can indeed extend the above assignment to an equivalence $\Phi\Sigma(\mathcal C)\to\mathcal C$ of ordinary categories in a natural way, this is in general \emph{not} a morphism of parsummable categories: namely, this functor is usually not $E\mathcal M$-equivariant, and in fact not even $\mathcal M$-equivariant on objects. However, Lemma~\ref{lemma:generalized-action}-$(\ref{item:lga-partition})$ can be interpreted as saying that it is still in a suitable sense compatible with summation.

To solve the equivariance issue, we will provide a general construction below turning any such functor $F\colon\mathcal C\to\mathcal D$ into a zig-zag $\mathcal C\to \Cyl(F)\xleftarrow{\simeq}\mathcal D$ of morphisms of parsummable categories. The existence of such a construction is in fact not surprising: for example, if we forget about the sum operations, then any non-equivariant functor between $E\mathcal M$-categories can be replaced by an (a priori possibly longer) zig-zag of equivariant functors. One way to see this, is to view $\cat{$\bm{E\mathcal M}$-Cat}$ as the category of simplicially enriched functors $B(E\mathcal M)\to\cat{Cat}$, where $B(E\mathcal M)$ is the simplicially enriched category with one object, whose simplicial monoid of endomorphisms is given by the nerve of $E\mathcal M$. As $E\mathcal M$ is a contractible groupoid, abstract nonsense then tells us that the forgetful functor $\forget\colon\cat{$\bm{E\mathcal M}$-Cat}\to\cat{Cat}$ induces an equivalence after localizing at the underlying equivalences of categories. If we take the usual Gabriel-Zisman construction of localizations, then for any $E\mathcal M$-categories $\mathcal C,\mathcal D$ choosing a preimage of a given functor $\forget\mathcal C\to\forget\mathcal D$ in the localization of $\cat{Cat}$ yields such a zig-zag.

\subsection{A mapping cylinder construction}
In this subsection, we will provide the general strictification construction promised above. More precisely, we will consider the following kind of functors:

\begin{defi}
Let $\mathcal C,\mathcal D$ be parsummable categories and let $F\colon\mathcal C\to\mathcal D$ be a functor of their underlying categories. Then we say that $F$ \emph{preserves sums} if the following conditions are satisfied:
\begin{enumerate}
\item $F(0)=0$
\item $F\times F\colon \mathcal C\times\mathcal C\to\mathcal D\times\mathcal D$ restricts to a functor $\mathcal C\boxtimes\mathcal C\to\mathcal D\boxtimes\mathcal D$, which we denote by $F\boxtimes F$, and the diagram
\begin{equation*}
\begin{tikzcd}
\mathcal C\boxtimes\mathcal C\arrow[d, "+"'] \arrow[r, "F\boxtimes F"] & \mathcal D\boxtimes\mathcal D\arrow[d, "+"]\\
\mathcal C\arrow[r, "F"'] & \mathcal D
\end{tikzcd}
\end{equation*}
commutes.
\end{enumerate}
\end{defi}

\begin{rk}
We can reformulate the first half of the second condition as saying that $\supp_{\mathcal D} F(X)\cap\supp_{\mathcal D} F(Y)=\varnothing$ for all $X,Y\in{\mathcal C}$ with $\supp_{\mathcal C}(X)\cap\supp_{\mathcal C}(Y)=\varnothing$. This makes it clear that more generally $F^{\times n}$ restricts to $F^{\boxtimes n}\colon\mathcal C^{\boxtimes n}\to\mathcal D^{\boxtimes n}$ for all $n\ge0$.
\end{rk}

\begin{ex}
Any morphism of parsummable categories preserves sums in the above sense.
\end{ex}

\begin{ex}\label{ex:action-sum-preserving}
If $\mathcal C$ is any parsummable category and $u\in\mathcal M$, then $u_*\colon\mathcal C\to\mathcal C$ preserves sums (although it is typically not $E\mathcal M$-equivariant): the first condition follows from the fact that $0$ has empty support (i.e.~it is $\mathcal M$-fixed), while the second one is a consequence of $E\mathcal M$-equivariance of the sum functor $\mathcal C\boxtimes\mathcal C\to\mathcal C$.
\end{ex}

\begin{rk}\label{rk:product-sum-preserving}
The category $\cat{ParSumCat}$ is complete, with finite limits created in $\cat{Cat}$, see~\cite[Example~4.11]{schwede-k-theory}. In particular, given two parsummable categories, the product of their underlying categories inherits a natural tame $E\mathcal M$-action and parsummable structure.

If $F_1\colon\mathcal C_1\to\mathcal D_1,F_2\colon\mathcal C_2\to\mathcal D_2$ preserve sums, then it is not true in general that $F_1\times F_2\colon\mathcal C_1\times\mathcal C_2\to\mathcal D_1\times\mathcal D_2$ preserves sums: namely, it can happen that while $(X_1,X_2), (Y_1,Y_2)\in\mathcal C_1\times\mathcal C_2$ have disjoint supports, their images under $F_1\times F_2$ do not. However, by assumption on $F_1$ and $F_2$, the only thing that can go wrong is that $\supp F_1(X_1)\cap\supp F_2(Y_2)\not=\varnothing$ or $\supp F_2(X_2)\cap \supp F_1(Y_1)\not=\varnothing$. Thus, this issue disappears for example as soon as $F_1\times F_2$ factors through $\mathcal D_1\boxtimes\mathcal D_2$.

Once we know, however, that $F_1\times F_2$ sends summable pairs of objects of $\mathcal C_1\times\mathcal C_2$ to summable pairs of objects of $\mathcal D_1\times\mathcal D_2$, it is trivial to check that $F_1\times F_2$ preserves sums as these can be calculated componentwise (assuming they exist).
\end{rk}

Before we can provide the desired strictification procedure, we need an auxiliary construction:

\begin{constr}\label{constr:F-lower-star}
Let $\mathcal C,\mathcal D$ be parsummable categories and let $F\colon\mathcal C\to\mathcal D$ be a functor of their underlying categories that preserves sums.

We now define a parsummable category $F_*\mathcal C$ as follows: the small category $F_*\mathcal C$ has objects the objects of $\mathcal C$, and if $X,Y\in F_*\mathcal C$, then we set $\Hom_{F_*\mathcal C}(X,Y)\mathrel{:=}\Hom_{\mathcal D}(FX,FY)$ with the composition inherited from $\mathcal D$.

The $\mathcal M$-action on $\Ob(F_*\mathcal C)=\Ob(\mathcal C)$ is inherited from the $E\mathcal M$-action on $\mathcal C$. If $u\in\mathcal M$, $X\in F_*\mathcal C$, then $u_\circ^X\colon X\to u_*X$ is given by the morphism $F(u_\circ^X)\colon F(X)\to F(u_*X)$, i.e.~the image of the structure isomorphism in $\mathcal C$ under $F$.

The sum on $\Ob(F_*\mathcal C)=\Ob(\mathcal C)$ is given by the sum in $\mathcal C$; in particular, the additive unit is the additive unit of $\mathcal C$. The sum on morphisms is given by the sum in $\mathcal D$.
\end{constr}

\begin{lemma}
The above is a well-defined parsummable category.
\begin{proof}
It is clear that $F_*\mathcal C$ is a well-defined small category.

The above defines an $\mathcal M$-action on $\Ob(F_*\mathcal C)=\Ob(\mathcal C)$ as $\mathcal C$ was assumed to be an $E\mathcal M$-category. To prove that $F_*\mathcal C$ is an $E\mathcal M$-category, it is therefore enough to show that $u^{v_*X}_\circ v_\circ^X=(uv)^X_\circ$ as morphisms in $\mathcal D$ for any $X\in F_*\mathcal{C}$ and $u,v\in\mathcal M$. Plugging in the definitions, the left hand side is given by $F(u^{v_*X})F(v_\circ^X)=F(u^{v_*X}v_\circ^X)$ while the right hand side is given by $F((uv)^X_\circ)$. Thus, the claim follows again immediately from $\mathcal C$ being an $E\mathcal M$-category.

By definition of the $\mathcal M$-action, $\supp_{F_*\mathcal C}(X)=\supp_{\mathcal C}(X)$ for all $X\in F_*\mathcal C$. In particular, $F_*\mathcal C$ is tame, the sum is well-defined on objects, and $0$ has empty support. Commutativity, associativity, and unitality of the sum on the level of objects are then immediately inherited from $\mathcal C$.

Next, let us show that the sum is well-defined on morphisms. For this we let $X,X',Y,Y'\in F_*\mathcal C$ and we let $f\colon F(X)\to F(Y)$ and $g\colon F(X')\to F(Y')$ define morphisms $X\to Y$ and $X'\to Y'$ in $F_*\mathcal C$. If $X$ and $X'$ are summable in $F_*\mathcal C$, hence in $\mathcal C$, then also the sum $F(X)+F(X')$ exists and it is equal to $F(X+X')$ as $F$ preserves sums. Arguing similarly for the target, we see that the sum $f+g$ exists in $\mathcal D$ and that it is a morphism $F(X+X')\to F(Y+Y')$, so that it defines a morphism $X+X'\to Y+Y'$ in $F_*\mathcal C$. This shows that the sum on morphisms is well-defined. Commutativity and associativity are then immediate from the corresponding statements for objects in $F_*\mathcal C$ and for morphisms in $\mathcal D$. Unitality follows analogously once we observe that $F(\id_0)=\id_0$ as $F(0)=0$.

It only remains to show that the sum is $E\mathcal M$-equivariant. It is again clear that the sum is $\mathcal M$-equivariant on objects, so it suffices to show that $u^X_\circ+ u^Y_\circ= u^{X+Y}_\circ$ as morphisms in $\mathcal D$ for all $u\in\mathcal M$ and all $X,Y\in F_*(\mathcal C)$ with $\supp(X)\cap\supp(Y)=\varnothing$. But the left hand side is defined as $F(u_\circ^X)+F(u_\circ^Y)$ while the right hand side is defined as $F(u^{X+Y}_\circ)$. As the sum in $\mathcal C$ is $E\mathcal M$-equivariant, the latter equals $F(u^X_\circ+u^Y_\circ)$; the claim therefore follows from $F$ preserving sums.
\end{proof}
\end{lemma}

\begin{constr}\label{constr:F-lower-star-I}
In the situation of Construction~\ref{constr:F-lower-star} we define $I\colon\mathcal C\to F_*\mathcal C$ as follows: on objects, $I$ is given by the identity, and it sends a morphism $f\colon X\to Y$ in $\mathcal C$ to the morphism $X\to Y$ in $F_*\mathcal C$ given by the morphism $F(f)\colon F(X)\to F(Y)$ of $\mathcal D$.

We moreover define $\hat F\colon F_*\mathcal C\to\mathcal D$ as the functor that is given on objects by $F$ and on hom sets by the identity, i.e.~it sends a morphism $X\to Y$ given by a morphism $f\colon F(X)\to F(Y)$ in $\mathcal D$ to $f$.
\end{constr}

\begin{lemma}\label{lemma:precylinder}
\begin{enumerate}
\item $I$ is a morphism of parsummable categories.
\item If $F$ is essentially surjective, then the functor $\hat F$ is an equivalence of categories.
\item As ordinary functors, $F=\hat F\circ I$.
\end{enumerate}
\begin{proof}
For the first statement, we observe that $I$ is clearly $\mathcal M$-equivariant on objects. For the proof of $E\mathcal M$-equivariance it is therefore enough that $I(u^X_{\circ})=u^{I(X)}_\circ$ as morphisms in $\mathcal D$ for all $X\in\mathcal C$, $u\in\mathcal M$. However, both sides are defined as $F(u^X_\circ)$, finishing the proof of $E\mathcal M$-equivariance.

Similarly, $I$ clearly preserves the additive unit and sums of objects, and it is additive on morphisms as $F$ was assumed to preserve sums. This completes the proof of the first statement.

For the second statement we observe that $\hat F$ is always fully faithful by construction. As it is given on objects by $F$, it is clearly essentially surjective if $F$ is.

The final statement follows immediately from the definitions.
\end{proof}
\end{lemma}

On the level of underlying categories, the above in particular provides a factorization into a functor that is \emph{bijective on objects} followed by a fully faithful functor. Such factorizations in $\cat{Cat}$ have been studied classically and in particular in the related context of strictification results for algebras, see e.g.~\cite[Lemma~3.3]{power-coherence} or more recently \cite[Section~4.2]{gmmo}.

\begin{constr}
We write $\cat{PsArParSumCat}$ (`pseudo-arrows in parsummable categories') for the category whose objects are functors $F\colon\mathcal C\to\mathcal D$ that preserve sums, and whose morphisms are commutative diagrams
\begin{equation}\label{diag:PsArParSumCat-morphism}
\begin{tikzcd}
\mathcal C\arrow[r, "F"]\arrow[d, "G"'] &\mathcal D\arrow[d, "H"]\\
\mathcal C'\arrow[r, "F'"'] & \mathcal D'
\end{tikzcd}
\end{equation}
such that the vertical arrows are \emph{morphisms of parsummable categories}. The composition in $\cat{PsArParSumCat}$ is induced by the composition in $\cat{ParSumCat}$.

We now define a functor $(\blank)_*\colon\cat{PsArParSumCat}\to\cat{ParSumCat}$ as follows: an object $F\colon\mathcal C\to\mathcal D$ is sent to $F_*\mathcal C$, and a morphism given by a commutative diagram $(\ref{diag:PsArParSumCat-morphism})$ is sent to the functor $F_*(G,H)\colon F_*\mathcal C\to F_*\mathcal D$ given on objects by $G$ and on morphisms by $H$.
\end{constr}

\begin{lemma}\label{lemma:precylinder-functoriality}
This makes $(\blank)_*$ into a well-defined functor $\cat{PsArParSumCat}\to\cat{ParSumCat}$, and the maps $I$ from Construction~\ref{constr:F-lower-star-I} assemble into a natural transformation $\pr_1\Rightarrow (\blank)_*$, where $\pr_1$ denotes the projection sending a morphism $(\ref{diag:PsArParSumCat-morphism})$ in $\cat{PsArParSumCat}$ to $G\colon\mathcal C\to\mathcal C'$.

\begin{proof}
To see that $F_*(G,H)$ is well-defined, we observe that if $X,Y\in F_*\mathcal C$ and $f\colon F(X)\to F(Y)$ is a morphism in $\mathcal D$ defining a morphism $X\to Y$ in $F_*\mathcal C$, then $Hf$ is a morphism $HF(X)\to HF(Y)$ in $\mathcal D'$, i.e.~a morphism $F'G(X)\to F'G(Y)$ by commutativity of $(\ref{diag:PsArParSumCat-morphism})$. Thus it defines a morphism $G(X)\to G(Y)$ in $F'_*\mathcal C'$ as desired. It is then clear that $F_*(G,H)$ is a functor.

Since $G$ is $E\mathcal M$-equivariant, $F_*(G,H)$ preserves the $\mathcal M$-action on objects. To prove that it is $E\mathcal M$-equivariant, it is therefore enough that $F_*(G,H)(u_\circ^X)=u^{F_*(G,H)(X)}_\circ$ as morphisms in $\mathcal D'$ for all $X\in F_*\mathcal C$ and $u\in\mathcal M$. But the left hand side is defined as $HF(u^X_\circ)$ (where we now view $X$ as an object of $\mathcal C$) whereas the right hand side is defined as $F'(u^{G(X)}_\circ)$. As $G$ is $E\mathcal M$-equivariant, the latter agrees with $F'G(u^X_\circ)$, so the claim follows again from the commutativity of $(\ref{diag:PsArParSumCat-morphism})$.

Similarly, the compatibility of $F_*(G,H)$ with sums of objects follows from the fact that $G$ is a morphism of parsummable categories, while the compatibility with sums of morphisms follows from the fact that $H$ is a morphism of parsummable categories.

It is then clear that $(\blank)_*$ is a functor. Finally, the naturality statement is again clear on objects, and on morphisms it follows once more from the commutativity of $(\ref{diag:PsArParSumCat-morphism})$.
\end{proof}
\end{lemma}

\begin{constr}\label{constr:cyl}
Let $\mathcal C,\mathcal D$ be parsummable categories, and let $F\colon\mathcal C\to\mathcal D$ preserve sums. We define $\Cyl(F)$ as follows: we fix an injection $\mu\colon\bm2\times\omega\rightarrowtail\omega$, and we define $F_\mu$ as the composition
\begin{equation*}
\mathcal C\times\mathcal D\xrightarrow{F\times \id}\mathcal D\times\mathcal D\xrightarrow{\mu_*}\mathcal D.
\end{equation*}
We then set $\Cyl(F){:=} (F_\mu)_*(\mathcal C\times\mathcal D)$. Explicitly, this means that an object of $\Cyl(F)$ is a pair $(X,Y)$ with $X\in\mathcal C,Y\in\mathcal D$, and the hom sets are given by $\Hom_{\Cyl(F)}\big((X,Y),(X',Y')\big)\mathrel{:=}\Hom_{\mathcal D}\big(\mu_*(F(X),Y), \mu_*(F(X'),Y')\big)$.

We define $I^{(1)}\colon\mathcal C\to\Cyl(F)$ as the composition
\begin{equation*}
\mathcal C\xrightarrow{(\blank,0)}\mathcal C\times\mathcal D\xrightarrow{I} (F_\mu)_*(\mathcal C\times\mathcal D) =\Cyl(F),
\end{equation*}
and analogously we get $I^{(2)}\colon\mathcal D\to\Cyl(F)$.
\end{constr}

\begin{cor}\label{cor:cylinder-factorization}
\begin{enumerate}
\item $\Cyl(F)$ is a well-defined parsummable category, and $I^{(1)}$, $I^{(2)}$ are morphisms of parsummable categories.
\item $I^{(1)}\colon\mathcal C\to\Cyl(F)$ is an underlying equivalence if $F$ is.
\item $I^{(2)}\colon\mathcal D\to\Cyl(F)$ is an underlying equivalence.
\end{enumerate}
\begin{proof}
For the first statement it suffices by Lemma~\ref{lemma:precylinder} that $F_\mu$ preserves sums. For this we factor $F_\mu$ as
\begin{equation*}
\mathcal C\times\mathcal D\xrightarrow{(\mu(1,\blank)_*\circ F)\times\mu(2,\blank)_*} \mathcal D\boxtimes\mathcal D\xrightarrow{+}\mathcal D.
\end{equation*}
Clearly, compositions of sum-preserving functors again preserve sums. As $+\colon\mathcal D\boxtimes\mathcal D\to\mathcal D$ is even a morphism of parsummable categories, the claim therefore follows from Remark~\ref{rk:product-sum-preserving}.

For the remaining statements we contemplate the diagram
\begin{equation*}
\begin{tikzcd}
\mathcal C\arrow[r, "I^{(1)}"]\arrow[rdd, bend right=20pt, "{\mu(1,\blank)_*\circ F}"'] & \Cyl(F)\arrow[d, equal] & \arrow[l, "I^{(2)}"']\mathcal D\arrow[ldd, "{\mu(2,\blank)_*}", bend left=20pt]\\
& (F_\mu)_*(\mathcal C\times\mathcal D)\arrow[d, "\widehat{F_\mu}"]\\
& \mathcal D
\end{tikzcd}
\end{equation*}
which commutes by a straight-forward computation using that $\widehat{F_\mu}\circ I=F_\mu$. Since $F_\mu(0,Y)=\mu(2,\blank)_*(Y)\cong Y$ for any $Y\in\mathcal D$, $F_\mu$ is essentially surjective, so Lemma~\ref{lemma:precylinder} implies that $\widehat{F_\mu}$ is an underlying equivalence. As so are $\mu(1,\blank)_*$ and $\mu(2,\blank)_*$, the remaining claims follow immediately by $2$-out-of-$3$.
\end{proof}
\end{cor}

Using that $\mu(1,\blank)_*$ and $\mu(2,\blank)_*$ are isomorphic (as ordinary functors) to the identity of $\mathcal D$, we can summarize the above situation as follows: any sum-preserving $F\colon\mathcal C\to\mathcal D$ admits a factoriation
\begin{equation}\label{eq:cylinder-factorization}
\begin{tikzcd}
\mathcal C\arrow[r, "I^{(1)}"] & \Cyl(F) \arrow[r, "\widehat{F_\mu}", ""{name=A, below, very near start, xshift=-6pt},""{name=B, below, very near end, xshift=6pt}] & \arrow[from=B, to=A, yshift=-2.5pt, bend left=17.5pt, dashed, "I^{(2)}"] \mathcal D
\end{tikzcd}
\end{equation}
up to natural isomorphism of ordinary functors, such that $I^{(1)}$ is a morphism of parsummable categories and $\widehat{F_\mu}$ is an equivalence of ordinary categories quasi-inverse to the morphism  $I^{(2)}\colon\mathcal D\to\Cyl(F)$ of parsummable categories.

We close this discussion by establishing in which sense the factorization $(\ref{eq:cylinder-factorization})$ is functorial:

\begin{constr}
We extend Construction~\ref{constr:cyl} to a functor
\begin{equation*}
\Cyl\colon\cat{PsArParSumCat}\to\cat{ParSumCat}
\end{equation*}
as follows: a commutative diagram as in $(\ref{diag:PsArParSumCat-morphism})$ is sent to $(F_\mu)_*(G\times H,H)$.
\end{constr}

\begin{cor}\label{cor:cylinder-natural}
The above is a well-defined functor. Moreover, the maps $I^{(1)}$ and $I^{(2)}$ assemble into natural transformations $\pr_1\Rightarrow\Cyl\Leftarrow\pr_2$.
\begin{proof}
To see that $\Cyl$ is a well-defined functor, it suffices by Lemma~\ref{lemma:precylinder-functoriality} that the total rectangle in
\begin{equation*}
\begin{tikzcd}
\mathcal C\times\mathcal D\arrow[d,"G\times H"']\arrow[r, "F\times\id"] & \mathcal D\times\mathcal D\arrow[d, "H\times H"]\arrow[r, "\mu_*"] & \mathcal D\arrow[d, "H"]\\
\mathcal C'\times\mathcal D'\arrow[r, "F'\times\id"'] & \mathcal D'\times\mathcal D'\arrow[r, "\mu_*"'] & \mathcal D'
\end{tikzcd}
\end{equation*}
commutes for any commutative square $(\ref{diag:PsArParSumCat-morphism})$. But indeed, for the left hand square this holds by assumption, while for the right hand square this follows from the assumption that $H$ be a morphism of parsummable categories.

The naturality of $I^{(1)},I^{(2)}$ follows from the same lemma once we observe that the inclusions $\mathcal C\hookrightarrow\mathcal C\times\mathcal D\hookleftarrow\mathcal D$ are clearly natural.
\end{proof}
\end{cor}

\subsection{From \texorpdfstring{$\bm{\Phi\Sigma}$}{Phi Sigma} to the identity}
Using the general procedure developed in the previous subsection, we will now construct the desired zig-zag between $\Phi\Sigma$ and the identity of $\cat{ParSumCat}$. While we could indeed extend the assignment from the beginning of this section to a natural sum-preserving functor $\Phi\Sigma(\mathcal C)\to\mathcal C$ for any parsummable category $\mathcal C$ and use this to produce our zig-zag, the calculations involved become simpler if we restrict to a certain subcategory of $\Phi\Sigma(\mathcal C)$ first:

\begin{constr}
Given a parsummable category $\mathcal C$, we write $\Theta(\mathcal C)$ for the full subcategory of $\Phi\Sigma(\mathcal C)$ spanned by those sequences $(X_1^\bullet,X_2^\bullet,\dots)$ such that each $X_i^\bullet$ is either a $0$-tuple (i.e.~$X_i^\bullet=\epsilon$) or a $1$-tuple.
\end{constr}

\begin{lemma}
The subcategory $\Theta(\mathcal C)\subset\Phi\Sigma(\mathcal C)$ is closed under the $E\mathcal M$-action and the sum operation, so that it becomes a parsummable category in its own right. Moreover, $\Theta$ defines a subfunctor of $\Phi\Sigma\colon\cat{ParSumCat}\to\cat{ParSumCat}$.
\begin{proof}
As $\Theta(\mathcal C)$ is by definition a full subcategory, it suffices for the first statement that its objects are closed under the $\mathcal M$-action and under the sum of $\Phi\Sigma(\mathcal C)$, and that they contain the additive unit. However, all of these claims follow immediately from the definition of $\Phi$.

For the second statement it suffices to observe that on objects $\Phi\Sigma(F)$ is given for any $F\colon\mathcal C\to\mathcal D$ by applying $F$ componentwise, so it clearly sends $\Theta(\mathcal C)$ to $\Theta(\mathcal D)$.
\end{proof}
\end{lemma}

\begin{lemma}\label{lemma:Theta-vs-Phi-Sigma}
The inclusions assemble into a natural levelwise underlying equivalence $\Theta\Rightarrow\Phi\Sigma$.
\begin{proof}
It only remains to show that $\Theta(\mathcal C)\hookrightarrow\Phi\Sigma(\mathcal C)$ is essentially surjective for any parsummable category $\mathcal C$. But indeed, a general object of $\Phi\Sigma(\mathcal C)$ is a sequence $(X_1^\bullet, X_2^\bullet,\dots)$ of tuples in $\mathcal C$ such that almost all $X_i^\bullet$ equal $\epsilon$; we construct an object $(Y_1^\bullet, Y_2^\bullet,\dots)$ of $\Theta\mathcal C$ together with isomorphisms $\alpha_i\colon X_i^\bullet\to Y_i^\bullet$ in $\Sigma\mathcal C$ as follows: if $X_i^\bullet=\epsilon$, then $Y_i^\bullet\mathrel{:=}\epsilon$ and $\alpha_i\mathrel{:=}\id_\epsilon$; otherwise, if $X_i^\bullet=(X_i^{(1)},\dots, X_i^{(m_i)})$, then we can (as in the proof of Proposition~\ref{prop:comparison-schwede}) pick an injection $\phi\colon\bm{m_i}\times\omega\rightarrowtail\omega$ and we set $Y_i^\bullet\mathrel{:=}(\phi_*(X_i^\bullet))$ (a $1$-tuple) and $\alpha_i\mathrel{:=}[1,\id_{Y_i},\phi]:X_i^\bullet\to Y_i^\bullet$ (which is indeed an isomorphism by Remark~\ref{rk:middle-term-identity}). Then $\bigotimes_{i\in\omega}\alpha_i\colon\bigotimes_{i\in\omega} X_i^\bullet\to\bigotimes_{i\in\omega} Y_i^\bullet$ is a well-defined isomorphism in $\Sigma(\mathcal C)$ as almost all $\alpha_i$ are the identity of the tensor unit $\epsilon$, so it defines an isomorphism $X^\bullet_\bullet\cong Y^\bullet_\bullet$ in $\Phi\Sigma(\mathcal C)$ as desired.
\end{proof}
\end{lemma}

\begin{rk}\label{rk:theta-notation}
If $A\subset\omega$ is any finite set, and $(X_a)_{a\in A}$ an $A$-indexed family of objects of a parsummable category $\mathcal C$, then we write $\langle A,X_\bullet\rangle\in\Theta(\mathcal C)$ for the object whose $i$-th entry is $(X_i)\in\Sigma(\mathcal C)$ if $i\in A$, and the tensor unit $\epsilon$ otherwise. Clearly, every object of $\Theta(\mathcal C)$ is of the form $\langle A,X_\bullet\rangle$ for some finite $A\subset\omega$ and some $A$-indexed family in $\mathcal C$; moreover, this representation is unique.

One immediately checks from the definitions that $\supp\langle A,X_\bullet\rangle=A$ and that for any finite $B\subset\omega$ with $A\cap B=\varnothing$ the sum $\langle A,X_\bullet\rangle + \langle B,Y_\bullet\rangle$ is given by $\langle A\cup B, X_\bullet+Y_\bullet\rangle$, where $X_\bullet+Y_\bullet$ is the $(A\cup B)$-indexed family with
\begin{equation*}
(X_\bullet+Y_\bullet)_i=\begin{cases}
X_i & \text{if }i\in A\\
Y_i & \text{if }i\in B.
\end{cases}
\end{equation*}
Finally, $\bigotimes_{i\in\omega}\langle A,X_\bullet\rangle_i=\kappa_A^*X_\bullet=(X_{\kappa_A(1)},\dots,X_{\kappa_A(|A|)})$ where $\kappa_A\colon\{1,\dots,|A|\}\to A$ is the unique order-preserving bijection.
\end{rk}

\begin{constr}
Fix an injection $\phi\colon\omega\times\omega\rightarrowtail\omega$.
We define for any parsummable category $\mathcal C$ a functor $S_{\mathcal C}\colon\Theta(\mathcal C)\to\mathcal C$ on objects via $S_{\mathcal C}\langle A,X_\bullet\rangle=(\phi|_A)_*(X_\bullet)$. If $\langle B,Y_\bullet\rangle$ is another object, then a morphism $\alpha\colon\langle A,X_\bullet\rangle\to\langle B,Y_\bullet\rangle$ in $\Theta(\mathcal C)\subset\Phi\Sigma(\mathcal C)$ is by definition the same as a morphism $\kappa_A^*X_\bullet\to\kappa_B^*Y_\bullet$ in $\Sigma(\mathcal C)$. By Lemma~\ref{lemma:free-to-choose-phi-psi} any such morphism admits a unique representative of the form $\big(\phi|_B\circ(\kappa_B\times\id),f,\phi|_A\circ(\kappa_A\times\id)\big)$ for some $f\colon\big(\phi|_A\circ(\kappa_A\times\id)\big)_*(\kappa_A^*X_\bullet)\to \big(\phi|_B\circ(\kappa_B\times\id)\big)_*(\kappa_B^*Y_\bullet)$ and we define $S_{\mathcal C}(\alpha)=f$.
\end{constr}

\begin{prop}
The above is a well-defined and sum-preserving functor.
\begin{proof}
Let us first show that $S_{\mathcal C}$ is well-defined on morphisms, i.e.~in the above situation $f$ is actually a morphism $S_{\mathcal C}\langle A,X_\bullet\rangle\to S_{\mathcal C}\langle B,Y_\bullet\rangle$. But indeed, by Lemma~\ref{lemma:generalized-action}-$(\ref{item:lga-perm})$, $\big(\phi|_A\circ(\kappa_A\times\id)\big)_*(\kappa_A^*X_\bullet)=(\phi|_A)_*(X_\bullet)=S_{\mathcal C}\langle A,X_\bullet\rangle$ and similarly for the target. Thus, $S_{\mathcal C}$ is well-defined, and it is then clearly a functor.

It remains to show that $S_{\mathcal C}$ preserves sums. By definition $S_{\mathcal C}\langle\varnothing,\varnothing\rangle=0$ (where $\varnothing$ denotes both the empty set as well as the empty family), i.e.~$S_{\mathcal C}$ preserves the additive unit. Moreover, Lemma~\ref{lemma:generalized-action}-$(\ref{item:lga-partition})$ shows that $S_{\mathcal C}$ sends disjointly supported objects to disjointly supported objects and that it preserves sums of objects.

It only remains to show that $S_{\mathcal C}$ preserves sums of morphisms, for which we let $\alpha\colon\langle A,X_\bullet\rangle\to\langle B,Y_\bullet\rangle$, $\beta\colon\langle A',X'_\bullet\rangle\to\langle B',Y'_\bullet\rangle$ be any morphisms such that $A\cap A'=\varnothing$ and $B\cap B'=\varnothing$.

By definition of $\Phi$, the sum $\alpha+\beta$ in $\Theta(\mathcal C)\subset\Phi\Sigma(\mathcal C)$ is the composition
\begin{equation*}
\bigotimes_{i\in A\cup A'} (X_\bullet+X'_\bullet)_i\to
\bigotimes_{i\in A} X_i \otimes \bigotimes_{i\in A'} X'_i\xrightarrow{\alpha\otimes\beta}
\bigotimes_{i\in B} Y_i \otimes \bigotimes_{i\in B'} Y'_i
\to
\bigotimes_{i\in B\cup B'} (Y_\bullet+Y'_\bullet)_i
\end{equation*}
where the outer two arrows are the coherence isomorphisms associated to the tautological bijections $A\amalg A'\to A\cup A'$, $B\cup B'\to B\amalg B'$; here we again make the disjoint unions into totally ordered sets by demanding that any element of the first summand be smaller than any element of the second summand.

If we write $\alpha=[\phi|_B\circ(\kappa_B\times\id),f,\phi|_A\circ(\kappa_A\times\id)]$ and $\beta=[\phi|_{B'}\circ(\kappa_{B'}\times\id),g,\phi|_{A'}\circ(\kappa_{A'}\times\id)]$, then
\begin{equation*}
\alpha\otimes\beta=\big[\phi|_{B\cup B'}\circ\big((\kappa_{B}+\kappa_{B'})\times\id\big), f+g, \phi|_{A\cup A'}\circ\big((\kappa_{A}+\kappa_{A'})\times\id\big)\big]
\end{equation*}
by definition of the tensor product in $\Sigma(\mathcal C)$, where $\kappa_{A}+\kappa_{A'}\colon \{1,\dots,|A|+|A'|\}\to A\cup A'$ is defined by $(\kappa_A+\kappa_{A'})(i)=\kappa_A(i)$ for $i\le|A|$ and $(\kappa_A+\kappa_{A'})(i)=\kappa_{A'}(i-|A|)$ otherwise, and $\kappa_B+\kappa_{B'}$ is defined analogously. After postcomposing with the tautological bijection $A\cup A'\cong A\amalg A'$, $\kappa_{A}+\kappa_{A'}$ becomes the unique order-preserving bijection $\{1,\dots,|A|+|A'|\}\to A\amalg A'$. We therefore conclude from the construction of coherence isomorphisms associated to bijections between finite totally ordered sets (see the discussion after Remark~\ref{rk:rk-before-coherence-isomorphisms}) that the above coherence isomorphism $\bigotimes_{i\in A\cup A'}(X_\bullet+X_\bullet')_i\to \bigotimes_{i\in A}X_i\otimes \bigotimes_{i \in A'} X_i'$ agrees with the coherence isomorphism associated to the unique permutation $\sigma\in\Sigma_{|A|+|A'|}$ such that $(\kappa_{A}+\kappa_{A'})\circ\sigma=\kappa_{A\cup A'}$. Lemma~\ref{lemma:sigma-coherence} applied to the injection $\phi|_{A\cup A'}\circ((\kappa_A+\kappa_{A'})\times\id)$ then shows that this coherence isomorphism is given by
\begin{align*}
&[\phi|_{A\cup A'}\circ((\kappa_A+\kappa_{A'})\times\id),\id,\phi|_{A\cup A'}\circ((\kappa_A+\kappa_{A'})\times\id)\circ(\sigma\times\id)]\\
&\qquad=
[\phi|_{A\cup A'}\circ((\kappa_A+\kappa_{A'})\times\id),\id,\phi|_{A\cup A'}\circ(\kappa_{A\cup A'}\times\id)].
\end{align*}
Similarly, the remaining coherence isomorphism is given by
\begin{equation*}
[\phi|_{B\cup B'}\circ(\kappa_{B\cup B'}\times\id),\id,\phi|_{B\cup B'}\circ((\kappa_{B}+\kappa_{B'})\times\id)].
\end{equation*}
Plugging this in, we get $\alpha+\beta=[\phi|_{B\cup B'}\circ(\kappa_{B\cup B'}\times\id),f+g,\phi|_{A\cup A'}\circ(\kappa_{A\cup A'}\times\id)]$, hence $S_{\mathcal C}(\alpha+\beta)=f+g=S_{\mathcal C}(\alpha)+S_{\mathcal C}(\beta)$ as desired.
\end{proof}
\end{prop}

\begin{lemma}
For any parsummable category $\mathcal C$, the functor $S_{\mathcal C}$ is an equivalence of categories. For varying $\mathcal C$, these assemble into a natural transformation ${\forget}\circ\Theta\Rightarrow\forget$ of functors $\cat{ParSumCat}\to\cat{Cat}$.
\begin{proof}
Let us first show that $S_{\mathcal C}$ is fully faithful. If $\langle A,X_\bullet\rangle,\langle B,Y_\bullet\rangle\in\Theta(\mathcal C)$ are any two objects, then $f\mapsto [\phi|_B\circ(\kappa_B\times\id),f,\phi|_A\circ(\kappa_A\times\id)]$ defines a bijection $\Hom_{\mathcal C}((\phi|_A\circ(\kappa_A\times\id))_*(\kappa_A^*X_\bullet),(\phi|_B\circ(\kappa_B\times\id))_\bullet(\kappa_B^*Y_\bullet))\to\Hom_{\Sigma(\mathcal C)}(\kappa_A^*X_\bullet,\kappa_B^*Y_\bullet)$
by Lemma~\ref{lemma:free-to-choose-phi-psi}. However, the source of this equals $\Hom_{\mathcal C}((\phi|_A)_*(X_\bullet),(\phi|_B)_*(Y_\bullet))$ by Lemma~\ref{lemma:generalized-action}-$(\ref{item:lga-perm})$, while the target equals $\Hom_{\Theta(\mathcal C)}(\langle A,X_\bullet\rangle,\langle B,Y_\bullet\rangle)$ by definition of $\Phi$ (also see Remark~\ref{rk:theta-notation}), and this assignment is then by construction right inverse to
\begin{equation*}
S_{\mathcal C}\colon
\Hom_{\Theta(\mathcal C)}(\langle A,X_\bullet\rangle,\langle B,Y_\bullet\rangle)
\to\Hom_{\mathcal C}(S_{\mathcal C}\langle A,X_\bullet\rangle,S_{\mathcal C}\langle B,Y_\bullet\rangle).
\end{equation*}
We conclude that also the latter is bijective, i.e.~$S_{\mathcal C}$ is fully faithful.

To see that $S_{\mathcal C}$ is also essentially surjective, we let $X\in\mathcal C$ be arbitrary. Then $S_{\mathcal C}\langle\{1\},X\rangle=\phi(1,_\blank)_*X$ (where we confuse $X$ with the $\{1\}$-indexed family with unique value $X$), and the structure isomorphism $[\phi(1,\blank),\id]_X$ of the $E\mathcal M$-action on $\mathcal C$ shows that this is isomorphic to the original object $X$.

Finally, we have to show that $S_{\mathcal D}\circ\Theta(F)=F\circ S_{\mathcal C}$ for any morphism $F\colon\mathcal C\to\mathcal D$ of parsummable categories. Unravelling definitions, we see that $\Theta(F)$ is given on objects by sending $\langle A,X_\bullet\rangle$ to $\langle A, F(X_\bullet)\rangle$ with $F(X_\bullet)_a=F(X_a)$ for all $a\in A$, so the desired equality holds on objects by Remark~\ref{rk:universally-natural-trafo}.

On morphisms, $\Theta(F)$ is given by sending a morphism $\langle A,X_\bullet\rangle\to\langle B,Y_\bullet\rangle$ in $\Theta(\mathcal C)\subset\Phi\Sigma(\mathcal C)$ given by a morphism $[\rho,f,\theta]$ in $\Sigma(\mathcal C)$ (for suitable injections $\theta,\rho$) to $[\rho, Ff,\theta]$. Specializing to $\theta=\phi|_A\circ(\kappa_A\times\id)$ and $\rho=\phi|_B\circ(\kappa_B\times\id)$ then proves that $S_{\mathcal D}\circ\Theta(F)=F\circ S_{\mathcal C}$ also holds on morphisms.
\end{proof}
\end{lemma}

Now we can finally prove our main result:

\begin{proof}[Proof of Theorem~\ref{thm:main-thm}]
By the above we have for any parsummable category $\mathcal C$ a natural sum-preserving equivalence of categories $S_{\mathcal C}\colon\Theta(\mathcal C)\to\mathcal C$, which gives rise to a functor $\cat{ParSumCat}\to\cat{PsArParSumCat}$. By Corollary~\ref{cor:cylinder-factorization} together with Corollary~\ref{cor:cylinder-natural}, we therefore get a natural zig-zag
\begin{equation*}
\Theta(\mathcal C)\xrightarrow{I^{(1)}} \Cyl(S_{\mathcal C}) \xleftarrow{I^{(2)}} \mathcal C
\end{equation*}
of morphisms of parsummable categories that are at the same time underlying equivalences of categories. Together with Lemma~\ref{lemma:Theta-vs-Phi-Sigma} we therefore obtain a natural zig-zag
\begin{equation}\label{eq:Phi-Sigma-vs-id}
\Phi\Sigma\Leftarrow\Theta\Rightarrow\Cyl(S_\bullet)\Leftarrow\id
\end{equation}
of levelwise underlying equivalences between endofunctors of $\cat{ParSumCat}$. We conclude by $2$-out-of-$3$ that $\Phi\Sigma$ is homotopical, hence so is $\Sigma$ because $\Phi$ reflects underlying equivalences by Lemma~\ref{lemma:Phi-reflect-preserve}.

On the other hand, the same lemma shows that $\Phi$ is homotopical, and with this established $(\ref{eq:Phi-Sigma-vs-id})$ precisely shows that $\Sigma$ is right homotopy inverse to $\Phi$, while Proposition~\ref{prop:T-equivalence} shows that it is also left homotopy inverse. This completes the proof of the theorem.
\end{proof}

\section{Outlook}\label{sec:outlook}
There is another notion of weak equivalence of parsummable categories that is interesting from the point of view of global homotopy theory and in particular in the context of global algebraic $K$-theory, and which we want to briefly discuss in this final section. For this we will need the following terminology:

\begin{defi}
A finite subgroup $G\subset\mathcal M$ is called \emph{universal} if the restriction of the tautological $\mathcal M$-action on $\omega$ to $G$ makes $\omega$ into a \emph{complete $G$-set universe}, i.e.~a (countable) $G$-set into which any finite $G$-set embeds equivariantly.
\end{defi}

Schwede \cite[Definition~2.16]{schwede-k-theory} uses the term `universal $G$-set' for a countable $G$-set $U$ such that each subgroup $H\subset G$ occurs as stabilizer of infinitely many elements of $U$; the equivalence to the above definition of a complete $G$-set universe is easy to check and also appears without proof as \cite[Proposition~2.17-(i)]{schwede-k-theory}.

It is not hard to check that any finite group $G$ is isomorphic to a universal subgroup of $\mathcal M$ and that any two such embeddings $G\hookrightarrow\mathcal M$ differ only by conjugation with an invertible element of $\mathcal M$.

\begin{defi}
A morphism $F\colon\mathcal C\to\mathcal D$ of parsummable categories is a \emph{global weak equivalence} if the induced functor $F^G\colon\mathcal C^G\to\mathcal D^G$ is a weak homotopy equivalence (i.e.~induces a homotopy equivalence on classifying spaces) for all universal $G\subset\mathcal M$.
\end{defi}

Schwede \cite[Definition~2.26]{schwede-k-theory} calls these `global equivalences,' but we prefer the above name as we want to emphasize that in general the global weak equivalences and the underlying equivalences studied in this article are incomparable.

In practice, however, one is usually interested in so-called saturated parsummable categories \cite[Definition~7.3]{schwede-k-theory}, for which the relationship between the two notions of weak equivalence is simpler:

\begin{defi}
A parsummable category $\mathcal C$ is called \emph{saturated} if the canonical inclusion $\mathcal C^G\hookrightarrow\mathcal C^{hG}\mathrel{:=}\Fun(EG,\mathcal C)^G$ of the honest fixed points into the homotopy fixed points is an equivalence of categories for all universal $G\subset\mathcal M$.
\end{defi}

As underlying equivalences induce equivalences on homotopy fixed points, any underlying equivalence of saturated parsummable categories is in particular a global weak equivalence, also see~\cite[Proposition~7.11]{schwede-k-theory}.

Our saturation construction appearing as \cite[Construction~7.23]{schwede-k-theory} (which we will not discuss here) provides a homotopy inverse with respect to the underlying equivalences to the inclusion of the full subcategory $\cat{ParSumCat}^{\text{sat}}$ spanned by the saturated parsummable categories, see \cite[Theorem~7.25]{schwede-k-theory}.

If $\mathscr C$ is a small permutative category, then $\Phi\mathscr C$ is usually not saturated, and to construct the global algebraic $K$-theory of $\mathscr C$ one first applies the saturation construction. Our main theorem above then implies that the resulting composition
\begin{equation}\label{eq:perm-sat}
\cat{PermCat}\xrightarrow\Phi\cat{ParSumCat}\xrightarrow{(\blank)^{\text{sat}}}\cat{ParSumCat}^{\text{sat}}
\end{equation}
is a homotopy equivalence with respect to the underlying equivalences. It follows formally that when we equip the right hand side with the global weak equivalences instead, then there is still a notion of weak equivalence on $\cat{PermCat}$ such that $(\ref{eq:perm-sat})$ is a homotopy equivalence---namely, those maps that are inverted by the above composition. It is not hard to check that these are precisely those strict symmetric monoidal functors whose underlying functors are global equivalences of categories in the sense of \cite[Definition~3.2]{schwede-cat}, i.e.~functors that induce weak homotopy equivalences on the categories of $G$-objects for all finite $G$.

This naturally leads to the question whether also the composition
\begin{equation}\label{eq:perm-sat-unsat}
  \cat{PermCat}\xrightarrow\Phi\cat{ParSumCat}\xrightarrow{(\blank)^{\text{sat}}}\cat{ParSumCat}^{\text{sat}}\hookrightarrow\cat{ParSumCat}
\end{equation}
is a homotopy equivalence with respect to the global weak equivalences, and building on Theorem~\ref{thm:main-thm} above we prove as one of the main results of \cite{sym-mon-global} that this is indeed the case.

Finally, $\Phi$ also sends underlying equivalences of small permutative categories to global weak equivalences by \cite[Proposition~11.7]{schwede-k-theory}, so one could once again ask if there is a coarser notion of weak equivalence on $\cat{PermCat}$ such that $\Phi$ becomes an equivalence of homotopy theories with respect to the global weak equivalences on $\cat{ParSumCat}$. This is not the case, however, as we will argue now:

Applying \cite[Construction~10.1]{schwede-k-theory} to the ring $\mathbb C$ of complex numbers yields a parsummable category $\mathcal C$ whose global algebraic $K$-theory is (by definition) the global algebraic $K$-theory of $\mathbb C$. Identifying $\mathbb Z/3$ with a universal subgroup of $\mathcal M$ isomorphic to it, \cite[proof of Theorem~10.3-(iii)]{schwede-k-theory} shows that $\pi_0(\mathcal C)\cong\mathbb N$ and $\pi_0(\mathcal C^{\mathbb Z/3})\cong\mathbb N^3$ as monoids; here we write $\pi_0$ for the path components of the classifying space (i.e.~two objects define the same element in $\pi_0$ iff they can be connected by a zig-zag of morphisms), and the left hand sides carry the natural commutative monoid structures induced from the sum operation (using \cite[Theorem~2.32]{schwede-k-theory} to choose disjointly supported representatives). From the point of view of $K$-theory, the standard generators correspond to the vector space $\mathbb C$ and the three isomorphism classes of irreducible complex representations of $\mathbb Z/3$, respectively.

We now claim:

\begin{prop}\label{prop:fixed-point-obstruction}
There exists no (small) permutative category $\mathscr D$ with $\pi_0(\Phi\mathscr D)\cong\mathbb N$ and $\pi_0\big((\Phi\mathscr D)^{\mathbb Z/3})\cong\mathbb N^3$.
\end{prop}

Before we prove the proposition, let us observe that it immediately implies that $\Phi$ is not essentially surjective after passing to the (say, $1$-categorical) localization of $\cat{ParSumCat}$ at the global weak equivalences as both $\pi_0$ and $\pi_0((\blank)^{\mathbb Z/3})$ invert global weak equivalences, so that they pass to functors $\Ho(\cat{ParSumCat})\to\cat{CMon}$. In particular, there is no notion of weak equivalence of small permutative categories so that $\Phi$ induces an equivalence of the corresponding (ordinary or $\infty$-categorical) localizations.

The proof of the proposition will basically be an elaboration on \cite[Proposition~11.9]{schwede-k-theory}. For this let us fix a small permutative category $\mathscr D$.

\begin{constr}
For any $X\in\mathscr D$ and any finite subset $S\subset\omega$ we write $\Pi(X,S)$ for the object of $\Phi\mathscr D$ with
\begin{equation*}
\Pi(X,S)_i=\begin{cases}
  X & \text{if $i\in S$}\\
  \textbf{1} & \text{otherwise}.
\end{cases}
\end{equation*}
We will now define for any universal $G\subset\mathcal M$ a homomorphism
\begin{equation*}
a_G\colon\prod_{\text{isom. classes of transitive $G$-sets}}\pi_0(\mathscr D)\to \pi_0((\Phi\mathscr D)^G)
\end{equation*}
of commutative monoids; here the monoid structure on $\pi_0(\mathscr D)$ is the one given by the tensor product and the monoid structure on the right is again induced by the sum.

As the above finite product is also a coproduct in the category of commutative monoids, it suffices to define $a_G$ on each factor separately. This is done as follows: given an isomorphism class of transitive $G$-sets, we pick a subset $S\subset\omega$ representing it; this is possible because $\omega$ is a universal $G$-set by assumption on $G$. We then send an element of $\pi_0(\mathscr D)$ represented by an object $X\in\mathscr D$ to the class $[\Pi(X,S)]$.
\end{constr}

\begin{lemma}
The above is a well-defined monoid homomorphism.
\begin{proof}
It suffices to show that the restriction of $a_G$ to each factor is well-defined and a homomorphism.

For this we fix $X\in\mathscr D$ and an isomorphism class of transitive $G$-sets. It is clear that $\Pi(X,S)$ is $G$-fixed for \emph{any} finite $G$-subset $S\subset\omega$, hence in particular for $S$ in the given isomorphism class. If $f\colon X\to Y$ is any morphism in $\mathscr D$, then the morphism $\bigotimes_{s\in S}f$ of $\mathscr D$ defines a map $\Pi(X,S)\to\Pi(Y,S)$ in $\Phi\mathscr D$, and we claim that this is $G$-fixed. Indeed, by construction together with Lemma~\ref{lemma:shuffle-coherence}, $g\in G$ acts by conjugation with the coherence isomorphism associated to the self-bijection $g.\blank\colon S\to S$; while these may be non-trivial, this means purely by naturality that $g$ sends a morphism $\Pi(X,S)\to\Pi(Y,S)$ of the form $\bigotimes_{s\in S}f_s$ to $\bigotimes_{s\in S}f_{g^{-1}s}$, so that it in particular preserves the above map, i.e.~$a_G$ is independent of the choice of representative in $\mathscr D$.

To finish the proof that $a_G$ is well-defined it therefore suffices that $\Pi(X,S)\cong\Pi(X,T)$ in $(\Phi\mathscr D)^G$ for any finite $G$-subsets $S,T\subset\omega$ that are $G$-equivariantly isomorphic (in fact, already a zig-zag of maps between $\Pi(X,S)$ and $\Pi(X,T)$ would have been enough). Indeed, if $\alpha\colon S\to T$ is any $G$-equivariant bijection, then the coherence isomorphism associated to $\alpha$ yields an isomorphism $\Pi(X,S)\cong\Pi(X,T)$ in $\Phi\mathscr D$, which we claim to be $G$-fixed. But as $\alpha$ commutes with $g.\blank$ for all $g\in G$ by definition, this is simply an instance of the transitivity of coherence isomorphisms (i.e.~if $\beta,\gamma$ are composable bijections of finite totally ordered sets, then the coherence isomorphism associated to $\beta\circ\gamma$ is the composite of the coherence isomorphisms associated to $\beta$ and $\gamma$, respectively).

Finally, for the proof that $a_G$ is a homomorphism when restricted to our fixed factor, we let $X^{(1)},X^{(2)}\in\mathscr D$ be any objects, and we let $S^{(1)},S^{(2)}\subset\omega$ be finite $G$-subsets in our fixed isomorphism class; we want to show that $[\Pi(X^{(1)},S^{(1)})]+[\Pi(X^{(2)},S^{(2)})]=[\Pi(X^{(1)}\otimes X^{(2)}),S^{(1)}]$ in $\pi_0((\Phi\mathscr D)^G)$. As one sees for example from the alternative definition of complete $G$-set universes in terms of stabilizers, taking out a finite $G$-set from $\omega$ still leaves a complete $G$-set universe, and applying this to the $G$-set generated by all $s\le\max S^{(1)}$, we may assume without loss of generality that $s^{(1)}<s^{(2)}$ for all $s^{(1)}\in S^{(1)},s^{(2)}\in S^{(2)}$, in which case an isomorphism $\Pi(X^{(1)},S^{(1)})+\Pi(X^{(2)},S^{(2)})\cong\Pi(X^{(1)}\otimes X^{(2)},S^{(1)})$ in $(\Phi\mathscr D)^G$ amounts to a $G$-equivariant isomorphism
\begin{equation*}
\bigotimes_{s\in S^{(1)}} X^{(1)}\otimes\bigotimes_{s\in S^{(2)}}X^{(2)} \cong \bigotimes_{s\in S^{(1)}}(X^{(1)}\otimes X^{(2)})
\end{equation*}
in $\mathscr D$ with respect to the $G$-action coming from the coherence isomorphisms as before.

For this we fix a $G$-equivariant bijection $\alpha\colon S^{(2)}\to S^{(1)}$ again. Then the above calculation shows that the coherence isomorphism associated to $\alpha$ yields a $G$-equivariant isomorphism $\bigotimes_{s\in S^{(2)}}X^{(2)}\cong\bigotimes_{s\in S^{(1)}}X^{(2)}$, so it suffices to construct a $G$-equivariant isomorphism $\bigotimes_{s\in S^{(1)}} X^{(1)}\otimes\bigotimes_{s\in S^{(1)}}X^{(2)}\cong\bigotimes_{s\in S^{(1)}}(X^{(1)}\otimes X^{(2)})$.

But if we equip the sets $\bm2\times S^{(1)}$ and $S^{(1)}\times\bm2$ with the lexicographical orders, then this is the same as a $G$-equivariant isomorphism
\begin{equation*}
\bigotimes_{(i,s)\in\bm2\times S^{(1)}} X^{(i)}\cong\bigotimes_{(s,i)\in S^{(1)}\times\bm2} X^{(i)}
\end{equation*}
where $G$ acts on both sides via the coherence isomorphisms associated to its action on $S^{(1)}$. The canonical isomorphism $\bm2\times S^{(1)}\cong S^{(1)}\times\bm2$ switching the factors is $G$-equivariant, so the same argument as before shows that the associated coherence isomorphism has the desired properties.
\end{proof}
\end{lemma}

\begin{proof}[Proof of Proposition~\ref{prop:fixed-point-obstruction}]
We will show that $a_G$ is surjective for all universal $G\subset\mathcal M$ and that $a_1$ is bijective. This leads to the desired contradiction as a small permutative category $\mathscr D$ with $\pi_0(\Phi\mathscr D)\cong\mathbb N$ and $\pi_0(\Phi(\mathscr D)^{\mathbb Z/3})\cong\mathbb N^3$ would then yield a surjective homomorphism
\begin{equation*}
\begin{tikzcd}
\mathbb N^2\cong\pi_0(\Phi\mathscr D)^{\times2}\arrow[r, "(a_1^{-1})^{\times2}", "\cong"'] & \pi_0(\mathscr D)^{\times2}\arrow[r, "a_{\mathbb Z/3}", two heads]&\pi_0\big((\Phi\mathscr D)^{\mathbb Z/3}\big)\cong\mathbb N^3,
\end{tikzcd}
\end{equation*}
which is impossible; here we used that there are precisely two transitive $\mathbb Z/3$-sets up to isomorphism.

To prove that $a_G$ is surjective, we observe that the support of any $G$-fixed object $X_\bullet\in\Phi\mathscr D$ is a (finite) $G$-subset of $\omega$, and that $X_\bullet$ is constant on its $G$-orbits. Thus, if we decompose $\supp X_\bullet=S_1\sqcup\cdots\sqcup S_k$ into disjoint orbits and pick orbit representatives $s_i\in S_i$, $i=1,\dots,k$, then $X_\bullet$ is actually equal in $(\Phi\mathscr D)^G$ to the sum $\Pi(X_{s_1},S_1)+\cdots+\Pi(X_{s_k},S_k)$, and in particular $[X_\bullet]=[\Pi(X_{s_1},S_1)]+\cdots+[\Pi(X_{s_k},S_k)]$ in $\pi_0((\Phi\mathscr D)^G)$. By construction, each $[\Pi(X_{s_i},S_i)]$ is contained in the image of $a_G$, hence so is $[X_\bullet]$ as $a_G$ is a monoid homomorphism.

Finally, to show that $a_1$ is also injective, it suffices to observe that it admits a left inverse induced by $X_\bullet\mapsto\bigotimes_{i\in\omega}X_i$.
\end{proof}

\begin{rk}
We remark that after passing to group completions the same argument shows that there is no (small) permutative category $\mathscr D$ such that the global algebraic $K$-theory of $\Phi\mathscr D$ is equivalent to the global algebraic $K$-theory $\textbf{K}_{\textup{gl}}(\mathbb C)$ of the complex numbers. On the other hand, there is indeed a small permutative category $\mathscr D$ such that the \emph{saturation} $(\Phi\mathscr D)^{\text{sat}}$ is globally weakly equivalent to the parsummable category $\mathcal C$ considered above, and in particular the global algebraic $K$-theory of $(\Phi\mathscr D)^{\text{sat}}$ (which by definition is the global algebraic $K$-theory of the permutative category $\mathscr D$) is equivalent to the global algebraic $K$-theory of $\mathbb C$; thus, the above example does not contradict the global version of Thomason's theorem announced in the introduction. In fact, it is not hard to show that $\mathscr D$ can be taken to be any small permutative category equivalent to the symmetric monoidal category of finite dimensional $\mathbb C$-vector spaces and $\mathbb C$-linear isomorphisms under $\oplus$. We remark that the mere existence of such $\mathscr D$ also follows from $(\ref{eq:perm-sat-unsat})$ being a homotopy equivalence (which we did not prove here) or, once one observes that $\mathcal C$ is saturated \cite[Theorem 10.3-(i)]{schwede-k-theory}, from $(\ref{eq:perm-sat})$ being a homotopy equivalence.
\end{rk}

\frenchspacing
\bibliographystyle{alpha}
\bibliography{literature.bib}
\end{document}